%% file: Tauberian_approach_scaling_lim.tex
\documentclass[11pt,reqno,final]{amsart}

\usepackage[utf8]{inputenc}
\usepackage[english]{babel}
\usepackage{mathtools, amssymb, amsthm, enumitem}
	\theoremstyle{plain}
		\newtheorem{thm}{Theorem}[section]
		\newtheorem{constr}[thm]{Construction}
		\newtheorem{prop}[thm]{Proposition}
		\newtheorem{lemma}[thm]{Lemma}
		\newtheorem{cor}[thm]{Corollary}
	\theoremstyle{definition}
		\newtheorem{defin}[thm]{Definition}
		
	\theoremstyle{remark}
		\newtheorem{rem}[thm]{Remark}
	\numberwithin{equation}{section}
\usepackage[final]{microtype}
\usepackage{upgreek, bbm, dsfont, csquotes}
\usepackage{graphicx, tikz, xcolor}
	\graphicspath{{figures/}}
\usepackage[foot]{amsaddr}
\usepackage{hyperref}
	\definecolor{MyBlue}{HTML}{0000FF}
	\hypersetup{colorlinks=true, citecolor = MyBlue, linkcolor=MyBlue, urlcolor = MyBlue}
\usepackage[hyperref = true, style=numeric,	sorting=nyt]{biblatex}	
	\AtEveryBibitem{\clearfield{url}}
\input{macros}

\title[A Tauberian approach to scaling limits of discrete structures]{A Tauberian approach to metric scaling limits of random discrete structures,\\ {\footnotesize \em with an application to random planar maps}}

\author{William FLEURAT}
\email{william.fleurat@universite-paris-saclay.fr}
\address{Université Paris-Saclay, France}

\begin{document}

% -----------------------------------------------------------------
%							ABSTRACT
% -----------------------------------------------------------------

\begin{abstract}
	We prove sandwich theorems and a Tauberian theorem in the space of compact metric measure spaces, endowed with the Gromov--Hausdorff--Prokhorov (GHP) topology.
	These results hold with respect to a close relative of Gromov's Lipschitz order.
	As a proof-of-concept of a general method to prove \textit{metric} scaling limits of random discrete structures, we give an application to the theory of random planar maps: the Brownian sphere is the scaling limit in the GHP topology of irreducible quandrangulations.
	Our main inputs are (i) the convergence of general quadrangulations to the Brownian sphere (Le~Gall, 2013; Miermont, 2013); and (ii) couplings where irreducible quadrangulations of the hexagon are ``grown'' by face-openings (Addario-Berry, 2014).
\end{abstract}

% -----------------------------------------------------------------
%							TITLE
% -----------------------------------------------------------------

\vspace*{-.3em}
\maketitle

% -----------------------------------------------------------------
%							SECTION 1
% -----------------------------------------------------------------

\section{Introduction}
\label{sec:introduction}

This work aims to present a novel approach to ``descend'' metric properties of large random discrete structures from a \textit{host} model to a related \textit{component} model, in a {condensation} regime.
For concreteness, we let $H_n$ (resp.~$C_n$) denote a uniformly random size-$n$ object from the \textit{host} model (resp.~from the \textit{component} model).
The \textit{condensation} phenomenon---which is ubiquitous in probabilistic combinatorics, see the related surveys~\cite{Janson12,Stufler20} and references therein---may be described as follows: with high probability as $n\to\infty$, one can find in $H_n$ an object $\mathfrak c(H_n)$ from the \textit{component} model, which is a random-size version of the $(C_n)_{n\geq1}$, with asymptotically linear size:
\begin{align*}
	\mathfrak c(H_n)\overset{(d)}{=} C_{N(n)},\qquad\text{with}\qquad \frac{N(n)}{n}\xrightarrow[n\to\infty]{}\theta
\end{align*}
in probability, for some $\theta>0$.

It is often possible to transfer asymptotics from $H_n$ to $\mathfrak c(H_n)$. Since the latter has the same distribution as $C_{N(n)}$, one may therefore hope to extract the asymptotic behavior of $C_n$ by comparison with $C_{N({m})}$, where $m$ is chosen so that $N(m)\approx n$, say $m=\lfloor n/\theta\rfloor$.
While transferring asymptotics from the non-randomized sequence $(C_n)_{n\geq1}$ to the randomized one $(C_{N(\lfloor n/\theta\rfloor)})_{n\geq1}$ is straightforward, the converse direction that we consider here is a problem of a \textit{Tauberian} nature, which requires ruling out bad behavior by means of an extra \textit{Tauberian assumption}.

To that end, we establish a \textit{Tauberian theorem} in the Gromov--Hausdorff--Prokhorov space of (isometry classes of) compact metric measure spaces, where the Tauberian assumption involves a variant of Gromov's Lipschitz order~\cite{Gromov99,GrieshammerRippl16}. Along the way, we prove that the corresponding order satisfies a sandwich theorem of independent interest.

As a proof-of-concept, we prove a new scaling limit result in the theory of random planar maps: random \textit{irreducible quadrangulations} properly normalized converge to the Brownian sphere, a result we obtain by reduction to the case of \textit{general quadrangulations} covered by Le~Gall~\cite{LeGall13} and Miermont~\cite{Miermont13}.
The proof decomposes into a few well-identified and largely independent steps, for which we provide arguments that we believe to be fairly robust. We therefore expect it to find applications beyond the specific model studied in this paper.

\subsection{Sandwich and Tauberian theorems in the GHP space}

We denote by $\bbM$ the set of \textit{compact metric measure spaces}---namely triples $\X=(X,d,\mu)$ where $(X,d)$ is a compact metric space and $\mu$ is a \textit{finite} measure on its Borel $\sigma$-algebra---and we denote by $\setGHP$ the quotient of $\bbM$ with respect to the relation of measure-preserving isometry: $(X,d,\mu)$ and $(X',d',\mu')$ are measure-preserving isometric if there is a bijective mapping $\phi \colon X\rightarrow X'$ such that $\phi_*\mu=\mu'$ and $d'(\phi(x),\phi(y))=d(x,y)$ for all $x,y\in X$.
We denote by $[X,d,\mu]$ the measure-preserving isometry class of $(X,d,\mu)\in\bbM$.
The quotient space $\setGHP$ may be equipped with the \textit{measured Gromov--Hausdorff} topology \cite{Fukaya87}, or \textit{Gromov--Hausdorff--Prokhorov} (GHP) topology, which is induced by the GHP metric $\dGHP$, see \cite{EvansWinter06,Miermont09,AbrahamDelmasHoscheit13}.
Convergence in the GHP topology will be denoted by the symbol ${\scriptstyle\smash{\toGHP}}$.

\subsubsection{A partial order on the GHP space}

We introduce a partial order on the space $\setGHP$ which is a variant of Gromov's Lipschitz order \cite{Gromov99,Shioya16}, see also Grieshammer and Rippl's unpublished work \cite{GrieshammerRippl16} for a closely related variant on the Gromov--Prokhorov (GP) space.

\begin{defin}\label{def:order-GHP}
	Given $\X=[X,d,\mu]$ and $\X'=[X',d',\mu']$ in $\setGHP$, we write $\X\orderGHP \X'$, if there exists a \textit{surjective}
mapping $\phi\colon X'\rightarrow X$ which is \textit{1-Lipschitz}, that is:
\begin{align}
	d\bigl(\phi(x'),\phi(y')\bigr)&\leq d'(x',y'),\label{eq:def-GHP-order-1} \qquad x',y'\in X';
\end{align}
and \textit{measure-contracting}, that is:
\begin{align}
\mu(A) &\leq \mu'\bigl(\phi^{-1}(A)\bigr),\qquad A\in\Bcal,\label{eq:def-GHP-order-2}
\end{align}
where $\Bcal$ is the Borel $\sigma$-algebra of $(X,d)$.
\end{defin}

We note that Definition~\ref{def:order-GHP} does not depend on the choice of representatives modulo measure-preserving isometry, as it should.
Let us give a concrete example.
	We denote by $V(G)$ the vertex-set of a graph $G$, which we equip with the graph distance $d_G$ and counting measure $\mu_G$.
	Given two graphs $G$ and $G'$, if there exists a surjective graph homomorphism $G'\to G$, then we have $[V(G),d_G,\mu_G]\orderGHP [V(G'),d_{G'},\mu_{G'}]$, see Lemma~\ref{lem:graph-hom-and-GHP-order}.

\subsubsection{A sandwich theorem}

The binary relation $\orderGHP$ actually defines a partial order on $\setGHP$ which interacts nicely with the GHP topology, as evidenced by the following sandwich theorem.

\begin{thm}[\textsc{GHP Sandwich theorem}]\label{thm:sandwich-thm-GHP}
	The binary relation $\orderGHP$ is a partial order on $\setGHP$, which obeys the following \emph{sandwich theorem}:
	\begin{align*}
	\begin{cases}
	\forall n,\quad\X^-_n\orderGHP\X_n\orderGHP\X^+_n,\\
	\X^-_n\toGHP\X_\infty,\\ \X^+_n\toGHP\X_\infty,
	\end{cases}
	\qquad \implies\qquad
	\Bigl[\X_n\toGHP\X_\infty\Bigr],
	\end{align*}
	for all sequences $(\X^-_n)_{n\geq1}$, $(\X_n)_{n\geq1}$, and $(\X^+_n)_{n\geq1}$ of elements of $\setGHP$, and all $\X_\infty\in\setGHP$.
\end{thm}

We note that our proof of this sandwich theorem relies on proving two natural properties of the partial order $\orderGHP$, namely that it is \textit{closed} and \textit{down-compact}, see Definition~\ref{def:closed-and-down-compact} and Lemma~\ref{lem:general-sandwich-theorem}. It is to be expected that natural variants of the order $\orderGHP$ for other ``Gromov--Hausdorff-type'' topologies also satisfy those properties and thus a sandwich theorem. 
In fact, Grieshammer and Rippl's partial order on the Gromov--Prokhorov space satisfies these properties---see Prop.~3.2 and Prop.~4.1 in \cite{GrieshammerRippl16}---and thus also a sandwich theorem. Also, our arguments can be readily adapted to the Gromov--Hausdorff topology by just dropping all mentions of the measures. 

\subsubsection{A second sandwich theorem}

We can go further and ask for a probabilistic version of the above sandwich theorem.
In the sequel, a \textit{random compact metric measure space} is a random variable $\Xf$ on some probability space, taking values in $\setGHP$ endowed with its Borel $\sigma$-algebra. Its distribution is therefore an element of the space $\mes_1(\setGHP)$ of Borel probability measures on $\setGHP$.
Naturally, $\mes_1(\setGHP)$ is endowed with the topology of weak convergence induced by the GHP topology. Such convergence will be represented by the symbol ${\scriptstyle\smash{\ToGHP}}$.
Stochastic ordering on $\mes_1(\setGHP)$ with respect to $\orderGHP$ will be denoted by $\storder$, that is $\P(\diff\X)\storder\P'(\diff\X)$ if and only if the inequality $\P(A)\leq \P'(A)$ holds for every $\orderGHP$-increasing event $A$---where the latter means that $\X'\in A$ whenever $\X\in A$ and $\X\storder\X'$.
We extend Theorem~\ref{thm:sandwich-thm-GHP} as follows.

\begin{thm}[\textsc{Probabilistic GHP Sandwich theorem}]\label{thm:sandwich-thm-GHP-random}
	The binary relation $\storder$ is a partial order on $\mes_1(\setGHP)$, which obeys the following \emph{sandwich theorem}:
	\begin{align*}
	\begin{cases}
	\forall n,\: \P^-_n(\diff\X)\storder\P_n(\diff\X)\storder\P^+_n(\diff\X),\\
	\P^-_n(\diff\X)\ToGHP\P_\infty(\diff\X),\\ \P^+_n(\diff\X)\ToGHP\P_\infty(\diff\X),
	\end{cases}
	\mkern-25mu \implies\quad
	\Bigl[\P_n(\diff\X)\ToGHP\P_\infty(\diff\X)\Bigr],
	\end{align*}
	for all sequences $(\P^-_n(\diff\X))_{n\geq1}$, $(\P_n(\diff\X))_{n\geq1}$, and $(\P^+_n(\diff\X))_{n\geq1}$ of elements of $\mes_1(\setGHP)$, and all $\P_\infty(\diff\X)\in\mes_1(\setGHP)$.
\end{thm}

We actually prove in Lemma~\ref{lem:lifting-closed-and-down-compact} that the two sufficient conditions we singled out for a sandwich theorem to hold (the order being \textit{closed} and \textit{down-compact}) automatically lift to the corresponding stochastic order, provided the underlying space is Polish.

\subsubsection{A Tauberian theorem}

The prototypical use case we have in mind for Theorem~\ref{thm:sandwich-thm-GHP-random} is ``de-randomizing'' or ``conditioning'' scaling limit results.
Suppose we are trying to prove that $\Xf_n$ properly rescaled converges to some $\Xf_\infty$ in distribution with respect to the GHP topology, but that we only have access to such a convergence for the sequence $(\Xf_{N_n},n\geq 1)$ where $N_n\approx n$ is a ``randomized time'' independent from the original sequence $(\Xf_n,n\geq1)$.
This randomization has an averaging effect on the distribution of the random variables at play, which may well hide oscillatory behavior of the original sequence.

This is strongly reminiscent of \textit{Tauberian problems} in real or complex analysis: given a transform that regularizes a sequence, \textit{Tauberian theorems} provide conditions under which the asymptotic behavior of the transform implies related asymptotics for the original sequence. These \textit{Tauberian} conditions ensure that the original sequence does not behave too wildly.
In our setting, the Tauberian assumption will be phrased in terms of stochastic increase with respect to the order $\orderGHP$.

For $a,b>0$ and $[X,d,\mu]\in\setGHP$, we define the rescaled space $(a,b)\cdot[X,d,\mu]=[X,a\cdot d,b\cdot\mu]$.
We say that a positive sequence $(x_n)_{n\geq1}$ is \textit{tame} if $x_{n+k_n}/x_n\to 1$ whenever $k_n=o(n)$.
For instance, $(n^{\alpha}(\log n)^\beta)_{n\geq2}$ is tame for every $\alpha,\beta\in\R$.

\begin{thm}[\textsc{GHP Tauberian theorem}]\label{thm:GHP-Tauberian-thm}
	Let $\left(\Xf_n, n\geq1\right)$ and $\Xf_\infty$ be random compact metric measure spaces.
	Let $(a_n)_{n\geq1}$, $(b_n)_{n\geq1}$ be tame sequences, and let $(N_n)_{n\geq1}$ be random variables which are independent of $(\Xf_n)_{n\geq1}$ and $\Xf_\infty$.
	Assume that:
	\begin{enumerate}
		\item 
			\emph{[Convergence with randomized time.]}
			In distribution for the GHP topology,
			\begin{align*}
				(a_n,b_n)\cdot\Xf_ {N_n}\xrightarrow[n\to\infty]{}\Xf_\infty.
			\end{align*}
		\item
			\emph{[Concentration of the randomized time.]}
			In probability,
			\begin{align*}
				\frac{N_n}{n}\xrightarrow[n\to\infty]{}1.
			\end{align*}
		\item
			\emph{[Tauberian assumption.]}
			The sequence $(\Xf_n,n\geq 1)$ is $\orderGHP$-stochastically non-decreasing, that is, with $\P_n(\diff\X)$ the law of $\Xf_n$ for $n\geq1$,
			\begin{align*}
				\P_1(\diff\X)\storder\P_2(\diff\X)\storder\P_3(\diff\X)\storder\dots.
			\end{align*}
	\end{enumerate}
	Then, $(a_n,b_n)\cdot\Xf_n\to\Xf_\infty$ in distribution for the GHP topology.
\end{thm}

\begin{rem}\label{rem:GHP-taub-theorem-case-lim-not-1}
	If Assumption \emph{(2)}~is replaced by ``$N_n/n\rightarrow c$ in probability'', $c>0$, then we obtain instead that $(a_{\lfloor n/c\rfloor},b_{\lfloor n/c\rfloor})\cdot\Xf_n\to\Xf_\infty$ in distribution for the GHP topology. This is an easy consequence of the case $c=1$.
\end{rem}

\subsection{Application to random planar maps}

A (rooted) \textit{planar map} is a proper embedding in the two-dimensional sphere, without edge-crossings, of a finite connected planar graph with a distinguished oriented edge.
We allow loops and multiple edges.
We identify two planar maps if there is an orientation-preserving homeomorphism of the sphere that maps one onto the other while matching their root-edges.
Given a map $\m$, we equip its vertex-set $V(\m)$ with the graph distance $d_\m$ and the counting measure $\mu_\m=\sum_{v\in V(\m)}\delta_v(\diff x)$, thus forming a discrete metric measure space $(V(\m),d_\m,\mu_\m)$. A \textit{quadrangulation} is a planar map whose faces all have degree $4$.

A major breakthrough in the theory of random planar maps was the proof---by Le~Gall\footnote{
		In the same paper, Le Gall also proved the convergence of suitably renormalized $p$-angulations with $n$ faces, for $p=3$ and all \textit{even} $p\geq 4$.
	}~\cite{LeGall13}
and independently by Miermont~\cite{Miermont13}---that the \textit{Brownian sphere} is the scaling limit of uniformly random quadrangulations with $n$ faces as $n\to\infty$.
In this paper, we will use their result\footnote{
	To be precise, in \cite{LeGall13,Miermont13}, the authors show convergence to $(S,d)$ in the Gromov--Hausdorff topology, which concerns (isometry classes of) compact metric spaces without mention of measures on them. See the discussion by Le~Gall after \cite[Thm~7]{LeGall19} on the measure $\mu$ and on why the result may be strengthened to GHP convergence to the \textit{measured} Brownian sphere $(S,d,\mu)$. Theorem~\ref{thm:Le-Gall--Miermont} is the GHP version thus obtained.
}---reproduced hereafter---as a definition of the Brownian sphere $(S,d,\mu)$, see also Marckert and Mokkadem \cite{MarckertMokkadem06} for its original definition using the so-called \textit{Brownian snake}, and Miller and Sheffield~\cite{MillerSheffield21} for an axiomatic characterization.

\begin{thm}[Le Gall, 2013; Miermont, 2013]\label{thm:Le-Gall--Miermont}
	Letting $Q_n$ be uniformly random in the set $\Quads_n$ of quadrangulations with $n$ faces, we have, in distribution for the GHP topology,
	\begin{align*}
	\left(\vertices(Q_n),\,\Bigl(\frac{9}{8n}\Bigr)^{1/4}\cdot d_{Q_n},\,\frac{1}{n+2}\cdot\mu_{Q_n}\right)\xrightarrow[n\rightarrow\infty]{}(S,d,\mu),
	\end{align*}
	where $(S,d,\mu)$ is the Brownian sphere.
\end{thm}

Convergence to the Brownian sphere has since then been proven for many other models: quadrangulations with no pendant vertices \cite{BeltranLeGall13}, general maps \cite{BettinelliJacobMiermont14}, bipartite maps \cite{Abraham16}, simple quadrangulations/triangulations \cite{AddarioBerryAlbenque17}, cubic maps \cite{CurienLeGall19}, 2-connected quadrangulations \cite{AddarioBerryWen17}, odd-angulations \cite{Addario-BerryAlbenque21}, Eulerian triangulations \cite{Carrance21},  maps with prescribed face degrees \cite{Marzouk22}, general/simple cubic planar graphs \cite{AlbenqueFusyLehericy23, Stufler24}, and Weil--Petersson punctured hyperbolic surfaces~\cite{BuddCurien25}.

We prove a new scaling limit result of this vein, for the model of \textit{irreducible quadrangulations}. A quadrangulation is \textit{irreducible} if it has at least four faces, no multiple edges, and if all its 4-cycles bound a face.

\begin{thm}\label{thm:scaling-limit-irred}
	Letting $\Qirr_n$ be uniformly random in the set $\ensQirred_n$ of \textit{irreducible} quadrangulations with $n$ faces, we have, in distribution for the GHP topology,
	\begin{align*}
		\left(\vertices(\Qirr_n),\,\frac{1}{(8n)^{1/4}}\cdot d_{\Qirr_n},\,\frac{1}{n+2}\cdot\mu_{\Qirr_n}\right)\xrightarrow[n\rightarrow\infty]{\distribGHP}(S,d,\mu),
		\end{align*}
	where $(S,d,\mu)$ is the Brownian sphere.
\end{thm}

A complete proof of Theorem~\ref{thm:scaling-limit-irred} modulo lemmas is given at the end of Section~\ref{sec:structure-proof}.
	It uses our GHP Tauberian theorem to transfer Le~Gall and Miermont's result (Theorem~\ref{thm:Le-Gall--Miermont}) that \emph{general} quadrangulations converge to the Brownian sphere.
	This is accomplished through several largely independent steps: (i)
	the condensation phenomenon allows one to find in a large \emph{general} quadrangulation a large \emph{irreducible}  one, whose size is random but concentrated, as is known from Gao and Wormald \cite{GaoWormald99}; (ii) convergence to the Brownian sphere is transferred from the former to the latter \emph{via} robust topological and measure-theoretic arguments; (iii) the Tauberian theorem in GHP space allows to transfer from irreducible quadrangulations with \emph{random} size to \emph{deterministic} size. The \textit{Tauberian assumption} we need to check will follow from the existence of increasing couplings for irreducible quadrangulations of the hexagon---constructed by Addario-Berry \cite{Addario-Berry14}---and a probabilistic argument to ``remove the hexagon''.

	While the ``upwards'' movement to derive asymptotics of a \textit{host} model from that of a \textit{component} model has been carried successfully in a number of works on random planar maps---see~\cite{AddarioBerryWen17,AlbenqueHoldenSun20,AlbenqueFusyLehericy23,Stufler24} for scaling limits and~\cite{Stufler23_JEMS,Stufler23_Bernoulli} for local limits---the converse direction, which requires a means to ``condition'' a randomized version of the {component model}, has to the best of our knowledge only been considered in \cite{GwynneMiller17} and \cite{BettinelliCurienFredesSepulveda25}.
	In both of these works, a portion of the unconditioned planar map is removed, and then used as a proxy to compare the conditioned and unconditioned versions: in \cite{GwynneMiller17}, Gwynne and Miller rely on fine properties of the peeling process; while in \cite{BettinelliCurienFredesSepulveda25}, Bettinelli, Curien, Fredes, and Sepulveda extract geometric estimates and absolute continuity relations using the known combinatorics of the model.

	By contrast, our unconditioning step is an immediate application of Theorem~\ref{thm:GHP-Tauberian-thm}, provided that the corresponding stochastic monotonicity assumption holds. The caveat is that the latter condition is arguably non-trivial to check, since it requires to provide couplings between uniformly random maps of size $n$ and $n+1$ (from the model of interest), in a ``distance-increasing'' and ``measure-increasing'' way. Still, Caraceni and Stauffer \cite{CaraceniStauffer20,CaraceniStauffer23} have shown that couplings similar to that of Addario-Berry \cite{Addario-Berry14} exist for a number of other models beyond irreducible quadrangulations.

\subsection{Organisation of the article}

In Section~\ref{sec:structure-proof}, taking for granted the GHP Tauberian theorem (Theorem~\ref{thm:GHP-Tauberian-thm}), we decompose our proof of Theorem~\ref{thm:scaling-limit-irred} into a few key lemmas, and conclude the proof modulo these lemmas.
The GHP sandwich theorems (Theorems~\ref{thm:sandwich-thm-GHP} and~\ref{thm:sandwich-thm-GHP-random}) and GHP Tauberian theorem (Theorem~\ref{thm:GHP-Tauberian-thm}) are proven in Section~\ref{sec:sandwich-tauberian-GHP}, along with various properties of the order $\orderGHP$.
Subsequent sections are dedicated to the proof of the key lemmas entering in the proof of Theorem~\ref{thm:scaling-limit-irred}. They go as follows: Section~\ref{sec:largest-components} elaborates on the condensation phenomenon for irreducible components of general quadrangulations observed by Gao and Wormald~\cite{GaoWormald99}; Section~\ref{sec:bottlenecks-and-Hausdorff-cvg} supplies robust topological and measure-theoretic arguments to transfer scaling limit results to ``the largest component'', focusing on obtaining the metric (Gromov--Hausdorff) part of the convergence; Section~\ref{sec:exch-and-Prokhorov} adds the convergence of the measures (Prokhorov part), using ideas of Addario-Berry and Wen~\cite{AddarioBerryWen17}; Section~\ref{sec:growing-irred-of-hexagon} uses couplings by Addario-Berry~\cite{Addario-Berry14} to verify the Tauberian assumption on irreducible quadrangulations \textit{of the hexagon}; and lastly in Section~\ref{sec:removing-hexagon}, we ``get rid of the hexagon''. A concluding section terminates this work.

\subsection*{Acknowledgments}

The author is thankful to Bénédicte Haas and Svante Janson for their comments on a previous version of this work, and to Grégory Miermont for  his careful reading and for many discussions.

This work has been carried for the most part at the ENS de Lyon, and finished at the Université Paris-Saclay (Orsay, France) with support from
SuPerGRandMa (ERC Consolidator Grant no 101087572).

% -----------------------------------------------------------------
%							SECTION 2
% -----------------------------------------------------------------

\section{Structure of the proof of Theorem~\ref{thm:scaling-limit-irred}}
\label{sec:structure-proof}

In this section, we take for granted the ``GHP Tauberian theorem'', Theorem~\ref{thm:GHP-Tauberian-thm}, which will be proved in Section~\ref{sec:sandwich-tauberian-GHP}, and jump straight ahead to the description of the key steps in our proof of Theorem~\ref{thm:scaling-limit-irred}.

We will repeatedly use the notation $\rmUnif(F)$ or $\rmUnif_F$ to denote the uniform probability distribution on a finite non-empty set $F$.

We assume that for every quadrangulation $\q$, we have fixed some arbitrary total order on the vertex-set, face-set, and edge-set of $\q$, which we denote by $<_\q$ in all four cases.
For instance, such orders can be constructed by performing a breadth-first exploration with clockwise priority, starting from the root oriented edge of $\q$.

\subsection{Step I: The largest irreducible component}
Let $4\leq\ell\leq n$.
In Section~\ref{sec:largest-components}, we define a specific set of quadrangulations $\widetilde\Quads$, and we consider, for $n,\ell\geq1$, the set $\DecSet_{n,\ell}$ of tuples $(\q_1,\dots,\q_\ell)\in\widetilde\Quads^\ell$ such that $\sum_i|\q_i|=n+\ell$.
We then show that given an irreducible quadrangulation $\qirr\in\ensQirred_\ell$, a collection $\Decs=(\q_1,\dots,\q_\ell)\in\DecSet_{n,\ell}$, and an integer $\ifrak\in\{1,\dots,4n\}$, we can suitably glue each $\q_j$ to the $j$-th face of $\qirr$ in $<_{\qirr}$-order.
This gives a quadrangulation $\q'$ rooted at the root-edge of $\qirr$, which allows to define $\q=\Gamma(\qirr;\Decs,\ifrak)$ the re-rooting of $\q'$ at the $\ifrak$-th oriented edge in $<_{\q'}$-order.

This construction reflects the fact that quadrangulations contain ``irreducible components'' which can serve as ``building blocks'' to reconstruct them.
A striking fact with most notions of components in planar maps is that there is usually a unique largest component which gathers a positive and well-concentrated proportion of the mass \cite{GaoWormald99, BanderierFlajoletSchaefferSoria01}.
We will use this line of ideas to prove the following.

\begin{lemma}\label{lem:summary-largest-comp}
	There exist random integers $(N_n)_{n\geq 4}$, and random tuples
	\begin{align*}
	(\Qirrfrak_n,\Decs_n,\ifrak_n)
		\qquad\text{in}\qquad\ensQirred_{N_n}\times\DecSet_{n,N_n}\times\{1,\dots,4n\}
	\end{align*}
	for $n\geq 4$, such that:
	(i) the law of $Q'_n:=\Gamma(\Qirrfrak_n;\Decs_n,\ifrak_n)$ is at $o(1)$ total variation distance from $\rmUnif(\Quads_n)$ as $n\to\infty$; (ii)
	$\Qirrfrak_n$ is uniformly distributed in $\smash{\ensQirred_{N_n}}$ conditionally on $N_n$; (iii) we have the deterministic bound $|K_n|\leq n^{2/3+0.1}$ where  $K_n=N_n-\lfloor n/9\rfloor$; and, (iv)  the collection $\Decs_n=(\q_1,\dots,\q_{N_n})\in\DecSet_{n,N_n}$ is exchangeable conditionally on $N_n$.
\end{lemma}

Our strategy to prove Lemma~\ref{lem:summary-largest-comp} is roughly to condition an $\rmUnif(\Quads_n)$-distributed random quadrangulation on having a unique largest irreducible component satisfying the deterministic bound in (iii) above.
Using results of Gao and Wormald \cite{GaoWormald99}, this amounts to conditioning by an event with probability $1-o(1)$.
The other distributional properties shall then follow easily.
Note that the exponent $2/3+0.1$ in the above has been chosen for concreteness.
In fact, Lemma~\ref{lem:summary-largest-comp} still holds if we replace $2/3+0.1$ by $2/3+\epsilon$ for any $\epsilon>0$; and in our proof of Theorem~\ref{thm:scaling-limit-irred}, a deterministic bound $|K_n|=o(n)$ is sufficient.
We prove Lemma~\ref{lem:summary-largest-comp} in Section~\ref{subsec:proof-lemma-summary-largest-irred}.

\subsection{Step II: A topological/measure-theoretic argument}

Our next step is quite general, and completely agnostic to combinatorial details, as it relies on simple topological and measure-theoretic properties of our objects.
Its conclusion can be summarized as follows: if a sequence $(Z_n)_n$ of ``nice'' spaces converges in the GHP sense to a limit $Z_\infty$ which ``has no bottlenecks''; and if for every $n$, the subset $X_n\subset Z_n$ is ``not too small'' and separated from the rest of $Z_n$ by bottlenecks; then actually $X_n\to Z_\infty$ in the GHP sense.
In the preceding, ``nice'' spaces correspond to \textit{compact geodesic metric measure spaces}, the limit ``having no bottlenecks'' means that it is 2-connected and diffuse (see below), and $X_n$ being ``not too small'' means that it has a mass bounded away from $0$.
We refer to Section~\ref{sec:bottlenecks-and-Hausdorff-cvg} for definitions.
Let us now make a more precise statement, in Lemma~\ref{lem:informal-lemma-star-decomp} below.
For convenience, it is stated with the Hausdorff and Prokhorov metrics, but it can be used to prove Gromov--Hausdorff or Gromov--Hausdorff--Prokhorov convergence, using embedding theorems.

Fix some compact metric space $(Z,\delta)$.
For $n\geq1$, we let $Z_n=X_n\cup\bigsqcup_{i\geq1} Y^{(i)}_n$ be a compact subspace of $(Z,\delta)$, where   $X_n$ is a non-empty compact subset of $Z_n$, and where $\smash{(Y^{(i)}_n,i\geq1)}$ is a collection of pairwise disjoint and possibly empty compact subsets of $Z_n$, each one of them intersecting $X_n$.
We suppose that we have chosen some $\smash{y_n^{(i)}\in Y^{(i)}_n}\cap X_n$ for every $n,i\geq1$.
This setup is represented schematically in Figure~\ref{fig:setup-star-decomp}.
\begin{figure}
	\begin{center}
		\includegraphics[scale=.65]{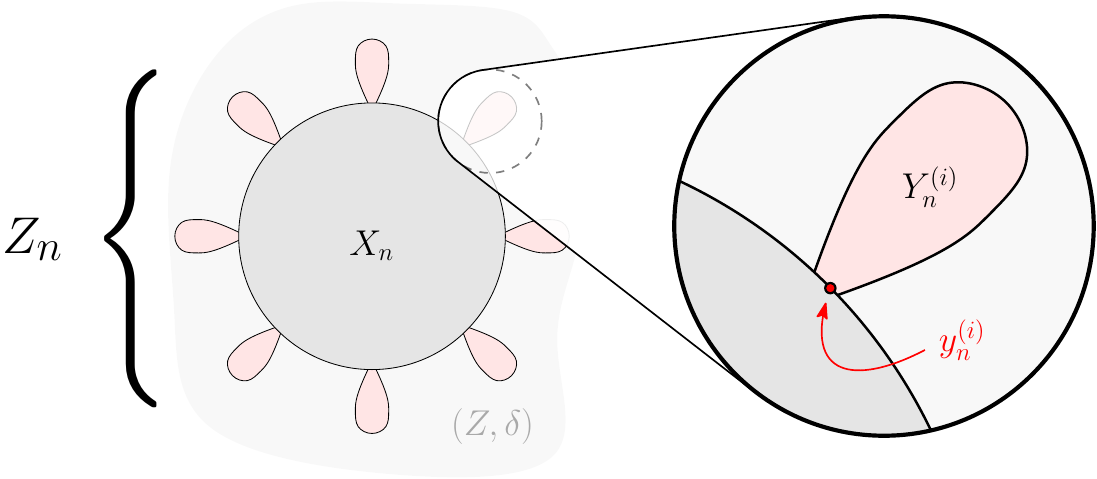}
	\end{center}
	\caption{A schematic representation of the setup of Lemma~\ref{lem:informal-lemma-star-decomp}.}
	\label{fig:setup-star-decomp}
\end{figure}

On top of that, we let $\mu_n$ be a finite Borel measure on $Z_n$ for every $n\geq1$.
Let us write $\gamma^{(i)}_n=\mu_n\bigl(Y^{(i)}_n\setminus X_n\bigr)$, and denote by $\mu^X_n=\mu_n(\cdot\cap X_n)$ the measure $\mu_n$ restricted to $X_n$.
Lastly, define for all $n\geq1$ the measure
\begin{align}
\widetilde\mu_n(\diff x)=\mu^X_n(\diff x)+\sum_{i\geq 1}\gamma^{(i)}_n\cdot\delta_{y^{(i)}_n}(\diff x),
\end{align}
which is supported on $X_n$, and which will serve as an approximation of $\mu_n$.

The following Lemma~\ref{lem:informal-lemma-star-decomp} is a combination of Lemmas~\ref{lem:star-decomposition-with-diffuse-lim},~\ref{lem:star-decomp-discrete-approx} and~\ref{lem:star-decomposition-diffuse-lim-bis}, which are proven in Section~\ref{sec:bottlenecks-and-Hausdorff-cvg}.

\begin{lemma}\label{lem:informal-lemma-star-decomp}
	Suppose that $Z_n\to Z_\infty$ in the $\delta$-Hausdorff metric, that each $Z_n$ is geodesic, and that $Z_\infty$ is 2-connected.
	Suppose additionally that $\mu_n\to\mu_\infty$ in the $\delta$-Prokhorov metric and that $\mu_\infty$ is diffuse.
	If  the following holds
	\begin{align*}
	\textstyle\sup_{i\geq0}\diam(X_n\cap Y^{(i)}_n)\to 0\qquad\text{and}\qquad\textstyle\liminf _n\mu_n(X_n)>0,
	\end{align*}
	then $\sup_{i_\geq1}\diam(Y^{(i)}_n)\to 0$, and thus: (i) we have $X_n\to Z_\infty$ in the $\delta$-Hausdorff metric, and (ii) we also have $\widetilde\mu_n\to\mu_\infty$ in the $\delta$-Prokhorov metric.
	Furthermore, we also have $\smash{\sup_{i\geq1}\gamma^{(i)}_n\to 0}$ as $n\to\infty$.
\end{lemma}

Let us succinctly  explain how this is useful in our context.
First, we can turn maps into geodesic spaces using the so-called \textit{metric graph} associated to a graph by replacing its edges by unit segments.
Also, in Lemma~\ref{lem:summary-largest-comp}, the irreducible quadrangulation $\Qirrfrak_n$ is realized as a submap of $Q'_n$, in such a way that its metric is the metric induced by $Q'_n$---that is, informally, that one cannot find shortcuts between points of $\Qirrfrak_n$ by venturing outside of $\Qirrfrak_n$.
This later fact will be easily proven, but it is crucial since it allows to consider the metric graph of $\Qirrfrak_n$ as a (metric) subspace $X_n$ of $Z_n$---the metric graph of $Q'_n$.

Once distances are renormalized by $n^{-1/4}$ and masses by $n^{-1}$, we are in a similar situation as in Lemma~\ref{lem:informal-lemma-star-decomp}, since: (i) we know by Lemma~\ref{lem:summary-largest-comp} that $X_n$ has (renormalized) mass concentrated around $1/9$ and thus bounded away from zero; (ii) the connected components of $Z_n\setminus X_n$ can ``touch'' $X_n$ only along a 4-cycle, which has (renormalized) diameter $4\cdot n^{-1/4}\to 0$; and (iii) $Z_n$, equipped with the (renormalized) mass measure of $Q'_n$, converges to the Brownian sphere using the results of Le~Gall and Miermont recalled as Theorem~\ref{thm:Le-Gall--Miermont}.
Note that the Brownian sphere is almost surely homeomorphic to the two-dimensional sphere \cite{LeGallPaulin08,Miermont08} and thus 2-connected, and it has diffuse measure by \cite[Thm.~3]{Miermont09}.

All in all, we will be able to use Lemma~\ref{lem:informal-lemma-star-decomp} to prove Lemma~\ref{lem:cvg-after-projection}.
Let $(N_n,\Qirrfrak_n,\Decs_n)$, $n\geq4$, be as in Lemma~\ref{lem:summary-largest-comp}.
For $n\geq4$, if $\Decs_n=(\q_1,\dots,\q_{N_n})$, then we set $\boldnu_{\Decs_n}(\diff x)=\sum_{j} \bigl(|V(\q_j)|-4\bigr)\cdot \delta_{v_j}(\diff x)$,
where $v_j$ is a uniformly chosen vertex incident to the $j$-th face of $\Qirrfrak_n$ in $<_{\Qirrfrak_n}$-order.
 
\begin{lemma}\label{lem:cvg-after-projection}
	We have the convergence
	\begin{align*}
	\left(\vertices(\Qirrfrak_n),\,\left(\frac{9}{8n}\right)^{1/4}\cdot d_{\Qirrfrak_n},\,\frac{1}{n+2}\cdot(\mu_{\Qirrfrak_n}+\boldnu_{\Decs_n})\right)\xrightarrow[n\rightarrow\infty]{\distribGHP}(S,d,\mu),
	\end{align*}
	where $(S,d,\mu)$ is the Brownian sphere.
	Furthermore, the collection $(\frac{1}{n+2}\cdot\boldnu_{\Decs_n}(v_j),1\leq j\leq N_n)$ is exchangeable conditionally on $N_n$, and its supremum converges to $0$ in probability as $n\to\infty$.
\end{lemma}

Lemma~\ref{lem:cvg-after-projection} will be proven in Section~\ref{subsec:proof-cvg-after-proj}.

\subsection{Step III: Concentration for exchangeable random measures}

In order to control the measure $\frac{1}{n+2}\cdot(\mu_{\Qirrfrak_n}+\boldnu_{\Decs_n})$ in Lemma~\ref{lem:cvg-after-projection}, we will use the following result, which we prove in Section~\ref{sec:exch-and-Prokhorov}.
It is greatly inspired by arguments of Addario--Berry and Wen \cite{AddarioBerryWen17}.

We let $(Z,\delta)$ be a compact metric space, and we consider a sequence $(D_n,n\geq1)$ of finite subsets of $Z$.
Suppose that $(\boldnu_n,n\geq1)$ are \textit{random} measures supported on $D_n$ respectively.

\begin{prop}[See Prop.~\ref{prop:Add-Wen-exch-argument}]\label{prop:Add-Wen-in-intro}
	Assume that the collection $(\boldnu_n(x))_{x\in D_n}$ is exchangeable
	for every $n\geq1$.
	If $\max_{x\in D_n} \boldnu_n(x) \to 0$ in probability, then
	\begin{align*}
	\dP\Bigl(\boldnu_n,|\boldnu_n|\cdot\rmUnif({D_n})\Bigr)\xrightarrow[n\to\infty]{}0
	\end{align*}
	in probability.
\end{prop}

Let again $(\Qirrfrak_n,\Decs_n)$, $n\geq4$, be as given by Lemma~\ref{lem:summary-largest-comp}.
Given a map $\m$, we let $\X(\m)=(V(\m),d_\m,\mu_\m)$.
Given a metric measure space $\X=(X,d,\mu)$ and $a,b>0$, we write $(a,b)\cdot \X=(X,a\cdot d, b\cdot\mu)$.
By combining Proposition~\ref{prop:Add-Wen-in-intro} and Lemma~\ref{lem:cvg-after-projection}, we will deduce the following.

\begin{lemma}\label{lem:cvg-after-concentration}
	If we set $a_n = (9/8n)^{1/4}$ and $b_n=9/(n+2)$, then we have the convergence:
	\begin{align*}
	(a_n,b_n)\cdot\X(\Qirrfrak_n)\xrightarrow[n\rightarrow\infty]{\distribGHP}(S,d,\mu),
	\end{align*}
	where $(S,d,\mu)$ is the Brownian sphere.
\end{lemma}

Lemma~\ref{lem:cvg-after-concentration} is proven in Section~\ref{subsec:proof-cvg-after-concentration}.

\subsection{Step IV: Stochastic increase for irreducible quadrangulations of the hexagon}\label{subsec:step-growing-irred-hex}
As we will see, the convergence in Lemma~\ref{lem:cvg-after-concentration} is amenable to the application of the ``GHP Tauberian theorem'', Theorem~\ref{thm:GHP-Tauberian-thm}. To verify the \textit{Tauberian assumption}, we will need the following stochastic increase statement regarding irreducible quadrangulations \textit{of the hexagon}. The next step will then ``remove the hexagon''.

We let $\ensQirredhex_{n}$ denote the set of \textit{irreducible quadrangulations of the hexagon} with $n$ faces, namely planar maps in which (i) the face to the right of the root edge has degree $6$ while all other faces have degree $4$, (ii) there are no multiple edges, and (iii) all 4-cycles bound a face.

\begin{lemma}\label{lem:stoch-increase-irred-hexagon}
	There exists a coupling $(\incrQirrhex_n,n\geq4)$ of the uniform distributions on $(\ensQirredhex_{n}, n\geq 4)$ respectively, such that $\X(\incrQirrhex_n)\orderGHP \X(\incrQirrhex_{n+1})$ for every $n\geq4$.
\end{lemma}

Lemma~\ref{lem:stoch-increase-irred-hexagon} is the combination of Proposition~\ref{prop:Addario-Berry-coupling-irred} and Lemma~\ref{lem:opening-face-gives-GHP-increase}, proven in Section~\ref{sec:growing-irred-of-hexagon}.
It will be deduced in a straightforward way from the existence of a coupling for the uniform distributions on $(\ensQirredhex_{n}, n\geq 4)$ which only perform a ``face-opening'' at each step.
Such a coupling has been constructed by Addario--Berry~\cite{Addario-Berry14}, using Luczak and Winkler's increasing coupling for binary trees \cite{LuczakWinkler04}, and a bijection by Fusy, Poulalhon and Schaeffer \cite{FusySchaefferPoulalhon08} relating binary trees and irreducible quadrangulations of the hexagon.
There are other models of random planar maps for which such couplings exist, as proven by Caraceni and Stauffer~\cite{CaraceniStauffer23}.

\subsection{Step V: Removing the hexagon}
The preceding Lemma~\ref{lem:stoch-increase-irred-hexagon} is not \textit{exactly} what we need to apply the GHP Tauberian theorem, Theorem~\ref{thm:GHP-Tauberian-thm}, since it deals with irreducible quadrangulations \textit{of the hexagon}, and not plain irreducible quadrangulations.
In order to compensate for this fact, we prove in Section~\ref{sec:removing-hexagon} the following statement.

\begin{lemma}[See Prop.~\ref{prop:dTV-coupling-irred}]\label{lem:comparison-irred-vs-hexagon}
	There exists an irreducible quadrangulation of the hexagon $\pat$ with nine quadrangular faces, as well as couplings $(\Qirred_n,\Qirredhex_{n-9})_{n\geq 11}$ between---respectively---the uniform distributions on $\ensQirred_n$ and $\ensQirredhex_{n-9}$, such that with probability $1-o(1)$, the quadrangulation $\Qirred_n$ is obtained from $\Qirredhex_{n-9}$ by gluing $\pat$ to its hexagon and choosing a new oriented root edge.
\end{lemma}

This is the last key lemma that we needed in order to deduce the scaling limit of irreducible quandrangulations, Theorem~\ref{thm:scaling-limit-irred}.

\subsection{Proof of Theorem~\ref{thm:scaling-limit-irred}}
In this proof, we take for granted the GHP Tauberian theorem, Theorem~\ref{thm:GHP-Tauberian-thm}, and the preceding lemmas. These results will be proven in the remaining sections.

	We let $(\Qirred_n,\Qirredhex_{n-9})_{n\geq 11}$ be the coupling given by Lemma~\ref{lem:comparison-irred-vs-hexagon}.
	With the notation of this lemma, observe%
		\footnote{
			To see this, use the canonical embedding of $(a,b)\cdot\X(\Qirredhex_{n-9})$ into $(a,b)\cdot\X(\Qirred_n)$ given by the gluing operation described in Lemma~\ref{lem:comparison-irred-vs-hexagon}, and then bound their Hausdorff and Prokhorov distance under this embedding. 
		}
	that for every $a,b>0$,
	\begin{align}\label{eq:proof-scaling-lim-irred-1}
	\dGHP\Bigl((a,b)\cdot\X\bigl(\Qirred_n\bigr),(a,b)\cdot\X\bigl(\Qirredhex_{n-9}\bigr)\Bigr)
		\leq a\cdot\diam (\pat) + b \cdot|V(\pat)|.
	\end{align}
	We also let $\Qirrfrak_n$, $n\geq4$, be as given by Lemma~\ref{lem:summary-largest-comp}.
	By this very lemma, for every $n\geq4$, we have the identity in distribution:
	\begin{align}\label{eq:proof-scaling-lim-irred-2}
	\Qirrfrak_n\overset{(\mathrm d)}{=}\Qirred_{N_n},
	\end{align}
	where $N_n=\lfloor n/9\rfloor+K_n$ and $K_n$, $n\geq4$ are random integers, independent from all previously defined random variables, and satisfying the deterministic bound $|K_n|\leq n^{2/3+o(1)}$.
	Let us set $a_n = (9/8n)^{1/4}$, $b_n=9/(n+2)$, for all $n\geq 11$.
	Recall that by Lemma~\ref{lem:cvg-after-concentration}, we have
	\begin{align*}
	(a_n,b_n)\cdot\X(\Qirrfrak_n)\xrightarrow[n\rightarrow\infty]{\distribGHP}(S,d,\mu),
	\end{align*}
	where $(S,d,\mu)$ is the Brownian sphere.
	This can be transferred to a statement concerning the sequence $(\Qirredhex_{n})_n$ using first \eqref{eq:proof-scaling-lim-irred-2} and then \eqref{eq:proof-scaling-lim-irred-1}.
	For $n\geq4$, this gives:
	\begin{align*}
	(a_n,b_n)\cdot\X(\Qirredhex_{N_{n-9}})\xrightarrow[n\rightarrow\infty]{\distribGHP}(S,d,\mu).
	\end{align*}
	We can now apply the ``GHP Tauberian theorem'', Theorem~\ref{thm:GHP-Tauberian-thm}.
	Indeed, let us check its assumptions: clearly the sequences $(a_n)_n$ and $(b_n)_n$ are tame; the last display verifies Assumption 1.~of this theorem; then Assumption 2.~is also verified since $N_n=\lfloor n/9\rfloor+K_n$ with $|K_n|\leq n^{2/3+o(1)}$ so that $N_{n}/n\to 1/9$ deterministically as $n\to\infty$; lastly Assumption~3.~corresponds to the stochastic increase established by Lemma~\ref{lem:stoch-increase-irred-hexagon}.
	
	All in all, we deduce from the ``GHP Tauberian theorem'', Theorem~\ref{thm:GHP-Tauberian-thm}, that%
		\footnote{
			Note that we have $N_n/n\to 1/9$ instead of $N_n/n\to 1$, so that we need to adapt the conclusion of Theorem~\ref{thm:GHP-Tauberian-thm} as described in Remark~\ref{rem:GHP-taub-theorem-case-lim-not-1}.
		}
	$(a_{9n},b_{9n})\cdot\X(\Qirredhex_{n})\to (S,d,\mu)$ in distribution for the GHP topology.
	Hence, using \eqref{eq:proof-scaling-lim-irred-1} again, we get $(a_{9n},b_{9n})\cdot\X(\Qirred_{n+9})\to (S,d,\mu)$.
	Since $a_{9n}\sim a_{9(n-9)}\sim 1/(8n)^{1/4}$ and $b_{9n}\sim b_{9(n-9)}\sim 1/(n+2)$, we obtain by Lemma~\ref{cor:continuity-rescaling-random}---which is a standard property of the GHP topology, proven in Section~\ref{subsec:corr-and-couplings} for completeness---that:
	\begin{align*}
	\Bigl(\tfrac{1}{(8n)^{1/4}},\tfrac{1}{n+2}\Bigr)\cdot\X(\Qirred_{n})\to (S,d,\mu).
	\end{align*}
	This concludes the proof of Theorem~\ref{thm:scaling-limit-irred}.
\qed

\bigskip

We now turn to the proof of the intermediate results which have been used in the proof of Theorem~\ref{thm:scaling-limit-irred}. The next section deals with general properties of the partial order $\orderGHP$, establishing in particular the sandwich theorems, Theorem~\ref{thm:sandwich-thm-GHP} and~\ref{thm:sandwich-thm-GHP-random}, as well as the GHP Tauberian theorem, Theorem~\ref{thm:GHP-Tauberian-thm}, which was used above.

% -----------------------------------------------------------------
%							SECTION 3
% -----------------------------------------------------------------

\section{Sandwich and Tauberian theorems in the GHP space}
\label{sec:sandwich-tauberian-GHP}

Given some metric space $(X,d)$, we denote by $\mesFin(X)$, resp.~$\mes_1(X)$, the set of \textit{finite} Borel measures, resp.~Borel \textit{probability} measures.

\subsection{The GH and GHP distances}
\label{subsec:preliminaries-GHP}

We recall the definition of the Gromov--Hausdorff and Gromov--Hausdorff--Prokhorov distances.
For more details, the reader is referred to \cite{BuragoBuragoIvanov01}, \cite[Section 6]{Miermont09} and \cite{AbrahamDelmasHoscheit13}.
We use the definition%
\footnote{
	To be precise, in \cite{AbrahamDelmasHoscheit13} the authors define a GHP distance for \textit{rooted} spaces, but their results adapt \textit{mutatis mutandi} to the unrooted case.
}
of the GHP distance found in the latter reference, which is an extension to compact spaces endowed with a finite Borel measure of the original definition for spaces endowed with a probability measure.

Let $(Z,\delta)$ be a metric space.
For compact subsets $K,K'\subset Z$, their $\delta$-\textit{Hausdorff distance} is
\begin{equation*}
	\dH(K, K') = \inf \bigl\{ \epsilon>0 \colon K\subset (K')^\epsilon \text{ and } K' \subset K^\epsilon \bigr\},
\end{equation*}
where for every $ A\subset Z$ and every $\epsilon>0$ we let $A^\epsilon=\{x\in Z: \delta(x,A)<\epsilon\}$ be the $\epsilon$-neighborhood of $A$ in $(Z,\delta)$.
The $\delta$-\textit{Prokhorov distance} between two finite measures $\mu,\mu'$ on the Borel $\sigma$-algebra $\Bcal_Z$ of $(Z,\delta)$ is
\begin{equation*}
	\dP(\mu,\mu') = \inf \bigl\{ \epsilon>0 \colon \forall A \in \Bcal_Z,\, \mu(A)\leq \mu'(A^\epsilon)+\epsilon \text{ and } \mu'(A) \leq \mu(A^\epsilon)+\epsilon \bigr\}.
\end{equation*}
The topology on $\mesFin(Z)$ generated by the $\delta$-Prokhorov distance is that of weak convergence (convergence against continuous bounded functions).

The \textit{Gromov--Hausdorff distance} between two (isometry classes of) compact metric spaces $(X,d)$ and $(X',d')$ is
\begin{align*}
	\dGH\bigl((X,d),(X',d')\bigr) = \inf_{(Z,\delta),\phi,\phi'} \dH\bigl(\phi(X), \phi'(X')\bigr),
\end{align*}
where the infimum runs over all metric spaces $(Z,\delta)$ and all isometric embeddings $\phi\colon (X,d)\to (Z,\delta) $ and $\phi'\colon (X',d') \to (Z,\delta) $.
This defines a metric on the quotient $\setGH=\bbK/\approx$ of the space $\bbK$ of compact metric spaces with respect to the relation $\approx$ of isometry.

The \textit{Gromov--Hausdorff--Prokhorov distance} between two compact metric measure spaces $\X=(X,d,\mu)$ and $\X'=(X',d',\mu')$ is
\begin{equation*}
	\dGHP(\X , \X')
	= \inf_{(Z,\delta),\phi,\phi'}\Bigl( \dH\bigl(\phi(X), \phi'(X')\bigr)  + \dP\bigl(\phi_*\mu, \phi'_* \mu'\bigr) \Bigr),
\end{equation*}
where again the infimum runs over all metric spaces $(Z,\delta)$ and all isometric embeddings $\phi\colon (X,d)\to (Z,\delta) $ and $\phi'\colon (X',d') \to (Z,\delta) $.
As mentioned in the introduction, we have $\dGHP(\X,\X')=0$ precisely when there exists a measure-preserving isometry between $\X$ and $\X'$, so that this only defines a pseudo-metric on $\bbM$, which induces an actual metric $\dGHP$ on the quotient space $\setGHP=\bbM/{\isom}$ with respect to measure-preserving isometry.
The space $(\setGHP,\dGHP)$ is a Polish space.

\subsection{Correspondences and couplings}
\label{subsec:corr-and-couplings}
There is a useful bound on $\dGHP$, which is frequently used in the literature, in terms of \textit{correspondences} and \textit{couplings}.
See \cite[Section 6]{Miermont09} for the GHP metric on compact spaces endowed with a Borel probability measure, and~\cite[Section 2.1]{LeGall19} for the GHP metric on compact spaces endowed with a finite Borel measure, which is our setting.

A correspondence between metric spaces $(X,d)$ and $(X',d')$ is a subset $R$ of $X \times X'$ such
that $\pi(R)=X$ and $\pi'(R)=X'$, with the projections $\pi\colon X \times X'\to X$ and $\pi'\colon X \times X'\to X'$. 
The distortion of R is
\begin{align*}
	\distort(R)=\sup\Bigl\{|d(x,y)-d'(x',y')|\colon (x,x')\in R,\, (y,y')\in R\Bigr\}.
\end{align*}

\begin{lemma}[Lemma~4 in~\cite{LeGall19}]\label{lem:bound-corresp}
	Let $\epsilon>0$ and suppose that there is a correspondence $R$ between some $(X,d)$ and $(X',d')$ with distortion bounded above by $\epsilon$ and a finite measure
	$\nu$ on the product $X \times X'$ such that $\nu(R^c) < \epsilon$ and
	\begin{align*}
		d_{\mathrm{P}}(\pi_*\nu,\mu)<\epsilon,\quad d'_{\mathrm{P}}(\pi'_*\nu,\mu')<\epsilon.
	\end{align*}
	Then $\dGHP((X,d,\mu),(X',d',\mu'))<3\epsilon$.
\end{lemma}

Given $a,b>0$ and $\X=(X,d,\mu)$ a compact metric measure space, we write $(a,b)\cdot \X=(X,a\cdot d, b\cdot \mu)$, which is the same space with the distance $d$ scaled by $a$ and the measure $\mu$ scaled by $b$.
We will need the following easy lemma.

\begin{lemma}\label{lem:continuity-rescaling}
	Let $(a_n,b_n)\to(1,1)$ in $\R_+^2$ and let $\X_n\to\X_\infty$ in the GHP metric.
	Then $(a_n,b_n)\cdot\X_n\to\X_\infty$ in the GHP metric.
\end{lemma}

\begin{proof}
	We write $\X_n=(X_n,d_n,\mu_n)$ for $n\geq1$ and $\X_\infty=(X_\infty,d_\infty,\mu_\infty)$.
	We let the reader check that for all $n\geq1$, the hypotheses of Lemma~\ref{lem:bound-corresp} are satisfied for 
	\begin{align*}
		\epsilon_n=2\cdot\max\Bigl(|a_n-1|\cdot\diam(\X_n),|b_n-1|\cdot\size{\mu_n}\Bigr);
	\end{align*}
	for $R_n$ the ``diagonal correspondence'' between $(a_n,b_n)\cdot\X_n$ and $\X_n$ given by $R_n=\Delta(X_n)$ with $\Delta\colon x\mapsto(x,x)$; and for $\nu_n=\Delta_*\mu_n$.
	Hence by Lemma~\ref{lem:bound-corresp}, $\dGHP((a_n,b_n)\cdot\X_n,\X_n)\leq 3\epsilon_n$. But $\epsilon_n\to 0$ since $(\diam(\X_n))_n$ and $(\size{\mu_n})_n$ are bounded whenever $(\X_n)_n$ converges in the GHP metric, as is easily verified.
\end{proof}

For $a,b>0$ and $\P(\diff\X)\in\mes_1(\setGHP)$, we let $(a,b)_*\P(\diff\X)$ be the image measure of $\P(\diff\X)$ with respect to the mapping $\X\mapsto (a,b)\cdot\X$.

\begin{cor}\label{cor:continuity-rescaling-random}
	Let $(a_n,b_n)\to(1,1)$ in $\R_+^2$ and let $\P_n(\diff\X)\to\P_\infty(\diff\X)$ weakly with respect to the GHP topology.
	Then $(a_n,b_n)_*\P_n(\diff\X)\to\P_\infty(\diff\X)$ weakly with respect to the GHP topology.
\end{cor}

\begin{proof}
	Let us write $\overline\N=\{1,2,\dots\}\cup\{\infty\}$.
	Since $(\setGHP,\dGHP)$ is Polish, we may apply the Skhorokhod representation theorem to get a coupling $(\X_n,n\in\overline\N)$ of the distributions $(\P_n(\diff\X),n\in\overline\N)$ such that $\X_n\to\X_\infty$ almost surely in the GHP metric.
	Then $(a_n,b_n)\cdot\X_n\to\X_\infty$ almost surely in the GHP metric by Lemma~\ref{lem:continuity-rescaling}.
	In particular, $(a_n,b_n)_*\P_n(\diff\X)\to\P_\infty(\diff\X)$ weakly with respect to the GHP topology.
\end{proof}

\subsection{Isometric embeddings and an Arzela--Ascoli theorem}

The following proposition is an extension to the GHP metric of a result for the GH metric proved by Gromov in his seminal paper~\cite{Gromov81}, see p.~65 therein.
For similar results with related ``Gromov-type'' metrics, see also \cite[Lem.~5.8 and A.1]{GrevenPfaffelhuberWinter09} for the Gromov-Hausdorff and Gromov-Prokhorov metrics, and \cite[Prop.~1.5]{GwynneMiller17} for the Gromov-Hausdorff-Prokhorov-uniform metric.

The statement and proof of \cite[Prop.~1.5]{GwynneMiller17}, for the Gromov-Hausdorff-Prokhorov-uniform (GHPU) metric on compact metric measure spaces endowed with a distinguished continuous curve, adapt \textit{mutatis mutandi} to the GHP metric.
This gives the following statement.

\begin{prop}[{see \cite[Prop.~1.5]{GwynneMiller17}}]\label{prop:GHP-common-embedding}
	Let $\X_\infty=(X_\infty,d_\infty,\mu_\infty)$ and $\X_n=(X_n,d_n,\mu_n)$ for $n\geq1$ be elements of $\setGHP$.
	Then $\X_n\rightarrow\X_\infty$ in the GHP metric if and only if there exists a compact metric space $(Z,\delta)$ and isometric embeddings $X_\infty\to Z$ and $X_n\to Z$ for $n\geq1$ such that if we identify $X_\infty$ and $X_n$ with their images under
	these embeddings, then $\dH(X_n,X_\infty)\to 0$ and $\dP(\mu_n,\mu_\infty)\to 0$.
\end{prop}

We will also need an Arzela--Ascoli theorem in the context of sequences of compact spaces which converge in the Hausdorff sense \cite[Appendix]{GrovePetersen91}.

Let $(Z,\delta)$ and $(Z',\delta')$ be two compact metric spaces.
For $n\geq1$, let $f_n\colon X_n\to X'_n$, where $X_n$ and $X'_n$ are compact subsets of $(Z,\delta)$ and $(Z',\delta')$ respectively.
The sequence $(f_n)_{n\geq1}$ is said to be \textit{equicontinuous} if for any $\epsilon>0$, there is a $\eta>0$ such that,
\begin{align*}
	\forall x,y\in X_n,\quad \delta(x,y)<\eta\implies \delta'(f_n(x),f_n(y))<\epsilon,
\end{align*}
for all $n\geq1$.

\begin{prop}[Appendix in \cite{GrovePetersen91}]\label{prop:GH-Arzela-Ascoli}
	In this setting, if $X_n\to X_\infty$ in the $\delta$-Hausdorff metric, if $X'_n\to X'_\infty$ in the $\delta'$-Hausdorff metric, and if the sequence $(f_n)_{n\geq1}$ is equicontinuous, then one can extract a subsequence which converges to a continuous mapping $f_\infty\colon X_\infty\to X'_\infty$, in the sense that along that subsequence, $f_n(x_n)\to f_\infty(x_\infty)$ in $(Z',\delta')$ whenever $x_n\to x_\infty$ in $(Z,\delta)$.
\end{prop}

As noted by the authors in \cite{GrovePetersen91}, Proposition~\ref{prop:GH-Arzela-Ascoli} is proved just as the classical Arzela--Ascoli theorem, making necessary adjustments.
Note that in Gromov's original paper \cite{Gromov81}, a version of this statement for isometries figures at p.~66.
For a different exposition, see also~\cite[Thm.~27.20]{Villani09}.

We will need to complement Proposition~\ref{prop:GH-Arzela-Ascoli} with the following lemma, which heuristically states that ``$(f,\mu)\mapsto \mu\circ f^{-1}$'' is upper semi-continuous.

\begin{lemma}\label{lem:GH-Arzela-Ascoli-strenghened}
	In the setting of Proposition~\ref{prop:GH-Arzela-Ascoli}, if we have finite Borel measures $\mu_1,\mu_2,\dots,\mu_\infty$ on $X_1,X_2,\dots,X_\infty$ respectively such that $\dP(\mu_n,\mu_\infty)\to 0$, then for every closed subset $A'$ of $X'_\infty$ we have
	\begin{align}\label{eq:asymptotic-inequality-Arzela-Ascoli}
		\limsup_{n\to\infty}\mu_n(f_n^{-1}(A'))\leq \mu_\infty(f^{-1}_\infty(A')).
	\end{align}
\end{lemma}

\begin{proof}
	Fix some closed subset $A'$ of $X'_\infty$.
	We claim that for every $\epsilon>0$ we have $f_n^{-1}(A')\subset(f_\infty^{-1}(A'))^{\epsilon}$ for all $n\geq1$ large enough.
	This implies:
	\begin{align*}
		\limsup_{n\to\infty}\mu_n(f_n^{-1}(A'))
		&\leq	\lim_{\epsilon\to 0}\,\limsup_{n\to\infty}\,\mu_n\Bigl(\,\overline{(f_\infty^{-1}(A'))^{\epsilon}}\,\Bigr)\\
		&\leq	\lim_{\epsilon\to 0}\,\mu_\infty\Bigl(\,\overline{(f_\infty^{-1}(A'))^{\epsilon}}\,\Bigr)\\
		&=\mu_\infty(f^{-1}_\infty(A')),
	\end{align*}
	where for the second line we use the Portmanteau theorem, while the third follows by continuity of measures, since $f^{-1}_\infty(A')$ is closed in $X_\infty$ by continuity of $f_\infty$.
	
	Let us now justify the claim.
	Suppose for contradiction that there exists $\epsilon>0$ and infinitely many $n\geq1$ such that we can find $x_n\in f_n^{-1}(A')\setminus(f_\infty^{-1}(A'))^{\epsilon}$.
	Taking if necessary a subsequence, we may assume that $x_n\to x_\infty$ in the compact $(Z,\delta)$.
	Then, $f_\infty(x_\infty)=\lim_n f_{n}(x_{n})$ is in $A'$ since $A'$ is closed in $(Z',\delta')$.
	But $\delta(x_\infty,f_\infty^{-1}(A'))=\lim_n \delta(x_n,f_\infty^{-1}(A'))\geq\epsilon$, hence a contradiction.
\end{proof}

\begin{rem}
	There need not be equality in \eqref{eq:asymptotic-inequality-Arzela-Ascoli} since $f_\infty^{-1}(A')$ can be much larger than every $f_n^{-1}(A')$, $n\geq1$.
	For instance, if $f_n\colon [0,1]\to[0,1],\,x\mapsto x/n$ for all $n\geq1$, then $f_\infty^{-1}(\{0\})=[0,1]$ while $f_n^{-1}(\{0\})=\{0\}$ for all $n\geq1$.
\end{rem}

\subsection{Properties of the partial order \texorpdfstring{$\orderGHP$}{}}
\label{subsec:prop-order-GHP}

We recall that the partial order $\orderGHP$ is defined as follows.
Given $\X=[X,d,\mu]$ and $\X'=[X',d',\mu']$ in $\setGHP$, we have $\X\orderGHP \X'$ if and only if there exists a \textit{surjective}
mapping $\phi\colon X'\rightarrow X$ which is
\begin{enumerate}
	\item \textit{1-Lipschitz}, that is:
	\begin{align*}
		\forall x',y'\in X',\ \quad& d\bigl(\phi(x'),\phi(y')\bigr)\leq d'(x',y');
	\end{align*}
	\item \textit{measure-contracting}, that is:
	\begin{align*}
		\forall A\in\Bcal,\ \quad& \mu(A) \leq \mu'\bigl(\phi^{-1}(A)\bigr),
	\end{align*}
	where $\Bcal$ is the Borel $\sigma$-algebra of $(X,d)$.
\end{enumerate}
Let us start with the following lemma.

\begin{lemma}\label{lem:isometry}
	Let $(X,d)$ be a compact metric space.
	If a surjective self-mapping $\psi\colon X\to X$ is 1-Lipschitz and measure-contracting, then it is a measure-preserving isometry.
\end{lemma}

\begin{proof}
	It is classical that a 1-Lipschitz surjective self-mapping is a (bijective) isometry, see \textit{e.g.}~\cite[Thm.~1.6.15]{BuragoBuragoIvanov01}.
	Let $\psi$ be a surjective and measure-\textit{contracting} self-mapping $\psi\colon X\to X$, and let us show that it is actually measure-\textit{preserving}.
	For every $A$  in the Borel $\sigma$-algebra $\Bcal$ of $X$, its complement $A^c$ is also in $\Bcal$, and using that $\psi$ is measure-contracting we have
	\begin{align*}
		\mu(A)\leq\mu\bigl(\psi^{-1}(A)\bigr)
		\qquad\text{and}\qquad
		\mu(A^c)\leq\mu\bigl(\psi^{-1}(A^c)\bigr).
	\end{align*}
	Since $\psi^{-1}(A^c)$ is the complement of $\psi^{-1}(A)$ in $X$, we have $\mu(\psi^{-1}(A))+\mu(\psi^{-1}(A^c))=|\mu|=\mu(A)+\mu(A^c)$.
	Hence the above inequalities must be equalities, so that $\mu(A)=\mu(\psi^{-1}(A))$ and $\psi$ is measure-preserving.
\end{proof}

\begin{prop}\label{prop:checking-partial-order}
	The relation $\orderGHP$ is a partial order on $\setGHP$.
\end{prop}

\begin{proof}
	The transitivity and reflexivity of $\orderGHP$ are clear.
	What is needed for $\orderGHP$ to be a partial order is therefore its anti-symmetry.
	This amounts to proving that for $\X=(X,d,\mu)$ and $\X'=(X',d',\mu')$ two compact metric measure spaces, if $\X\orderGHP\X'\orderGHP\X$ then $\X$ and $\X'$ are measure-preserving isometric.
	Hence let $\X,\X'$ be such that $\X\orderGHP\X'\orderGHP\X$, with respective Borel $\sigma$-algebras $\Bcal$ and $\Bcal'$.
	By definition, there exist surjective mappings $\phi\colon X'\to X$ and $\phi'\colon X\to  X'$ such that:
	\begin{multline*}
		\left\{
		\begin{matrix*}[l]
			\forall x',y'\in X',\quad &d(\phi(x'),\phi(y')) \leq d'(x',y'),\\
			\forall A\in\Bcal,\quad &\mu(A) \leq \mu'(\phi^{-1}(A)),
		\end{matrix*}
		\right.\\
		\text{and}\qquad
		\left\{
		\begin{matrix*}[l]
			\forall x,y\in X,\quad &d'(\phi'(x),\phi'(y))\leq d(x,y),\\
			\forall A'\in\Bcal',\quad &\mu'(A') \leq \mu((\phi')^{-1}(A')).
		\end{matrix*}
		\right.
	\end{multline*}
	In particular, the self-mapping $\psi\colon X\to X$ defined by $\psi=\phi\circ\phi'$ satisfies the hypotheses of Lemma~\ref{lem:isometry} and it is therefore a measure-preserving isometry.
	Hence the inequalities in the last display must be equalities and $\phi$ and $\phi'$ are measure-preserving isometries.
\end{proof}

The following properties are useful in order to prove a sandwich theorem, see Lemma~\ref{lem:general-sandwich-theorem}.

\begin{defin}\label{def:closed-and-down-compact}
	Let $\preceq$ be a partial order on some metric space $\Xtop$.
	It is said to be:
	\begin{enumerate}
		\item \emph{closed} if its graph $\{(x,y)\in X\colon x\preceq y\}$ is closed in $\Xtop$,
		\item \emph{down-compact} if given any compact $K^+\subset \Xtop$, the set $$\down_\preceq(K^+)=\{x\in \Xtop\colon \exists x^+\in K^+,\, x\preceq x^+\}$$ is also compact.
	\end{enumerate}
\end{defin}

We now verify that the order $\orderGHP$ is indeed closed and down-compact.

\begin{prop}\label{prop:checking-order-is-closed}
	The partial order $\orderGHP$ on $(\setGHP,\dGHP)$ is closed.
\end{prop}

\begin{proof}
	We have to prove that if $\X_n\to\X_\infty$ and $\X'_n\to\X'_\infty$ in the GHP metric, and if $\X_n\orderGHP\X'_n$ for all $n$, then $\X_\infty\orderGHP\X'_\infty$.
	
	By Proposition~\ref{prop:GHP-common-embedding}, we only need to treat the case where there are two compact metric spaces $(Z,\delta)$ and $(Z',\delta')$ such that $X_\infty$ and $X_n$ are subsets of $(Z,\delta)$ equipped with the induced metric, $X'_\infty$ and $X'_n$ are subsets of $(Z',\delta')$ equipped with the induced metric, and $\dH(X_n,X_\infty)\to 0$, $\dP(\mu_n,\mu_\infty)\to 0$, $\dH'(X'_n,X'_\infty)\to 0$, and $\dP'(\mu'_n,\mu'_\infty)\to 0$.
	
	Let us therefore consider that this is the case, and for every $n\geq1$ assume that we have $\X_n\orderGHP\X'_n$, \textit{i.e.} there is a surjective mapping $\phi_n\colon X'_n\to X_n$, such that
	\begin{align}
		\forall x'_n,y'_n\in X'_n,\ \qquad& \delta(\phi_n(x'_n),\phi_n(y'_n))\leq \delta'(x'_n,y'_n)\label{eq:proof-closed-order:d-incr}\\
		\forall A\in\Bcal_n,\ \qquad&\mu_n(A) \leq \mu'_n(\phi^{-1}_n(A))\label{eq:proof-closed-order:mu-incr},
	\end{align}
	where $\Bcal_n$ is the Borel $\sigma$-algebra of $(X_n,d_n)$.
	We need to prove that $\X_\infty\orderGHP\X'_\infty$.
	
	The condition \eqref{eq:proof-closed-order:d-incr} directly implies that the sequence $(\phi_n)_{n\geq1}$ is equicontinuous,%
	\footnote{
		Here the roles of `` $X$ '' and `` $X'$ '' are reversed compared to how we defined equicontinuity, since we consider `` $X'\to X$ '' mappings.
	}	
	so that by Proposition~\ref{prop:GH-Arzela-Ascoli}, up to taking a subsequence we may assume that $(\phi_n)_{n\geq1}$ converges, to some limit $\phi_\infty\colon X'_\infty\rightarrow X_\infty$.
	The meaning of this convergence, as given by Proposition~\ref{prop:GH-Arzela-Ascoli}, is that $\phi_n(x'_n)\to \phi_\infty(x'_\infty)$ in $(Z,\delta)$ whenever $x'_n\to x'_\infty$ in $(Z',\delta')$.
	
	We claim that $\phi_\infty\colon X'_\infty\to X_\infty$ is surjective.
	Let $x_\infty\in X_\infty$.
	Since $X_n\to X_\infty$ in the Hausdorff metric, we may find elements $x_n\in X_n$, $n\geq1$, such that $x_n\to x_\infty$.
	For $n\geq1$, by surjectivity of $\phi_n$ we may find $x'_n\in X'_n$ such that $\phi_n(x'_n)=x_n$.
	Since $(Z',\delta')$ is compact, after possibly taking a subsequence, we have $x'_n\to x'_\infty$ for some element $x'_\infty\in Z'$, which is actually an element of $X'_\infty$ since $X'_n\to X'_\infty$ in the Hausdorff metric.
	Lastly, $\phi_\infty(x'_\infty)=\lim_n \phi_n(x'_n)=\lim_n x_n=x_\infty$, so that $x_\infty\in\phi_\infty(X'_\infty)$.
	This proves that $\phi_\infty\colon X'_\infty\to X_\infty$ is indeed surjective.
	
	We claim that:
	\begin{align}
		\forall x'_\infty,y'_\infty\in X'_\infty,\ \qquad& \delta(\phi_\infty(x'_\infty),\phi_\infty(y'_\infty))\leq \delta'(x'_\infty,y'_\infty)\label{eq:proof-closed:claim-1}\\
		\forall A\in\Bcal_\infty,\ \qquad&\mu_\infty(A) \leq \mu'_\infty(\phi^{-1}_\infty(A)),
		\label{eq:proof-closed:claim-2}
	\end{align}
	where $\Bcal_\infty$ is the Borel $\sigma$-algebra of $X_\infty$ equipped with the metric induced by $\delta$.
	Once proven, this precisely tells that $\X_\infty\orderGHP\X'_\infty$ and the proof is complete.
	
	Let us prove this claim.
	Since $\phi_n(x'_n)\to \phi_\infty(x'_\infty)$ in $(Z,\delta)$ whenever $x'_n\to x'_\infty$ in $(Z',\delta')$, we obtain \eqref{eq:proof-closed:claim-1} from \eqref{eq:proof-closed-order:d-incr} by taking suitable convergent sequences $x'_n\to x'_\infty$ and $y'_n\to y'_\infty$, as in the proof of the surjectivity of $\phi_\infty$.
	Since $\mu_\infty$ is a finite Borel measure on the compact $X_\infty$, it is inner regular.
	Hence we only need to check \eqref{eq:proof-closed:claim-2} for subsets $A$ of $X_\infty$ which are \textit{compact}. It is sufficient%
		\footnote{Let $A$ be any compact in $X_\infty$ and fix some $x\in X_\infty$. The boundaries in $X_\infty$ of the sets $\bigl(A\cap B(x,r), r>0\bigr)$ are disjoint, and $\mu_\infty$ is a finite measure, so that except in finitely many instances, the boundaries of the sets $\bigl(A\cap B(x,r), r>0\bigr)$ have zero $\mu_\infty$-measure. We can therefore apply the reasoning to them and conclude by continuity of measure.}
	to consider compacts $A$ whose boundary in $X_\infty$ have zero $\mu_\infty$-measure.
	For such an $A$, we have $\mu_n(A)\to \mu_\infty(A)$ by the Portmanteau theorem, while Lemma~\ref{lem:GH-Arzela-Ascoli-strenghened} tells us that $\limsup_{n\to\infty}\mu'_n(\phi^{-1}_n(A))\leq\mu'_\infty(\phi^{-1}_\infty(A))$.
	Therefore, taking the limit superior on both sides of \eqref{eq:proof-closed-order:mu-incr} yields~\eqref{eq:proof-closed:claim-2}.
\end{proof}

\begin{prop}\label{prop:checking-order-is-down-compact}
	The partial order $\orderGHP$ on $(\setGHP,\dGHP)$ is down-compact.
\end{prop}

\begin{proof}
	By Theorem~2.6 in \cite{AbrahamDelmasHoscheit13}, a subset $A\subset\setGHP$ is relatively compact if and only if%
	\footnote{\cite[Thm.~2.6]{AbrahamDelmasHoscheit13} only says that these conditions are sufficient, but it is easily verified that they are necessary.
		For instance, if there exists a sequence of elements of $\setGHP$ violating one of three conditions, then one can extract a converging subsequence and reason as in \cite[Sect.~4.2]{AbrahamDelmasHoscheit13}.}
	the following three conditions hold:
	\begin{enumerate}
		\item
		we have $\sup_{(X,d,\mu)\in A}\,\diam(X)<\infty$;
		\item
		for every $\epsilon>0$, there exists a finite integer $N_\epsilon\geq1$ such that for $(X,d,\mu)\in A$, there is a covering of $X$ by at most $N_\epsilon$ balls with radius less than $\epsilon$; and,
		\item
		we have $\sup_{(X,d,\mu)\in A}\,\mu(X)<\infty$.
	\end{enumerate}
	The crucial observation is that if $A\subset\setGHP$ satisfies the above conditions, then so does $\down_{\orderGHP}(A)=\{\X\in\setGHP\colon\exists\X^+\in A,\,\X\orderGHP\X'\}$, as is clear from the definition of the order $\orderGHP$.
	In particular, if $K^+\subset\setGHP$ is compact, then $\down_\orderGHP(K^+)$ is relatively compact in $\setGHP$; and using Proposition~\ref{prop:checking-order-is-closed}, it is also closed in $\setGHP$.
	Hence, $\down_\orderGHP(K^+)$ is compact in $\setGHP$, as needed.
\end{proof}

We conclude this section with the following immediate yet useful lemma.

\begin{lemma}\label{lem:order-vs-scaling}
	Let $a,b>0$.
	\begin{enumerate}
		\item We have $(a,b)\cdot\X\orderGHP(a,b)\cdot\X'$ whenever $\X\orderGHP\X'$ in $\setGHP$.
		\item We have $(a,b)_*\P(\diff\X)\storder(a,b)_*\P'(\diff\X)$ whenever $\P(\diff\X)\storder\P'(\diff\X)$ in $\mes_1(\setGHP)$.
	\end{enumerate}
\end{lemma}

\subsection{Two lemmas on sandwich theorems}

\begin{lemma}\label{lem:general-sandwich-theorem}
	If $\Xtop$ is a metric space and $\preceq$ is a closed down-compact partial order on it, then we have the following sandwich theorem:
	\begin{align}\label{eq:lemma-general-sandwich-theorem}
		\begin{cases}
			\forall n,\quad\xtop^-_n\preceq\xtop_n\preceq\xtop^+_n,\\
			\xtop^-_n\to\xtop_\infty,\\ \xtop^+_n\to\xtop_\infty,
		\end{cases}
		\qquad\implies\qquad
		\xtop_n\to\xtop_\infty,
	\end{align}
	for all sequences $(\xtop^-_n)_{n\geq1},(\xtop_n)_{n\geq1},(\xtop^+_n)_{n\geq1}$ of elements of $\Xtop$, and all $\xtop_\infty\in\Xtop$.
\end{lemma}

\begin{proof}
	It is sufficient to prove that every subsequence of $(\xtop_n)_{n\geq1}$ admits a subsequence converging to $\xtop_\infty$.
	Since $\xtop^+_n\to\xtop_\infty$, the subset $K^+=\{\xtop^+_1,\xtop^+_2,\dots\}\cup\{\xtop_\infty\}$ of $\Xtop$ is compact, and thus $\down_\preceq(K^+)$ is compact too, using the down-compactness of $\preceq$.
	But $(\xtop_n)_{n\geq1}$ takes values in $\down_\preceq(K^+)$, so that every subsequence of $(\xtop_n)_{n\geq1}$ admits a converging subsequence, say with limit $\xtop'_\infty$.
	It remains to prove that necessarily $\xtop'_\infty=\xtop_\infty$.
	But since $\preceq$ is closed we must have $\xtop_\infty\preceq\xtop'_\infty\preceq\xtop_\infty$ and therefore $\xtop'_\infty=\xtop_\infty$ since $\preceq$ is a partial order.
\end{proof}

\begin{lemma}\label{lem:lifting-closed-and-down-compact}
	If $\Xtop$ is a Polish space and $\preceq$ is a closed down-compact partial order on it, then $\mes_1(\Xtop)$ is also Polish and $\stpreceq$ is also a closed down-compact partial order.
\end{lemma}

\begin{proof}
	First, if $\Xtop$ is Polish and $\preceq$ is a closed partial order, then $\mes_1(\Xtop)$ is also Polish and $\stpreceq$ is also a closed partial order by~\cite[Thm.~2]{KamaeKrengel78}.
	Now suppose that $\preceq$ is down-compact, and let us prove that $\stpreceq$ is also down-compact.
	
	Let $\boldK^+$ be an arbitrary compact subset of $\mes_1(\Xtop)$, and fix some $\epsilon>0$.
	By Prokhorov's theorem, $\boldK^+$ is a tight collection in $\mes_1(\Xtop)$, so that there exists some compact subset $K^+_\epsilon\subset\Xtop$ such that
	\begin{align}\label{eq:tightness-proof-lifting-property}
		\sup_{\P^+(\diff\xtop)\in\boldK^+}\P^+(\Xtop\setminus K^+_\epsilon)\leq \epsilon.
	\end{align}
	Let $\P(\diff\xtop)\in\down_{\stpreceq}(\boldK^+)$, which means by definition that $\P(\diff\xtop)\stpreceq\P^+(\diff\xtop)$ for some $\P^+(\diff\xtop)\in\boldK^+$.
	Since $\preceq$ is down-compact, we have that $D_\epsilon=\down_\preceq(K^+_\epsilon)$ is a \textit{compact} subset of $\Xtop$.
	By construction, $\Xtop\setminus D_\epsilon$ is an $\preceq$-increasing Borel event on $\Xtop$.
	Since $\P(\diff\xtop)\stpreceq\P^+(\diff\xtop)$, this implies:
	\begin{align*}
		\P(\Xtop\setminus D_\epsilon)
		\leq \P^+(\Xtop\setminus D_\epsilon)
		\leq \P^+(\Xtop\setminus K^+_\epsilon)
		\leq\epsilon,
	\end{align*}
	where the second inequality comes from the inclusion $K^+_\epsilon\subset D_\epsilon$, and the last one uses~\eqref{eq:tightness-proof-lifting-property}.
	Hence we just proved that for every $\epsilon>0$ we can find a compact subset $D_\epsilon\subset\Xtop$ such that
	\begin{align*}
		\sup_{\P(\diff\xtop)\in\down_{\stpreceq}({\boldK^+})}\P(X\setminus D_\epsilon)\leq \epsilon.
	\end{align*}
	That is, we proved that $\down_{\stpreceq}({\boldK^+})$ is a tight collection in $\mes_1(\Xtop)$, and therefore it is pre-compact by Prokhorov's theorem.
	On top of being pre-compact, $\down_{\stpreceq}(\boldK^+)$ is also closed (and hence compact) in $\mes_1(\Xtop)$.
	This is easily seen from the definition of $\down_{\stpreceq}$, using that $\stpreceq$ is a closed partial order and that $\boldK^+$ is closed in $\mes_1(\Xtop)$.
	
	Wrapping up, we proved that for every compact subset $\boldK^+\subset\mes_1(\Xtop)$, the set $\down_{\stpreceq}(\boldK^+)$ is also a compact subset of $\mes_1(\Xtop)$.
	Hence, $\stpreceq$ is down-compact, as needed.
\end{proof}

\subsection{Proofs of the GHP sandwich and Tauberian theorems}

\begin{proof}[Proof of Theorem \ref{thm:sandwich-thm-GHP}]
	The binary relation $\orderGHP$ is a partial order on $\setGHP$ by Proposition~\ref{prop:checking-partial-order}, it is closed by Proposition~\ref{prop:checking-order-is-closed}, and it is down-compact by Proposition~\ref{prop:checking-order-is-down-compact}.
	We conclude with Lemma~\ref{lem:general-sandwich-theorem} that the claimed sandwich theorem is satisfied.
\end{proof}

\begin{proof}[Proof of Theorem \ref{thm:sandwich-thm-GHP-random}]
	Since $(\setGHP,\dGHP)$ is Polish~\cite[Thm.~2.5]{AbrahamDelmasHoscheit13}, and $\orderGHP$ is a closed down-compact partial order on it by Theorem~\ref{thm:sandwich-thm-GHP}, we deduce using Lemma~\ref{lem:general-sandwich-theorem} that $\mes_1(\setGHP)$ is also Polish and $\storder$ is also a closed down-compact partial order.
	Again, we conclude with Lemma~\ref{lem:general-sandwich-theorem} that the claimed sandwich theorem is satisfied.
\end{proof}

\begin{proof}[Proof of Theorem \ref{thm:GHP-Tauberian-thm}]

We recall the setting of Theorem~\ref{thm:GHP-Tauberian-thm}: we have random compact metric measure spaces $\left(\Xf_n, n\geq1\right)$ and $\Xf_\infty$, two tame sequences $(a_n)_{n\geq1}$, $(b_n)_{n\geq1}$, and random variables $(N_n)_{n\geq1}$ independent of $(\Xf_n)_{n\geq1}$ and $\Xf_\infty$ such that:
\begin{enumerate}
	\item
	There holds $(a_n,b_n)\cdot\Xf_ {N_n}\to\Xf_\infty$ in distribution for the GHP topology.
	\item
	There holds $N_n/n\rightarrow 1$ in probability.
	\item
	We have $\P_1(\diff\X)\storder\P_2(\diff\X)\storder\P_3(\diff\X)\storder\dots$, where $\P_n(\diff\X)$ is the distribution of $\Xf_n$ for every $n\geq1$.
\end{enumerate}

Since $N_n/n\rightarrow 1$ in probability, we can find%
\footnote{
	For instance, set $k_n=\lfloor{n\epsilon_n}\rfloor$ where $\epsilon_n$ is the infimal $\epsilon$ such that the probability that $N_{n-\lfloor{n\epsilon}\rfloor}<n<N_{n+\lfloor{n\epsilon}\rfloor}$ is at least $1-\epsilon$.
}
positive integers $k_n=o(n)$ such that $\Prob{N_{n-k_n}< n}\to 1$ and $\Prob{N_{n+k_n}>n}\to 1$.
In particular, the random variables $L^-_n=\min(n, N_{n-k_n})$ and $L^+_n=\max(n, N_{n+k_n})$ are at vanishing total variation distance from $ N_{n-k_n}$ and $ N_{n+k_n}$ respectively.
Hence we can replace $N_{n\pm k_n}$ by $L^\pm_n$ in the convergence $(a_{n\pm k_n},b_{n\pm k_n})\cdot\Xf_{N_{n\pm k_n}}\to\Xf_\infty$ which follows from (1) above: this gives $(a_{n\pm k_n},b_{n\pm k_n})\cdot\Xf_ {L^\pm_n}\to\Xf_\infty$ after replacement.
Now, since $(a_n)_{n\geq1}$ and $(b_n)_{n\geq1}$ are tame and $k_n=o(n)$, we have $a_{n\pm k_n}\sim a_n$ and $b_{n\pm k_n}\sim b_n$, so that by Corollary~\ref{cor:continuity-rescaling-random}, we deduce that $(a_n,b_n)\cdot\Xf_ {L^\pm_n}\to \Xf_\infty$ in distribution for the GHP topology.
Hence we have shown that
\begin{align}\label{eq:proof-application-conditioning-1}
	(a_n,b_n)_*\P^\pm_n(\diff\X)\xRightarrow[n\rightarrow\infty]{\mathrm{GHP}}\P_\infty(\diff\X),
\end{align}
where $\P^\pm_n(\diff\X)$ is the distribution of $\Xf_ {L^\pm_n}$ for every $n\geq1$.
For all $n\geq1$, the distribution $\P^-_n(\diff\X)$ is a mixture of $(\P_k(\diff\X),k\leq n)$ while  $\P^+_n(\diff\X)$ is a mixture of $(\P_k(\diff\X),k\geq n)$.
Hence (3) above gives:
\begin{align*}
	\P^-_n(\diff\X)\storder \P_n(\diff\X)\storder\P^+_n(\diff\X).
\end{align*}
Since the order $\storder$ is scale-invariant in the sense of Lemma~\ref{lem:order-vs-scaling}, we get
\begin{align}\label{eq:proof-application-conditioning-2}
	(a_n,b_n)_*\P^-_n(\diff\X)
	\storder (a_n,b_n)_*\P_n(\diff\X)
	\storder (a_n,b_n)_*\P^+_n(\diff\X).
\end{align}
Finally, from \eqref{eq:proof-application-conditioning-1} and \eqref{eq:proof-application-conditioning-2}, the sandwich theorem Theorem~\ref{thm:sandwich-thm-GHP-random} yields that:
\begin{align*}
	(a_n,b_n)_* \P_n(\diff\X)\xRightarrow[n\rightarrow\infty]{\mathrm{GHP}}\P_\infty(\diff\X),
\end{align*}
which means that $(a_n,b_n)\cdot\Xf_n\to\Xf_\infty$ in distribution for the GHP topology.
\end{proof}

This concludes the present section regarding the partial order $\orderGHP$. We have established in particular the GHP Tauberian theorem, Theorem~\ref{thm:GHP-Tauberian-thm}, which serves as the main tool behind the proof of Theorem~\ref{thm:scaling-limit-irred} on the scaling limit of irreducible quadranagulations. The proof of the latter theorem was given in Section~\ref{sec:structure-proof} up to some some key lemmas, which we now address. In the next section, we begin with the first step: finding a large \textit{irreducible component} in a large \textit{general quadrangulation}, in order to prove Lemma~\ref{lem:summary-largest-comp}.

% -----------------------------------------------------------------
%							SECTION 4
% -----------------------------------------------------------------

\section{Step I: The largest irreducible component}
\label{sec:largest-components}

Given some combinatorial object or algebraic structure, there are often two dual ways to make sense of what ``an object found in another object'' means, namely by considering either \textit{sub-objects} or \textit{quotients}.

Informally, for us a ``sub-object'' of a planar map $\m$ will correspond to a ``map embedded in $\m$'', that is a \textit{submap} as defined in Section~\ref{subsec:prelim-submaps}; and a ``quotient'' of a planar map $\m$ will be any planar map that we can obtain by removing/collapsing some parts of $\m$.
This will give us two notions of what the ``largest irreducible quadrangulation'' found in some quadrangulation $\q$ means:
\begin{enumerate}
	\item
	 The largest \textit{irreducible component}, where an irreducible component is a submap which is irreducible and satisfies an additional condition which guarantees uniqueness properties.
	 \item
	 The largest \textit{irreducible block}, where an irreducible block is obtained informally by collapsing the sides of the 2-cycles which do not contain some fixed edge, and then emptying the ``sides of the 4-cycles'' which do not contain that edge.
\end{enumerate}

Definitions vary among authors, and the words ``component'' and ``block'' are often used interchangeably, mainly to describe what we call a ``block''.
Sometimes there is a clear choice of embedding which turns the block into an actual submap (for instance 2-connected blocks in general maps).
One of the aims of this section is to clarify the links between the two notions, explaining in particular why the \textit{blocks} considered in the literature can be canonically realized as submaps, which is convenient for metric comparisons between a quadrangulation and its largest irreducible block.

After some preliminaries, we define the notion of irreducible \textit{components} we will consider.
Then, we show that it leads to the same objects as the notion of \textit{blocks} mentioned above.
This will allow us to use results of Gao and Wormald~\cite{GaoWormald99} on block-sizes, and in particular to deduce a \textit{condensation} phenomenon: asymptotically, in a uniformly random quadrangulation with $n$ faces, there is a unique largest component with approximately $n/9$ faces.
Lastly, we complement this result with some distributional properties of the largest irreducible component.

\subsection{Preliminaries on maps and submaps}
\label{subsec:prelim-submaps}

Let us first recall some definitions from the introduction. A \textit{planar map} is a proper embedding in the two-dimensional sphere, without edge-crossings, of a finite connected planar graph with a distinguished oriented edge. Loops and multiple edges are allowed and two planar maps are identified if there is an orientation-preserving homeomorphism of the sphere that maps one onto the other while matching their root-edges.

We recall that a \textit{quadrangulation} is a planar map in which all faces have degree four. It is \textit{simple} if it has no 2-cycle, and it is \textit{irreducible} if it is simple, has at least four faces, is such that all 4-cycles bound a face.

\subsubsection{Sides of a simple cycle}

We will repeatedly speak about the ``sides'' of a simple cycle, so let us give a formal definition once and for all.
In a planar map, a $k$-cycle (reps.~oriented $k$-cycle) is a path (resp.~oriented path) of length $k$ which ends at its starting vertex, and it is said to be a \textit{simple cycle} if it uses any other vertex at most once.

Consider some planar $\m$, and let us fix an embedding of it in the 2-sphere.
Under this embedding, every simple cycle $C$ in $\m$ corresponds to a Jordan curve, separating the 2-sphere into two simply connected regions whose closures are, say, $R_1$ and $R_2$ respectively.
We call these the \textit{sides} of $C$.
We then say that two faces/edges/vertices/paths/\textellipsis of $\m$ are \textit{on the same side} of $C$ if the embedding sends them both in $R_1$ or both in $R_2$.
In particular, the only edges and vertices which are on both sides are the ones which belong to $C$.
Note that the relation of ``being on the same side'' does not actually depend on the choice of embedding (in its orientation-preserving homeomorphism class).

Given an \textit{oriented} simple cycle $\vec C$, its \textit{interior side}, or simply its \textit{interior}, is the side of $\vec C$ around which $\vec C$ runs counter-clockwise.
The other side is called the \textit{exterior side}, or simply the \textit{exterior} of $\vec C$.
By convention, the \textit{oriented boundary} $\orBdry f$ of a face $f$ with simple boundary in some map $\m$ is the oriented cycle whose interior is that face.

\subsubsection{Submaps}

A submap of a planar map $\m$ is a pair $(\sm,\alpha)$ where $\sm$ is a rooted planar map and $\alpha\colon V(\sm)\sqcup E(\sm)\to V(\m)\sqcup E(\m)$ is an injective mapping sending vertices/edges of $\sm$ to vertices/edges of $\m$ respectively, in such a way that any embedding of $\m$ in the sphere induces through $\alpha$ a proper embedding of the underlying graph of $\sm$ which is precisely in the orientation-preserving homeomorphism class characterizing $\sm$.
Informally, this just means that $\alpha$ describes a graph homomorphism which is compatible with the embeddings prescribed by $\sm$ and $\m$.
In a slightly abusive way, we will often drop mention of the mapping $\alpha$ and identify the vertices/edges of $\sm$ with the vertices/edges of $\m$ they are mapped to \textit{via} the mapping $\alpha$.

We will say that two submaps $\sm$ and $\sm'$ of a planar map $\m$ are \textit{equivalent up to re-rooting}, or \textit{re-rooting equivalent}, if we have $E(\sm)=E(\sm')$ as subsets of $E(\m)$.
Note however that, in this case, the oriented root edges of $\sm$ and $\sm'$---seen as oriented edges of $\m$---might be distinct.

\subsubsection{Detaching and gluing}

Suppose we are given some algorithm $\psi$ which to a map $\m$ and a face $f$ on it associates an oriented edge $\psi(\m;f)$ on the oriented cycle $\orBdry f$.
It will be implicit in notation.
Let $(\sm,\alpha)$ be a submap of a map $\m$.
Suppose that $\sm$ has $\ell\geq2$ faces, all of them with a \textit{simple boundary}, and enumerate them as $(f_1,\dots, f_\ell)$ using the order $<_{\sm}$.
For $i=1,\dots,\ell$, the interior of $\orBdry f_i$ can be rooted at $\psi(\sm;f_i)$.
This yields, using any proper embedding of $\m$, a rooted planar map $\m_i=\m_i(\m;\sm,\alpha)$.
We can thus form the collection $\Decs=\Decs(\m;\sm,\alpha)$ of rooted maps by setting $\Decs=(\m_i)_{1\leq i\leq \ell}$.
Lastly, we let $\ifrak=\ifrak(\m;\sm,\alpha)\in\{1,\dots,4|\m|\}$ be the index of the root edge of $\m$ in the order $<_{\m'}$, where $\m'$ is $\m$ re-rooted at the root edge of $\sm$.
All in all, we have defined a triple $(\sm,\Decs,\ifrak)$, which is a function of $(\m,\sm,\alpha)$.
Hence we may define:
\begin{align*}
\Detach(\m;\sm,\alpha)=(\sm,\Decs,\ifrak).
\end{align*}
There is an inverse operation.
Consider given a planar map $\sm$ whose faces $(f_1,\dots, f_\ell)$---in the order $<_{\sm}$---have a simple boundary, a collection $\Decs=(\m_i)_{1\leq i\leq \ell}$ of rooted planar maps whose root faces have a simple boundary of length equal to that of $\orBdry f_i$ respectively, and an integer $1\leq \ifrak\leq 4\sum_{i}(|\m_i|-1)$.
Then, one easily reverses the operation described above, so as to ``glue back'' the collection of maps of $\Decs$ in order to realize a rooted map $\m'$ and a submap $(\sm,\alpha)$ satisfying that $\Detach(\m;\sm,\alpha)=(\sm,\Decs,\ifrak)$.
We write:
\begin{align*}
\Glue(\sm;\Decs,\ifrak)&=(\m,\sm)\\
\Gamma(\sm;\Decs,\ifrak)&=\m.
\end{align*}

\begin{rem}
	The ``detaching'' and ``gluing'' operations can still be defined if we relax the assumption that the submap $\sm$ we consider is such that all its faces have a simple boundary.
	This requires to extend the notion of ``sides of a cycle'' to cycles which are only non-crossing.
	However, we will not need this level of generality.
\end{rem}

\subsection{Irreducible submaps and components} 

Irreducibility (as defined in the introduction) is in some sense a strengthening of the notion of ``connectedness'': informally, a quadrangulation is irreducible if it not possible to partition it non-trivially into two pieces whose intersection is a 4-cycle; whereas a space is connected if it cannot be non-trivially partitioned into two pieces whose intersection is the empty set.
Hence, just as connected components are maximal connected subsets in a topological space, it is to be expected that there is a well-defined notion of ``irreducible components''.
It is tempting to declare that an irreducible component is just a maximal irreducible submap, with respect to inclusion.
But in order to enforce desirable uniqueness properties, we will need to consider ``rightmost'' irreducible submaps, as we now define.

\subsubsection{Canonical orientation and rightmost cycles}

Let $\q$ be a quadrangulation.
Since its faces has even degrees, it is bipartite, that is it admits a \textit{proper bi-coloring}: a coloring of each vertex with black or white such that no two neighbor vertices have the same color.
If we require the root oriented edge of $\q$ to be oriented black-to-white, then there is a unique such proper bi-coloring; which is easily constructed by a greedy exploration starting from the root edge.
The obtained bi-coloring of $\q$ will be called the \textit{canonical bi-coloring}, and when there is no confusion possible, we will simply speak about \textit{black} and \textit{white} vertices of $\q$ in reference to their color under the canonical bi-coloring.
The black-to-white orientation $\vec e$ of a non-oriented edge $e$ of $\q$ will be called the \textit{canonical orientation} of $e$.

Let $C$ be a simple cycle on $\q$.
We will say that it is a \textit{rightmost cycle} if the following holds: for every edge $e$ on $C$, say between some black vertex $v_\bullet$ and some white vertex $v_\circ$, all the edges joining $v_\bullet$ to $v_\circ$ belong to the side of $C$ which sits to the left of $e$, when oriented in the canonical black-to-white orientation.
Accordingly, a submap $\sm$ of the quadrangulation $\q$ is said to be a \textit{rightmost submap} if its faces are bordered by cycles which are rightmost in $\q$.
Note that this notion depends on $\sm$ only through its edge-set, so that $\sm$ is rightmost if and only if any of the submaps of $\q$ which are re-rooting equivalent to $\sm$ is rightmost.

\subsubsection{Irreducible components}

We are now equipped to define irreducible components, as well as component-rooted quadrangulations.

\begin{defin}\label{def:irred-components}
	An \textit{irreducible component} of a quadrangulation $\q$ is a rightmost submap $(\sm,\alpha)$ corresponding to an irreducible quadrangulation $\sm$.
	It is automatically maximal with respect to the submap relation.%
		\footnote{
			Indeed, if it is a submap of another irreducible quadrangulation $\sm'$, then the borders of the faces of $\sm$ cannot be separating in $\sm'$, so that they also bound a face of $\sm'$.
		}
\end{defin}
\noindent
We recall that we made the convention that irreducible quadrangulations must have at least $4$ faces.
In particular a quadrangulation must have at least $4$ faces to possess an irreducible component.

\begin{rem}\label{rem:submaps-and-rerooting}
	Let $\q$ be a quadrangulation, and let $\q'$ be the same quadrangulation obtained by re-rooting at some edge.
	Let us justify that the collections of irreducible components of $\q$ and $\q'$ coincide, when these components are viewed as \textit{maps} $\sm$, even though they may differ as \textit{submaps} $(\sm,\alpha)$.
	
	Indeed, either $\q$ and $\q'$ have the same canonical bi-coloring, in which case the irreducible components of $\q$ and of $\q'$ are identified; or the bi-coloring of $\q'$ is obtained from that of $\q$ by switching colors.
	In the latter case, this reverses the canonical orientation, so that irreducible components of $\q$ are identified with ``leftmost'' irreducible submaps of $\q'$ instead.
	These are in bijection with irreducible components of $\q'$ by taking their ``rightmost version'', that is by replacing each of their edges by the rightmost edge of $\q'$ with same endpoints.
	Clearly, under this bijection, the irreducible \textit{submaps} of $\q$ and $\q'$ which are identified correspond to isomorphic \textit{maps}. 
\end{rem}

The following lemma is quite intuitive and it will be essential in cases where we want to say that ``there is a unique largest component''.

\begin{lemma}[Separation lemma]\label{lem:sepration-lemma}
	Let $\q$ be a quadrangulation and $\sm_1$, $\sm_2$ be two of its irreducible components.
	If $\sm_1$ is not \emph{re-rooting equivalent} to $\sm_2$, then there exists a simple 4-cycle which separates $\sm_1$ and $\sm_2$, that is a simple 4-cycle $C$ such that $\sm_1$ is on one side of $C$ and $\sm_2$ on the other.
\end{lemma}

We will need the following intermediate lemma.

\begin{lemma}\label{lem:irred-comp-vs-cycles}
	Let $\q$ be a quadrangulation and let $\sm$ be one of its irreducible components.
	For every oriented rightmost simple 4-cycle $\vec C$ on $\q$, either $\sm$ in in the interior of $\vec C$, or it is in its exterior.
\end{lemma}

\begin{proof}[Proof of Lemma~\ref{lem:irred-comp-vs-cycles}]
	Suppose for the sake of contradiction that $\sm$ intersects strictly both the interior of $\vec C$ and its exterior.
	It is well-known that irreducible quadrangulations are 3-connected, that is that one needs to remove at least 3 vertices to disconnect them.
	In particular, $\sm$ must contain at least three of the vertices on the cycle $\vec C$.
	Let us therefore say that $\vec C$ visits vertices $v_1$, $v_2$, $v_3$, $v_4$ in that order, where $\sm$ contains $v_1$, $v_2$, $v_3$, and possibly $v_4$.
	We write $e_1$ and $e_2$ for the edges of $\vec C$ joining $v_1$ to $v_2$, resp.~$v_2$ to $v_3$.
	
	We claim that $e_1$ and $e_2$ both belong to $E(\sm)$.
	Suppose that is not the case, say $e_1\notin E(\sm)$.
	Then, $e_1$ is a path in the interior of some face $f$ of $\sm$, between the vertices $v_1$ and $v_2$, which are on the boundary of $f$.
	But $\q$ being bipartite implies that $v_1$ and $v_2$ are of different colors, which implies that they are actually adjacent vertices on $\orBdry f$ since it is a 4-cycle.
	All in all, we deduce that there is an edge $e'_1\in E(\sm)$ between $v_1$ and $v_2$.
	Now, since $\vec C$ is a rightmost cycle and $\sm$ is a rightmost submap, we deduce that $e_1$ and $e'_1$ are actually both the rightmost edge between $v_1$ and $v_2$.
	Hence $e_1\in E(\sm)$, which gives a contradiction and proves the claim.
	
	Now observe that $v_4$ is a vertex of $\sm$.
	Otherwise, the edges $e_1$ and $e_2$, together with the boundary $\orBdry f'$ of the face $f'$ containing $v_4$, would create a ``4-cycle with a diagonal'', which cannot happen in an irreducible quadrangulation.%
		\footnote{
			This uses our convention that irreducible quadrangulations must have at least 4 faces.
			Otherwise, the unique irreducible quadrangulation with 3 faces possesses a ``4-cycle with a diagonal''.
		}
	Hence, $v_4$ is a vertex of $\sm$, so that re-using the arguments at the beginning of this proof, the edges of $\vec C$ actually all belong to $\sm$.
	Since $\sm$ intersects strictly both the interior of $\vec C$ and its exterior, we obtain that $\vec C$ is a separating 4-cycle in $\sm$, which contradicts the irreducibility of $\sm$.
	This concludes our proof by contradiction.
\end{proof}

\begin{proof}[Proof of Lemma~\ref{lem:sepration-lemma}]
	Let $\q$ be a quadrangulation, let $\sm_1$, $\sm_2$ be two of its irreducible components and suppose that $\sm_1$ and $\sm_2$ are not \textit{re-rooting equivalent}, that is $E(\sm_1)\neq E(\sm_2)$.
	In particular, up to exchanging $\sm_1$ and $\sm_2$, there is some edge $e_1\in E(\sm_1)\setminus E( \sm_2)$.
	Hence, under a proper embedding of $\q$ in the 2-sphere, the edge $e_1$, seen as a path on the sphere, must belong---except maybe at its endpoints---to the strict interior of some face $f_2$ of $\sm_2$.
	Since $\sm_2$ is a rightmost and simple submap, the cycle $\orBdry f_2$ is simple and rightmost.
	Hence it follows by Lemma~\ref{lem:irred-comp-vs-cycles} and the fact that $\sm_1$ intersects the strict interior of $\orBdry f_2$ that it cannot also intersect its strict exterior.
	But $\sm_2$ is on the exterior of $\orBdry f_2$ since $f_2$ is one of its faces.
	Hence $\orBdry f_2$ separates $\sm_1$ and $\sm_2$.
\end{proof}

\subsubsection{Reconstruction}

We will use the ``detach'' and ``glue'' operations from Section~\ref{subsec:prelim-submaps}.
Once and for all, we choose an algorithm $\psi$ which to a quadrangulation $\q$ and a face $f$ on it associates an oriented edge $\psi(\q;f)$ on the oriented cycle $\orBdry f$, \textbf{which we furthermore require to be white-to-black}.

For $4\leq \ell\leq n$, we let $\CompRootedQ{n}{\ell}$ denote the set of tuples $(\q;\sm,\alpha)$ where $\q$ is a quadrangulation with $n$ faces, and $(\sm,\alpha)$ is an irreducible component of $\q$ with $\ell$ faces.
We wish to describe the image of $\CompRootedQ{n}{\ell}$ under the mapping $\Detach$ from Section~\ref{subsec:prelim-submaps}, with the above choice of algorithm $\psi$.

To this end, we let $\widetilde\Quads$ be the set of quadrangulations $\q$ such that both the root edge and its opposite edge in the root face are \textit{solo edges}---we call \textit{solo edge} an edge such that no other edge has the same endpoints.
Then, for $4\leq \ell\leq n$, we let $\DecSet_{n,\ell}$ be the set of tuples $(\q_1,\dots,\q_\ell)$ such that each $\q_i$ belongs to $\widetilde\Quads$ and such that $\sum_i|\q_i|=n+\ell$.

\begin{lemma}\label{lem:bij-psi-irred}
	For $4\leq \ell\leq n$, the mapping $\Detach$ induces a bijection
	\begin{align*}
	\Detach\colon\CompRootedQ{n}{\ell}\longrightarrow \ensQirred_\ell \times \DecSet_{n,\ell}\times \{1,\dots,4n\},
	\end{align*}
	whose inverse is given by $\Glue$.
\end{lemma}

\begin{proof}
	First, it is justified in Section~\ref{subsec:prelim-submaps} that $\Detach$ is bijective with the domain of definition and image that are given there.
	Hence to prove the lemma, we only need to prove that the image of  $\CompRootedQ{n}{\ell}$ under $\Detach$ is the set $\ensQirred_\ell \times \DecSet_{n,\ell}\times \{1,\dots,4n\}$.
	Fix $4\leq \ell\leq n$ and let $(\q,\sm,\alpha)\in \smash{\CompRootedQ{n}{\ell}}$.
	We form $\Detach(\q;\sm,\alpha)=(\sm,\Decs=(\q_i)_{1\leq i\leq \ell},\ifrak)$.
	By definition of irreducible components, we have $\sm\in\ensQirred_\ell$.
	We also need to justify that $\Decs\in\DecSet_{n,\ell}$.
	By construction, we have $|\q|= \sum_{1\leq i\leq \ell}(|\q_i|-1)$, that is $\sum_i|\q_i|=n+\ell$.
	On the other hand, recall that irreducible components are required to be rightmost.
	From the definition of rightmost submaps, it follows easily that, on the (counter-clockwise) oriented boundary $\orBdry f$ of a face $f$ of $\sm$, the white-to-black edges are such that no other edge with same endpoints visits the interior of $f$---that is they induce \textit{solo edges} in the map that is obtained from the ``detaching operation'' applied to the face $f$.
	With the requirement we made that the rooting algorithm $\psi$ chooses a white-to-right edge to root the $(\q_i)_{1\leq i\leq \ell}$, we therefore deduce that each $\q_i$ is in fact in $\widetilde\Quads$.
	We have therefore verified that $(\sm,(\q_i)_{1\leq i\leq \ell},\ifrak)$ belongs to $\ensQirred_\ell \times \DecSet_{n,\ell}\times \{1,\dots,4n\}$.
	Similar considerations show that the elements of this set are sent into $\smash{\CompRootedQ{n}{\ell}}$ by the mapping $\Glue$.
	This concludes the proof.
\end{proof}

\subsection{Blocks and largest components}

In \cite{GaoWormald99}, Gao and Wormald studied several notions of ``components''---which we shall call \textit{blocks} here---and proved asymptotic results for the size of the largest such block in several models.
Roughly, in their approach, it is shown that a map $\m$ from some ensemble $\Acal$ can be built from a ``root-block'' $\b$ belonging to some ensemble $\Bcal$ in which we have ``substituted'' maps from $\Acal$ in place of the edges/faces/\textellipsis, and where $\m$ inherits the root edge of $\b$.
The blocks of $\m$ are then all the root-blocks that can be obtained upon starting with an other rooting of $\m$.
Let us describe how one defines the \textit{simple root-block of a general quadrangulation}, as well as the \textit{irreducible root-block of a simple quadrangulation}.

\subsubsection{Simple root-block}

The following construction is dual, via Tutte's ``angular'' bijection, to the decomposition of general maps as a collection of a map as a ``2-connected root-block'', in the corners of which some maps are grafted.
See the original work of Tutte \cite{Tutte63} for this decomposition, and \cite{FleuratSalvy24} for the equivalence with the construction below.

Given a \textit{general} quadrangulation $\q$, its \textit{simple root-block} $\SComp(\q)$ is obtained as follows: (i) embed $\q$ in the plane in such a way that the unbounded face lies to the right of the oriented root edge; and, (ii) for each maximal%
\footnote{
	A 2-cycle is \textit{maximal} under this embedding in the plane if it is not contained in the bounded side of another 2-cycle.
}
2-cycle in $\q$, empty its bounded side and collapse the 2-gonal face thus obtained into a single edge.
Note that root edge may have been ``fused'' with another edge in the process.

\subsubsection{Irreducible root-block}

We follow the approach of Mullin and Schellenberg \cite{MullinSchellenberg68}.
Given a \textit{simple} quadrangulation $\qsim$, its \textit{irreducible root-block} $\IrrComp(\qsim)$ is obtained as follows.
Embed $\qsim$ in the plane in such a way that the unbounded face lies to the right of the oriented root edge, or equivalently such that the oriented boundary $\orBdry f_0$ of the root face $f_0$ runs clockwise.
We then distinguish the two following cases:
\begin{enumerate}
	\item
		If there is a ``diagonal'', that is a path of length two in the strict interior of $\orBdry f_0$ joining two antipodal vertices of $\orBdry f_0$, then $\qsim$ will be said to have no root-block, which we write symbolically $\IrrComp(\qsim)=\emptyset$.
	\item
		Otherwise, the irreducible root-block $\IrrComp(\qsim)$ is obtained by emptying the bounded side of each maximal%
		\footnote{
			A 4-cycle $C$ is \textit{maximal} if it is not the boundary of the root face and if it is not contained in the bounded side of another such 4-cycle.
			In order to empty maximal 4-cycles, we need their bounded sizes to be non-intersecting.
			This is the case when there is no ``diagonal'', as the reader may verify, see \cite{MullinSchellenberg68} for an extensive discussion.
		}
		4-cycle of $\q$.
\end{enumerate}

\subsubsection{Relation with Definition~\ref{def:irred-components}}

The following lemma verifies that Definition~\ref{def:irred-components} is consistent with the above notions of blocks.

\begin{lemma}\label{lem:equiv-two-defs}
	Let $\q$ be a quadrangulation.
	\begin{enumerate}
		\item 
			For every irreducible component $(\sm,\alpha)$ of $\q$, if we let $\q'$ be $\q$ re-rooted at the oriented root edge of $\sm$, then we have $\sm=\IrrComp(\SComp(\q'))$.
		\item
			For every $\q'$ obtained by re-rooting $\q$ at some oriented edge, if $\IrrComp(\SComp(\q'))\neq\emptyset$, then $\IrrComp(\SComp(\q'))$ can be realized as an irreducible component $(\sm,\alpha)$ of $\q$.
	\end{enumerate}
\end{lemma}

\begin{proof}
	Let us begin with the first statement.
	Let $\sm$ be an irreducible component of $\q$, and let $\q'$ be $\q$ re-rooted at the oriented root edge of $\sm$.
	Since $\sm$ is simple, all the 2-cycles of $\q'$ are inside some face of $\sm$, and each one can only be incident to at most one edge of $\sm$.
	In particular, the operation of collapsing maximal 2-cycles in $\q'$ does not change $\sm$ as a map, but it realizes it as a submap of $\SComp(\q')$, where again $\sm$ has the same root edge as $\SComp(\q')$.
	Now, $\sm$ being irreducible implies that its faces are bordered by maximal 4-cycles of $\SComp(\q')$; and, using the ``separation'' Lemma~\ref{lem:sepration-lemma}, there can be no other separating 4-cycles.%
		\footnote{
			Please note that any cycle is trivially rightmost in a simple quadrangulation, so that we can apply the separation lemma.
		}
	In particular, emptying the maximal 4-cycles of $\SComp(\q')$ to get $\IrrComp(\SComp(\q'))$ in fact yields $\sm$, as needed.
	
	As regards the second statement, let $\q'$ be obtained from $\q$ by re-rooting at some oriented edge, and suppose that $\IrrComp(\SComp(\q'))\neq\emptyset$.
	Then $\IrrComp(\SComp(\q'))$ can clearly be identified with a unique submap of $\SComp(\q')$.
	On the other hand, $\SComp(\q')$ can be realized in a non-unique way as a submap of $\q'$, just by choosing for each $e\in \SComp(\q')$ an edge $\alpha(e)$ among those of $\q'$ with same endpoints which have been collapsed to form $e$.
	All in all, $\IrrComp(\SComp(\q'))$ can be realized as a submap of $\q'$, and we can assume that it is a rightmost one just by taking its ``rightmost version'' as in Remark~\ref{rem:submaps-and-rerooting}.
	This means that $\IrrComp(\SComp(\q'))$ can be realized as an irreducible component of $\q$, as needed.
\end{proof}

In particular, if $\q$ is a quadrangulation, then its irreducible components (as defined in Definition~\ref{def:irred-components}) are precisely the ``irreducible root-blocks of simple root-blocks of re-rooted versions of $\q$'', but with an explicit realization as submaps of $\q$.

\subsubsection{Largest blocks/components}

The formulation via \textit{root-blocks} is interesting in that it allows to write identities for related generating functions, which then allow to compute asymptotics for the size and number of largest components in random quadrangulations using analytic tools.
For the cases at hand, this has been performed by Gao and Wormald~\cite{GaoWormald99}; see also Banderier, Flajolet, Schaeffer and Soria \textit{et.~al.} \cite{BanderierFlajoletSchaefferSoria01} for a generalization to many models, as well as precise local limit results for component sizes.

For $\q$ a quadrangulation, we define the sizes of its largest simple/irreducible components as follows:
\begin{align*}
\lbs(\q)=\max_{\q'} |\SComp(\q')|
\qquad \text{and}\qquad
\lbirr(\q)=\max_{\q'} |\IrrComp(\SComp(\q'))|,
\end{align*}
where $\q'$ ranges over re-rootings of $\q$.
Note in particular that if $\qsim$ is a simple quadrangulation, then $\lbirr(\qsim)=\max_{(\qsim)'} |\IrrComp((\qsim)')|$.
Also, by Lemma~\ref{lem:equiv-two-defs}, for every quadrangulation $\q$ we have the alternative definition:
\begin{align*}
\lbirr(\q) = \max \left\{|\sm|\colon \text{$\sm$ irred.~component of $\q$}\right\}.
\end{align*}
The following follows%
	\footnote{
		To be precise, using~\cite[Prop.~2.11]{FleuratSalvy24}, the statement about $\lbs(Q_n)$ is equivalent to (1) in \cite[Thm.~2]{GaoWormald99}, which concerns 2-connected blocks of maps.
		The statement about $\lbirr(\Qsimple_n)$ corresponds to (6) in the same theorem since, as the authors recall, 3-connected blocks of 2-connected maps are dual to irreducible blocks of simple quadrangulations (a word of caution: the authors call \textit{simple quadrangulations} ``quadrangulations'', and they call \textit{irreducible quadrangulations} ``simple quadrangulations'').
	}
 from Theorem~2 of \cite{GaoWormald99}.

\begin{prop}[Gao and Wormald, 1999]\label{prop:Gao-Wormald}
	For $n\geq 2$, let $Q_n$ (resp.~$\Qsimple_n$) be a uniformly random general (reps.~simple) quadrangulation with $n$ faces.
	Then, for every $\epsilon>0$, we have the asymptotics
	\begin{align*}
	\lbs(Q_n)=\frac{n}{3}+o(n^{2/3+\epsilon})
	\qquad\text{and}\qquad
	\lbirr(\Qsimple_n)=\frac{n}{3}+o(n^{2/3+\epsilon}),
	\end{align*}
	in probability as $n\to\infty$.
\end{prop}

Gao and Wormald \cite{GaoWormald99} also prove in their Lemma~4 that the largest irreducible component in a uniform simple quadrangulation is \textit{unique} with high probability.
Using our ``separation lemma'', we can adapt their arguments to get the following.
For $n\geq 2$, we let $Q_n$ (resp.~$\Qsimple_n$) be a uniformly random general (reps.~simple) quadrangulation with $n$ faces.

\begin{lemma}\label{lem:uniqueness-largest-comp}
	With probability $1-o(1)$ as $n\to\infty$, the maximum $\lbirr(Q_n) = \max \left\{|\sm|\colon \text{$\sm$ irred.~component of $\q$}\right\}$ is attained for a unique irreducible component up to re-rooting equivalence.%
		\footnote{
			We recall that two submaps are \textit{re-rooting equivalent} if they have the same edge-sets.
		}
	Furthermore, for every $\epsilon>0$, the second largest irreducible component has $o(n^{2/3+\epsilon})$ faces in probability.
	The same goes when $Q_n$ is replaced by $\Qsimple_n$ in the preceding.
\end{lemma}

\begin{proof}
	We adapt the arguments of \cite[Lemma~4]{GaoWormald99} to the statement concerning $Q_n$---the one about $\Qsimple_n$ is dealt with similarly.
	Since $|\Quads_n|\sim c\rho^n n^{-5/2}$ for some $c,\rho>0$ (see \cite{Tutte63}), it suffices to show that, if we denote by $A_n$ the set of quadrangulations having at least two non-equivalent irreducible components with at least $n^{2/3+\epsilon}$ faces, then $|A_n|=o(\rho^n n^{-5/2})$.
	Let $\q\in A_n$.
	Two non-equivalent irreducible components are separated by a 4-cycle using the ``separation'' Lemma~\ref{lem:sepration-lemma}.
	If we (i) distinguish such a 4-cycle between the two largest components, and (ii) choose an oriented root-edge on it, (iii) cut $\q$ along the 4-cycle; then we obtain two rooted quadrangulations $\q_1$ and $\q_2$ whose number of faces are at least $ n^{2/3+\epsilon}$ and sum to $n+2$.
	To recover $\q$, one only needs to glue the root faces of $\q_1$ and $\q_2$, and choose the root edge among the $4n$ available oriented edges.
	Using that $|Q_k|\sim c \rho^k k^{-5/2}$, the preceding gives:
	\begin{multline*}
	|A_n|
		\leq 	\sum_{k=n^{2/3+\epsilon}}^{n-n^{2/3+\epsilon}}
				4n\cdot|\Quads_{k}|\cdot|\Quads_{n+2-k}|
		= O(1)n^{-3/2}\rho^n\int_{n^{2/3+\epsilon}}^\infty \frac{\diff t}{t^{5/2}}\\
		=O(n^{-3\epsilon/2}) \rho^n n^{-5/2}.
	\end{multline*}
	Hence $|A_n|=o(\rho^n n^{-5/2})$, which concludes the proof.
\end{proof}

We can now use Proposition~\ref{prop:Gao-Wormald} and Lemma~\ref{lem:uniqueness-largest-comp} to deduce asymptotics for the largest irreducible component in uniform general quadrangulations.
For $n\geq 2$, we still let $Q_n$ be a uniformly random general quadrangulation with $n$ faces.

\begin{cor}\label{cor:largest-comp-irred}
	For every $\epsilon>0$, we have the asymptotics
	\begin{align*}
	\lbirr(Q_n)=\frac{n}{9}+o(n^{2/3+\epsilon}),
	\end{align*}
	in probability as $n\to\infty$.
\end{cor}

\begin{proof}
	We let $Q'_n$ be $Q_n$ re-rooted at a uniformly random oriented edge among those which realize the maximum $\lbs(\q)=\max_{\vec e} |\SComp(Q_n)|$.
	
	We claim that there is an event $\Ecal_n$ of probability $1-o(1)$ such that conditionally on $\Ecal_n$, the following holds: (i) the block $\SComp(Q'_n)$ is uniformly random in $\ensQsimple_{N_n}$ conditionally on its size $N_n:=\lbs(Q_n)$; and, (ii) with probability $1-o(1)$ conditionally on $\Ecal_n$, a largest irreducible block of $\SComp(Q'_n)$ is a largest irreducible block of $Q_n$, that is $\lbirr(Q_n)=\lbirr(\SComp(Q'_n))$.
	
	Given the claim, by (i) we have conditionally on $\Ecal_n$ and $N_n$ that the simple quadrangulation $\SComp(Q'_n)$ has the same distribution as $\Qsimple_{N_n}$, where $(\Qsimple_k,k\geq2)$ are uniformly random simple quadrangulations with $k$ faces sampled independently from previously defined random variables.
	Hence, using (ii) and Proposition~\ref{prop:Gao-Wormald}, conditionally on $\Ecal_n$, we have $\lbirr(Q_n)\smash{\overset{(\mathrm{d})}{=}}\lbirr(\Qsimple_{N_n})$, and for every $\epsilon>0$,
	\begin{align*}
	\lbirr(\Qsimple_{N_n})
		=\frac{\lbs(Q_n)}{3}+o(\lbs(Q_n)^{2/3+\epsilon})
		=\frac{n}{9}+o(n^{2/3+\epsilon}),
	\end{align*}
	in probability as $n\to\infty$.
	This proves the corollary.
	
	We now turn to our claim above.
	The event $\Ecal_n$ we consider is ``$Q_n$ has a unique largest simple component up to re-rooting equivalence, and all the other ones have size at most $n^{2/3+0.1}$'', where we define a \textit{simple component} to be a maximal rightmost simple quadrangular submap---there is a similar equivalence between simple blocks and simple components as the one for irreducible blocks/components in Lemma~\ref{lem:equiv-two-defs}.
	We omit the proof that  $\Ecal_n$ has probability $1-o(1)$, which can be proven along similar lines as Lemma~\ref{lem:uniqueness-largest-comp}; and that the resulting largest simple component is a uniformly distributed quadrangulation conditionally on its size---this appears with different definitions of components/blocks in \cite{AddarioBerryWen17} and \cite{FleuratSalvy24}, or one can adapt the proof given in Section~\ref{subsec:proof-lemma-summary-largest-irred} below that the largest irreducible component in $Q_n$ is uniform conditionally on its size and on being the unique largest.
	
	We now justify (ii) in the claim.
	It suffices to show that conditionally on the event $\Ecal_n$, with probability $1-o(1)$, any largest irreducible component of $Q_n$ is included in ``the'' unique largest \textit{simple} component of $Q_n$.
	The event that $\lbirr(Q_n)\geq n^{2/3+0.1}$ has probability $1-o(1)$, even after conditioning on $\Ecal_n$; so it suffices to show that conditionally on both events, the preceding statement regarding largest components holds.
	Hence, we now reason conditionally on both events.
		
	We let $\sm$ be a largest irreducible component of $Q_n$.
	In particular, it is a rightmost simple quadrangular submap, and thus contained in a maximal such submap, \textit{i.e.}~it is contained in a \textit{simple component} $\sm'$.
	Since on $\Ecal_n$, only \textit{the} unique largest simple component can have size more than $n^{2/3+0.1}$, and since $\sm$ has size $\lbirr(Q_n)\geq n^{2/3+0.1}$, we deduce that the preceding simple component $\sm'$ is actually \textit{the} unique largest simple component of $Q_n$.
\end{proof}

\subsection{Proof of Lemma~\ref{lem:summary-largest-comp}}
\label{subsec:proof-lemma-summary-largest-irred}

For $n\geq 2$, let $Q_n$ be a uniformly random general quadrangulation with $n$ faces.
We let $\Ecal'_n$ denote the event that:
\begin{enumerate}
	\item the maximum $\lbirr(Q_n) = \max \left\{|\sm|\colon \text{$\sm$ irred.~component of $\q$}\right\}$ is attained for a unique irreducible component up to re-rooting equivalence; and,
	\item the value $K'_n=\lbirr(Q_n)-\lfloor n/9\rfloor$ satisfies $|K'_n|\leq n^{2/3+0.1}$.
\end{enumerate}
By Corollary~\ref{cor:largest-comp-irred} and Lemma~\ref{lem:uniqueness-largest-comp}, $\Ecal'_n$ has probability $1-o(1)$ as $n\to\infty$.
For $n\geq 2$, we let $(\Qirrfrak_n,\Decs_n,\ifrak_n)$ be a sample of the conditional distribution with respect to $\Ecal'_n$ of the tuple $(\sm,\Decs,\ifrak)=\Psi(Q_n;\Smax,\alpha)$, where $(\Smax,\alpha)$ is a uniformly chosen largest irreducible component of $Q_n$---on the event $\Ecal'_n$, this amounts to fixing uniformly at random an oriented root edge for the ``unique irreducible component up to re-rooting equivalence''.
Lastly, we let $N_n=\lbirr(Q_n)=|\Qirrfrak_n|$, which we write as $N_n=\lfloor n/9\rfloor+K_n$ (notice in particular that $K_n$ has the law of $K'_n$ conditioned to $\Ecal'_n$).
Let us verify that $(N_n,\Qirrfrak_n,\Decs_n)$, $n\geq2$, thus defined satisfy the conclusions of Lemma~\ref{lem:summary-largest-comp}.

Observe that, apart from the irreducible components which are re-rooting equivalent to $\Smax$, the other irreducible components of $Q_n$ are irreducible components of some $\q_i$ using the ``separation lemma'' Lemma~\ref{lem:sepration-lemma}.
Hence the event $\Ecal'_n\cap\{\lbirr(Q_n)=\ell\}$ is precisely the event that $|\Smax|=\ell$ and that $\Decs\in\DecSet_{n,\ell}^{<\ell}$; where $\DecSet_{n,\ell}^{<\ell}$ corresponds to those collections $(\q_i)_i\in\DecSet_{n,\ell}$ for which we have $\lbirr(\q_i)<\ell$ for all $i$.
Using that $\Detach$ induces a bijection $\Detach\colon\CompRootedQ{n}{\ell}\longrightarrow \ensQirred_\ell \times \DecSet_{n,\ell}\times \{1,\dots,4n\}$ by Lemma~\ref{lem:bij-psi-irred}, we therefore conclude that, for $\ell$ in the support of $\lfloor n/9\rfloor+K_n$, conditionally given $\Ecal'_n\cap\{\lbirr(Q_n)=\ell\}$, the tuple $(\sm,\Decs,\ifrak)=\Psi(Q_n;\Smax,\alpha)$ is uniform on
\begin{align*}
\ensQirred_\ell\times \DecSet_{n,\ell}^{<\ell}\times \{1,\dots,4n\}.
\end{align*}
By construction, this means that, conditionally on $N_n=\ell$, the pair $(\Qirrfrak_n,\Decs_n)$ is uniformly distributed on $\ensQirred_\ell\times \DecSet_{n,\ell}^{<\ell}$.

The later directly gives us (ii) of Lemma~\ref{lem:summary-largest-comp}, and also gives us (iv) since for every $4\leq \ell\leq n$, the set $\DecSet_{n,\ell}^{<\ell}$, is invariant under permutation of coordinates.
Let us therefore verify that (i) and (iii) hold.
First, recall that $N_n=\lfloor n/9\rfloor+K_n$.
The bound $|K_n|\leq n^{2/3+0.1}$ holds by definition of the event $\Ecal'_n$ and the fact that $K_n$ is distributed as $K'_n$ conditioned to $\Ecal'_n$, and $\Qirrfrak_n$ is uniformly distributed on $\ensQirred_{N_n}$ conditionally on $N_n$ by the preceding paragraph.
This gives us (iii).
Lastly, observe that by construction, $\Glue(\Qirrfrak_n;\Decs_n,\ifrak_n)$ has the same distribution as $(Q_n;\Smax,\alpha)$ conditioned to the event $\Ecal'_n$.
Taking the first coordinates, this means that $Q'_n=\Gamma(\Qirrfrak_n;\Decs_n,\ifrak_n)$ is distributed as $Q_n$ conditioned to the event $\Ecal'_n$.
Since $\Ecal'_n$ has probability $1-o(1)$, we conclude that the distribution of $Q'_n$ is at vanishing total variation distance from $\rmUnif(\Quads_n)$.
This gives (i) and finishes the proof of Lemma~\ref{lem:summary-largest-comp}.\qed

\bigskip

Now that Lemma~\ref{lem:summary-largest-comp} has been proven, we shall adress the second key lemma used in the proof of Theorem~\ref{thm:scaling-limit-irred}---laid down in Section~\ref{sec:structure-proof}---namely Lemma~\ref{lem:cvg-after-projection}. This is the purpose of the next section.

% -----------------------------------------------------------------
%							SECTION 5
% -----------------------------------------------------------------

\section{Step II: Bottlenecks and Hausdorff convergence}
\label{sec:bottlenecks-and-Hausdorff-cvg}

By the Theorem of Le~Gall and Miermont, that we recalled as Theorem~\ref{thm:Le-Gall--Miermont}, we know that uniformly random quadrangulations with a large number of faces converge to the Brownian sphere in the GHP sense, after a suitable normalization.
On the other hand, in Section~\ref{sec:largest-components} we have seen that the largest simple/irreducible component in uniformly random general/simple quadrangulations with a large number of faces gathers a proportion of the mass which is bounded away from $0$.
The purpose of this section is to lay down some general topological and measure-theoretic arguments which will allow to deduce that these largest components converge jointly to the (same) Brownian sphere as the quadrangulation they are obtained from.

\subsection{Bottlenecks and convergence of geodesic spaces}

Let us begin by recalling what a \textit{geodesic space} is.

\begin{defin}
	A compact metric space $(X,d)$ is said to be a \textit{geodesic space} if for every $x,y\in X$, there exists an isometry $p\colon[0,d(x,y)]\to X$ such that $p(0)=x$ and $p(d(x,y))=y$.
	Such an isometry is called a \textit{geodesic path} between $x$ and $y$.
\end{defin}
\noindent
It is well-known that an Hausdorff limit of compact geodesic spaces is still a geodesic space \cite[Thm.~7.5.1]{BuragoBuragoIvanov01}.
\vspace{\baselineskip}

Let $(Z,\delta)$ be a fixed compact metric space.
For $A\subset Z$ and $\epsilon>0$, we let $A^\epsilon=\{x\in Z\colon\delta(x,A)<\epsilon\}$ be the $\epsilon$-neighborhood of $A$ in $(Z,\delta)$.

\begin{lemma}\label{lem:intersection-neighborhood}
	Let $X$ and $Y$ be two compact subsets of $(Z,\delta)$.
	If $(X\cup Y,\delta)$ is a geodesic space, then $X^\epsilon\cap Y^\epsilon\subset (X\cap Y)^{2\epsilon}$ for every $\epsilon>0$.
\end{lemma}

\begin{proof}
	Let $\epsilon>0$ and let $z\in X^\epsilon\cap Y^\epsilon$.
	By definition, there exist $x\in X$ and $y\in Y$ such that $\delta(x,z)<\epsilon$ and $\delta(x,z)<\epsilon$.
	In particular, we have $\delta(x,y)<2\epsilon$ by the triangle inequality.
	Since $X\cup Y$ is a geodesic space, there exists a geodesic path $p\colon [0,\delta(x,y)]\to X\cup Y$ such that $p(0)=x$ and $p(\delta(x,y))=y$, as in Figure~\ref{fig:setup-lemma-intersection}.
	In particular, for every $t\in [0,\delta(x,y)]$, we have $\delta(p(t),\{x,y\})=\min(t,\delta(x,y)-t) <\epsilon$, using that $\delta(x,y)<2\epsilon$.
	Hence, $\delta(z,p(t))<\epsilon + \max(\delta(z,x),\delta(z,y))<2\epsilon$.
	This proves that $p$ takes its values in in $\ball(z,2\epsilon)$.
	Therefore, in order to prove that $z\in (X\cap Y)^{2\epsilon}$, it suffices to show that we can find $t\in[0,\delta(x,y)]$ such that $p(t)\in X\cap Y$.
	If that were not the case, then $[0,1]=p^{-1}(X)\sqcup p^{-1}(Y)$ would be a partition into closed sets of the connected space $[0,1]$, hence a contradiction. 
\end{proof}

\begin{figure}
	\begin{center}
		\includegraphics[scale=1.2]{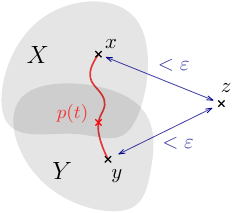}
	\end{center}
	\caption{Setup of the proof of Lemma~\ref{lem:intersection-neighborhood}.
		The geodesic path $p\colon [0,\delta(x,y)]\to X\cup Y$ such that $p(0)=x$ and $p(\delta(x,y))=y$ has been represented in red.}
	\label{fig:setup-lemma-intersection}
\end{figure}

\begin{lemma}\label{lem:Hausdorff-cvg-intersection-geod-spaces}
	Consider two converging sequences $X_n\to X_\infty$ and $Y_n\to Y_\infty$ in the $\delta$-Hausdorff metric.
	Suppose that $X_n\cup Y_n$ is a geodesic space for every $n\geq1$.
	Then $X_n\cap Y_n\to X_\infty\cap Y_\infty$ in the $\delta$-Hausdorff metric.
\end{lemma}

\begin{proof}
	It is easily seen from the definition of Hausdorff convergence that $X_n\cup Y_n\to X_\infty\cup Y_\infty$.
	Let $\epsilon>0$.
	Since $X_n\to X_\infty$ and $Y_n\to Y_\infty$ in the $\delta$-Hausdorff metric, for $n$ large we have $X_n\cap Y_n\subset (X_\infty)^\epsilon\cap (Y_\infty)^\epsilon\subset (X_\infty\cap Y_\infty)^{2\epsilon}$, where the second inclusion comes from Lemma~\ref{lem:intersection-neighborhood}, using that $X_\infty\cup Y_\infty$ is a geodesic space, since it is an Hausdorff limit of the geodesic spaces $X_n\cup Y_n$, $n\geq1$.
	Similarly, we also deduce from the convergences $X_n\to X_\infty$ and $Y_n\to Y_\infty$ that for large $n$, we have $Z_\infty=X_\infty\cap Y_\infty\subset (X_n)^\epsilon\cap (Y_n)^\epsilon\subset (X_n\cap Y_n)^{2\epsilon}$, using Lemma~\ref{lem:intersection-neighborhood} and that $X_n\cup Y_n$ is a geodesic space.
	Hence for large $n$, we have $X_n\cap Y_n\subset (Z_\infty)^{2\epsilon}$ and $Z_\infty\subset (X_n\cap Y_n)^{2\epsilon}$, and the result follows.
\end{proof}

\begin{defin}
	A metric space $(X,d)$ is \textit{2-connected} if $X\setminus\{x_0\}$ is connected for every $x_0\in X$.
\end{defin}

\begin{lemma}\label{lem:bottlenecks-and-diam}
	Suppose that $X_n\cup Y_n\to Z_\infty$ in the $\delta$-Hausdorff metric, where $X_n,Y_n$ are compact subsets of $(Z,\delta)$ for each $n$.
	Assume that $X_n\cup Y_n$ is a geodesic space for every $n\geq1$.
	If $Z_\infty$ is 2-connected and if $\diam(X_n\cap Y_n)\to0$, then $\min(\diam(X_n),\diam(Y_n))\to 0$.
\end{lemma}

\begin{proof}
	We only need%
		\footnote{
			The general case follows, since given any subsequence of $(X_n,Y_n)_n$, we can extract further so that both coordinates converge in the Hausdorff sense, so that every subsequence of $(\min(\diam(X_n),\diam(Y_n)))_n$ has a subsequence converging to $0$, which gives that $\min(\diam(X_n),\diam(Y_n))\to 0$.
			We used that the set of compact subsets of a compact metric space is itself compact when equipped with the Hausdorff metric.
		}
	to consider the case where $X_n\to X_\infty$ and $Y_n\to Y_\infty$ in the $\delta$-Hausdorff metric, for some compact subsets $X_\infty,Y_\infty$ of $Z_\infty$.
	Then, $Z_\infty=X_\infty\cup Y_\infty$ since both sides are the $\delta$-Hausdorff limits of $(X_n\cup Y_n,n\geq1)$.
	We claim that $X_\infty\cap Y_\infty$ is a singleton $\{z_*\}$ for some $z_*\in Z_\infty$.
	If this claim is true, then $Z_\infty\setminus\{z_*\}=(X_\infty\setminus\{z_*\})\sqcup (Y_\infty\setminus\{z_*\})$ is a partition of $Z_\infty\setminus\{z_*\}$ into two closed%
	\footnote{
		The sets $(X_\infty\setminus\{z_*\})$ and $(Y_\infty\setminus\{z_*\})$ are closed in $Z_\infty\setminus\{z_*\}$, not in $Z_\infty$.
	}
	subsets, so that one of them must be empty since $Z_\infty\setminus\{z_*\}$ is connected, which is absurd since $Z_\infty$ is 2-connected.
	Hence, either $X_\infty$ is a singleton or $Y_\infty$ is, so that the result follows by contradiction
	
	We still have to prove our claim that $X_\infty\cap Y_\infty$ is a singleton.
	First, we have $X_n\cap Y_n\neq 0$ for every $n\geq1$, since  otherwise we would have $Z_n=X_n\sqcup Y_n$, thus contradicting that $Z_n$ is connected, as a geodesic space.
	Hence taking a subsequential limit in a sequence of elements of $X_n\cap Y_n$ respectively, we obtain an element $z_*\in X_\infty\cap Y_\infty$.
	Now by Lemma~\ref{lem:Hausdorff-cvg-intersection-geod-spaces}, we have $X_n\cap Y_n\to X_\infty\cap Y_\infty$ in the $\delta$-Hausdorff metric, so that $\diam(X_\infty\cap Y_\infty)=\lim_{n\to\infty}\diam(X_n\cap Y_n)=0$, and thus $X_\infty\cap Y_\infty$ is actually reduced to the singleton $\{z_*\}$, as claimed.
\end{proof}

\subsection{Bottlenecks and ``star-decompositions''}

Let $(Z,\delta)$ be a fixed compact metric space.
For $n\geq1$, we let $Z_n$ be a compact subspace of $(Z,\delta)$.
Suppose that for every $n\geq1$, we are given a ``star-decomposition''%
	\footnote{
		The name ``star-decomposition'' is motivated by the observation that when $Z_n$ is connected, the incidence graph of the collection $(X_n,Y^1_n,Y^2_n,\dots)$ is a star-graph, where the center of the star is the vertex associated to $X_n$.
	}
of $Z_n$ as follows
\begin{align*}
Z_n=X_n\cup\bigsqcup_{i\geq1} Y^{(i)}_n,
\end{align*}
where $X_n$ is a non-empty compact subset of $Z_n$, and where $(Y^{(i)}_n,i\geq1)$ is a collection of pairwise disjoint and possibly empty compact subsets of $Z_n$.
In this setting, we can refine Lemma~\ref{lem:bottlenecks-and-diam} as follows.

\begin{lemma}\label{lem:star-decomposition-with-2-connected-lim}
	Suppose that $Z_n\to Z_\infty$ in the $\delta$-Hausdorff metric, that each $Z_n$ is geodesic, and that $Z_\infty$ is 2-connected.
	If it holds that
	\begin{align*}
	\textstyle\sup_{i\geq0}\diam(X_n\cap Y^{(i)}_n)\to 0\qquad\text{and}\qquad\textstyle\liminf _n\diam(X_n)>0,
	\end{align*}
	then $\sup_{i_\geq1}\diam(Y^{(i)}_n)\to 0$, and thus $X_n\to Z_\infty$ in the $\delta$-Hausdorff metric.
\end{lemma}

\begin{proof}
	It is easily verified from the definition of the  Hausdorff metric that the convergence $\sup_{i_\geq1}\diam(Y^{(i)}_n)\to 0$ implies $\dH(X_n,Z_n)\to 0$ and thus $X_n\to Z_\infty$ in the $\delta$-Hausdorff metric, so let us prove that we indeed have $\sup_{i_\geq1}\diam(Y^{(i)}_n)\to 0$.
	Suppose for contradiction that $\sup_{i_\geq1}\diam(Y^{(i)}_n)$ does not go to $0$ as $n\to \infty$.
	Then, after possibly taking a subsequence, we may assume that there exists a sequence $(i_n)_n$ such that $\liminf_n \diam(Y^{(i_n)}_n)>0$.
	We set for every $n\geq1$,
	\begin{align*}
	\widehat X^{(i_n)}_n=X_n\cup\bigsqcup_{i\neq i_n}Y^{(i)}_n.
	\end{align*}
	Note that for every $n\geq1$, we have $\widehat X^{(i_n)}_n\cup Y^{(i_n)}_n=Z_n$, which is a geodesic space, and the $\delta$-Hausdorff limit $Z_\infty$ is 2-connected.
	Also, since the $(Y^{(i)}_n,i\geq1)$ are pairwise disjoint, we have $\widehat X^{(i_n)}_n\cap Y^{(i_n)}_n=X_n\cap Y^{(i_n)}_n$, whose diameter goes to zero as $n\to\infty$ by the assumption that $\sup_{i\geq0}\diam(X_n\cap Y^{(i)}_n)\to 0$.
	Hence, we deduce from Lemma~\ref{lem:bottlenecks-and-diam} that  $\min(\diam(\widehat X^{(i_n)}_n),\diam(Y^{(i_n)}_n))\to 0$.
	This gives the desired contradiction, since we have $\liminf_n \diam(Y^{(i_n)}_n)>0$ by construction, and since we also have $\diam(\widehat X^{(i_n)}_n)\geq \diam(X_n)$ for every $n\geq1$, where $\liminf _n\diam(X_n)>0$ by assumption.
\end{proof}

The condition that $\liminf _n\diam(X_n)>0$ in the preceding lemma may sometimes be difficult to establish.
Instead, it may be easier to prove that the ``mass'' of $X_n$ does not vanish asymptotically.
We therefore adapt the preceding lemma as follows.

Again, for $n\geq1$, we let $Z_n=X_n\cup\bigsqcup_{i\geq1} Y^{(i)}_n$ be a compact subspace of $(Z,\delta)$, where   $X_n$ is a non-empty compact subset of $Z_n$, and where $\smash{(Y^{(i)}_n,i\geq1)}$ is a collection of pairwise disjoint and possibly empty compact subsets of $Z_n$.
On top of that, we let $\mu_n$ be a finite Borel measure on $Z_n$ for every $n\geq1$.

\begin{lemma}\label{lem:star-decomposition-with-diffuse-lim}
	Suppose that $Z_n\to Z_\infty$ in the $\delta$-Hausdorff metric, that each $Z_n$ is geodesic, and that $Z_\infty$ is 2-connected.
	Suppose additionally that $\mu_n\to\mu_\infty$ in the $\delta$-Prokhorov metric and that $\mu_\infty$ is diffuse.%
		\footnote{We recall that a Borel measure $\mu$ is \textit{diffuse} if every singleton has zero $\mu$-measure.}
	If it holds that
	\begin{align*}
	\textstyle\sup_{i\geq0}\diam(X_n\cap Y^{(i)}_n)\to 0\qquad\text{and}\qquad\textstyle\liminf _n\mu_n(X_n)>0,
	\end{align*}
	then $\sup_{i_\geq1}\diam(Y^{(i)}_n)\to 0$, and thus $X_n\to Z_\infty$ in the $\delta$-Hausdorff metric.
\end{lemma}

\begin{proof}
	By Lemma~\ref{lem:star-decomposition-with-2-connected-lim}, it is sufficient to prove that $\liminf_n\diam(X_n)>0$.
	Suppose for contradiction that this is not the case so that, for some subsequence $(X_{k_n})_n$, we have $\diam(X_{k_n})\to 0$.
	By extracting further a subsequence, we can assume that $X_{k_n}\to X_\infty$ in the $\delta$-Hausdorff metric, where $X_\infty$ is necessarily a singleton $X_\infty=\{x_\infty\}$, since $\diam(X_{k_n})\to 0$.
	Now, since $\mu_n\to\mu_\infty$ in the $\delta$-Prokhorov metric, we have by the Portmanteau theorem that for all $\epsilon>0$,
	\begin{align*}
	\mu_\infty\Bigl(\,\overline\ball^{(Z,\delta)}(x_\infty,\epsilon)\Bigr)
		\geq \limsup_{n\to\infty}\mu_{k_n}\Bigl(\,\overline\ball^{(Z,\delta)}(x_\infty,\epsilon)\Bigr).
	\end{align*}
	Since $X_{k_n}\to \{x_\infty\}$ in the $\delta$-Hausdorff metric, we have $X_{k_n}\subset\overline\ball^{(Z,\delta)}(x_\infty,\epsilon)$ for large $n$.
	Hence, by the last display, we have
	\begin{align*}
	\mu_\infty\Bigl(\,\overline\ball^{(Z,\delta)}(x_\infty,\epsilon)\Bigr)
		\geq \limsup_{n\to\infty}\mu_{k_n}(X_{k_n})
		\geq \liminf_{n\to\infty}\mu_n(X_n).
	\end{align*}
	Hence, by taking $\epsilon\to0$ and using the continuity of measure, we get that $\mu_\infty(\{x_\infty\})\geq \liminf_{n\to\infty}\mu_n(X_n)$.
	By assumption $\liminf_{n\to\infty}\mu_n(X_n)>0$, which gives is a contradiction since $\mu_\infty$ is assumed to be diffuse.
\end{proof}

\subsection{Discrete approximations of the measures}
\label{subsec:Discrete-approximation-measures}

We come back to the setting of the last section: we let $(Z,\delta)$ be a compact space; for $n\geq1$, we let $Z_n=X_n\cup\bigsqcup_{i\geq1} \smash{Y^{(i)}_n}$ be a compact subspace of $(Z,\delta)$, where $X_n$ is a non-empty compact subset of $Z_n$, and where $\smash{(Y^{(i)}_n,i\geq1)}$ is a collection of pairwise disjoint and possibly empty compact subsets of $Z_n$; and lastly, we let $\mu_n$ be a finite Borel measure on $Z_n$ for every $n\geq1$.

Under the assumptions of Lemma~\ref{lem:star-decomposition-with-diffuse-lim}, said lemma tells us that the $Y^{(i)}_n$, $n,i\geq1$, have vanishing diameter as $n\to\infty$.
Therefore, it makes sense to approximate their contribution to the measure $\mu_n$ by a discrete measure.
Suppose that we have fixed some $\smash{y_n^{(i)}\in Y^{(i)}_n}$ for every $n,i\geq1$.
Let us write $\gamma^{(i)}_n=\mu_n\bigl(Y^{(i)}_n\setminus X_n\bigr)$, and denote by $\mu^X_n=\mu_n(\cdot\cap X_n)$ the measure $\mu_n$ restricted to $X_n$.
This allows to define for all $n\geq1$ the measure
\begin{align}\label{eq:star-decomp-discrete-approx-expr}
\widetilde\mu_n(\diff x)=\mu^X_n(\diff x)+\sum_{i\geq 1}\gamma^{(i)}_n\cdot\delta_{y^{(i)}_n}(\diff x),
\end{align}
which will serve to approximate $\mu_n$.

\begin{lemma}\label{lem:star-decomp-discrete-approx}
	In the setting of Lemma~\ref{lem:star-decomposition-with-diffuse-lim} and with the preceding notation, we have $\dP(\widetilde\mu_n,\mu_n)\to 0$ as $n\to\infty$, and in particular $\widetilde\mu_n\to\mu_\infty$ in the $\delta$-Prokhorov metric.
\end{lemma}

\begin{proof}
	For $n,i\geq1$, we let $\widetilde Y^{(i)}_n=Y^{(i)}_n\setminus X_n$, in order to lighten notation.
	Let $\epsilon>0$.
	Fix $n$ large enough such that $\sup_{i_\geq1}\diam(Y^{(i)}_n)\leq \epsilon$ using Lemma~\ref{lem:star-decomposition-with-diffuse-lim}.
	Now given $A\subset Z_n$ a Borel subset, we have by $\sigma$-additivity,
	\begin{align}\label{eq:lem:star-decomp-discrete-approx-1}
	\mu_n(A^{\epsilon})
	=		\mu^X_n(A^\epsilon)+\sum_{i\geq1}\mu_n\bigl(A^\epsilon\cap\widetilde Y^{(i)}_n\bigr)	
	&\geq	\mu^X_n(A)+\sum_{i\geq1}\gamma^{(i)}_n\indic{\{\widetilde Y^{(i)}_n\subset A^{\epsilon}\}}\nonumber	\\
	&\geq	\mu^X_n(A)+\sum_{i\geq1}\gamma^{(i)}_n\indic{\{y^{(i)}_n\in A\}}\nonumber\\
	&= \widetilde\mu_n(A),
	\end{align}
	where the last inequality uses that $\widetilde Y^{(i)}_n$ is included in the $\epsilon$-neighborhood of $y^{(i)}_n$ for every $i\geq1$, since $\sup_{i_\geq1}\diam(Y^{(i)}_n)\leq \epsilon$.
	On the other hand, we have by definition that $\gamma^{(i)}_n=\mu_n(\widetilde Y^{(i)}_n)$, which yields:
	\begin{align*}
	\widetilde\mu_n(A^{\epsilon})
	&=		\mu^X_n(A^\epsilon)+\sum_{i\geq1}\gamma^{(i)}_n\indic{\{y^{(i)}_n\in A^{\epsilon}\}}	\\
	&\geq 	\mu^X_n(A)+\sum_{i\geq1}\mu_n\bigl(A\cap\widetilde Y^{(i)}_n\bigr)\indic{\{y^{(i)}_n\in A^{\epsilon}\}}.
	\end{align*}
	Now for $i\geq1$, we either have $A\cap\widetilde Y^{(i)}_n=\emptyset$ so that the corresponding term in the last sum is zero, or $A\cap\widetilde Y^{(i)}_n\neq\emptyset$ which implies $y^{(i)}_n\in A^{\epsilon}$ since $\sup_{i_\geq1}\diam(Y^{(i)}_n)\leq \epsilon$.
	Hence the indicator in the last displayed sum is redundant, and the last display reduces to
	\begin{align}\label{eq:lem:star-decomp-discrete-approx-2}
	\widetilde\mu_n(A^{\epsilon})
	\geq \mu^X_n(A)+\sum_{i\geq1}\mu_n\bigl(A\cap\widetilde Y^{(i)}_n\bigr)
	= \mu_n(A).
	\end{align}
	The combination of \eqref{eq:lem:star-decomp-discrete-approx-1} and \eqref{eq:lem:star-decomp-discrete-approx-2} gives by definition of the Prokhorov distance that $\dP(\mu_n,\widetilde\mu_n)\leq \epsilon$.
	Since this holds for all large enough $n$, the result follows.
\end{proof}
\begin{figure}
	\begin{center}
		\includegraphics[scale=1]{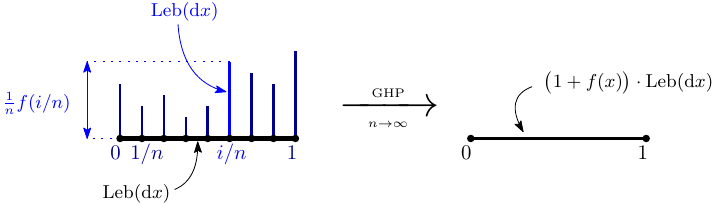}
	\end{center}
	\caption{%
		The ``comb'' counter-example for Remark~\ref{rem:comb-counter-example}.
		Fix any continuous function $f\colon[0,1]\to \R_+$.
		We let $Z_n\subset\R^2$ be defined by $Z_n=X_n\cup\bigsqcup_i Y^{(i)}_n$, where $X_n=[0,1]\times\{0\}$ is equipped with its Lebesgue measure, and for every $n$, the segment $Y^{(i)}_n=\{i/n\}\times[0,\frac 1n f(i/n)]$ is also equipped with its Lebesgue measure.
		Then we replace the induced metric on $Z_n\subset\R^2$ by its \textit{length metric} $d_n$ induced by the length of shortest paths \textit{inside} $Z_n$, so as to turn $Z_n$ into a geodesic space.
		It is easily seen that $(Z_n,d_n,\mu_n)$ converges in the GHP metric to $[0,1]$ with its usual distance, and equipped with the measure $(1+f(x))\mathrm{Leb}(\diff x)$.
		Such a convergence can be made to hold in the Hausdorff+Prokhorov sense inside some compact space $(Z,\delta)$ by Proposition~\ref{prop:GHP-common-embedding}
		}.
	\label{fig:comb-counter-example}
\end{figure}
As a first step towards controlling the discrete part in~\eqref{eq:star-decomp-discrete-approx-expr}, we prove the following lemma.

\begin{lemma}\label{lem:star-decomposition-diffuse-lim-bis}
	In the setting of Lemma~\ref{lem:star-decomposition-with-diffuse-lim}, we have $\sup_{i_\geq1}\mu_n(Y^{(i)}_n)\to 0$, and therefore also $\smash{\sup_{i\geq1}\gamma^{(i)}_n\to 0}$, as $n\to\infty$.
\end{lemma}

\begin{proof}
	If that is not the case, then using that $\sup_{i_\geq1}\diam(Y^{(i)}_n)\to 0$ by Lemma~\ref{lem:star-decomposition-with-diffuse-lim}, an argument along the same lines as the one given in the proof of this lemma shows that $\mu_\infty$ is not diffuse, leading to a contradiction.
\end{proof}

\begin{rem}\label{rem:comb-counter-example}
	Note that even when $\sup_{i_\geq1}\mu_n(Y^{(i)}_n)\to 0$, there is still a lot of liberty in what can happen to the measures $\smash{\mu^X_n}$ and in particular they can be very different from $\mu_\infty$ as $n\to\infty$.
	See for instance the ``comb'' example presented in Figure~\ref{fig:comb-counter-example}, in which $X_\infty=[0,1]$, with $\mu_\infty(\diff x)=(1+f(x))\cdot\mathrm{Leb}(\diff x)$ but $\mu^X_n\to \mathrm{Leb}(\diff x)$.
\end{rem}

\subsection{Proof of Lemma~\ref{lem:cvg-after-projection}}
\label{subsec:proof-cvg-after-proj}

We place ourselves in the setting of Lemma~\ref{lem:summary-largest-comp}.
That is, we have random integers $(N_n)_{n\geq 4}$, and random tuples
\begin{align}\label{eq:recalling-coupling-largest-comp}
(\Qirrfrak_n,\Decs_n,\ifrak_n)
\qquad\text{in}\qquad\ensQirred_{N_n}\times\DecSet_{n,N_n}\times\{1,\dots,4n\}
\end{align}
for $n\geq 4$, such that:
(i) the law of $Q'_n:=\Gamma(\Qirrfrak_n;\Decs_n,\ifrak_n)$ is at $o(1)$ total variation distance from $\rmUnif(\Quads_n)$ as $n\to\infty$; (ii)
$\Qirrfrak_n$ is uniformly distributed in $\smash{\ensQirred_{N_n}}$ conditionally on $N_n$; (iii) we have the deterministic bound $|K_n|\leq n^{2/3+0.1}$ where  $K_n=N_n-\lfloor n/9\rfloor$; and, (iv)  the collection $\Decs_n=(\q_1,\dots,\q_{N_n})\in\DecSet_{n,N_n}$ is exchangeable conditionally on $N_n$.

Since the $Q'_n$ is at vanishing total variation distance from $\rmUnif(\Quads_n)$ when $n\to\infty$, the result of Le~Gall~\cite{LeGall13} and Miermont~\cite{Miermont13} recalled as Theorem~\ref{thm:Le-Gall--Miermont} still holds when replacing $Q_n$ by $Q'_n$ in its statement.
In particular, using the Skorokhod representation theorem, we can assume that we have a coupling of the random quadrangulations $(Q'_n,n\geq 2)$ and of an instance $(S,d,\mu)$ of the Brownian sphere, such that the convergence
\begin{align}\label{eq:almost-sure-cvg-Q-prime}
\left(\vertices(Q'_n),\,\Bigl(\frac{9}{8n}\Bigr)^{1/4}\cdot d_{Q'_n},\,\frac{1}{n+2}\cdot\mu_{Q'_n}\right)\xrightarrow[n\rightarrow\infty]{}(S,d,\mu)
\end{align}
holds in the almost sure sense with respect to the GHP metric.
If, for every $n$, we furthermore re-sample the random variables $(N_n,\Qirrfrak_n,\Decs_n)$ conditionally given the coupling $(Q'_n,n\geq 2)$, then we can assume that all the random variables $((Q'_n,N_n,\Qirrfrak_n,\Decs_n), n\geq 2)$ are coupled in such a way that (i) the preceding almost sure convergence holds and (ii) the conclusions of Lemma~\ref{lem:summary-largest-comp} listed after \eqref{eq:recalling-coupling-largest-comp} are satisfied.

We recall that the Brownian sphere $(S,d,\mu)$ is almost surely homeomorphic to the two-dimensional sphere \cite{LeGallPaulin08,Miermont08}, and in particular 2-connected.
Furthermore, its mass measure $\mu$ is almost surely diffuse by \cite[Thm.~3]{Miermont09}.
We will therefore reason deterministically on the almost-sure event that \eqref{eq:almost-sure-cvg-Q-prime} holds, that $Q'_n=\Gamma(\Qirrfrak_n;\Decs_n,\ifrak_n)$ for every $n$, and that $(S,d,\mu)$ is 2-connected with diffuse measure.

We recall that $\Gamma$ corresponds to the first component of the mapping $\Glue$, so let us write $\Glue(\Qirrfrak_n;\Decs_n,\ifrak_n)=(Q'_n,\Qirrfrak_n,\alpha_n)$.
In particular, $(\Qirrfrak_n,\alpha_n)$ is a submap of $Q'_n$.
It is an easy exercise to verify that \eqref{eq:almost-sure-cvg-Q-prime} also holds upon replacing its left-hand side by the \textit{metric graph} of $Q'_n$ similarly renormalized, where this metric graph is obtained by realizing geometrically each edge using a unit segment with unit mass joining its endpoints.
Clearly, this turns the left-hand side of \eqref{eq:almost-sure-cvg-Q-prime} into a geodesic space, and it converges in the GHP sense to $(S,d,\mu)$ which is \textit{2-connected} with \textit{diffuse} measure.
Using (a GHP variant of) Gromov's embedding theorem, for instance Proposition~\ref{prop:GHP-common-embedding}, we can assume that the latter convergence holds in the Hausdorff--Prokhorov sense inside some compact metric space $(Z,\delta)$.

This is almost the setting of Lemma~\ref{lem:star-decomposition-with-diffuse-lim}.
We let $Z_n$ be the metric graph of $Q'_n$ with distances renormalized as in the left-hand side of \eqref{eq:almost-sure-cvg-Q-prime}, which is therefore a geodesic compact subset of $(Z,\delta)$, and we let $\mu_n$ be the image measure of $\frac{1}{n+2}\mu_{Q'_n}$ under the natural mapping $V(Q'_n)\to Z_n$.
Then, we let $X_n$ be the subset of $Z_n$ corresponding to $\Qirrfrak_n$.
More precisely, recall that $(\Qirrfrak_n,\alpha_n)$ is a submap of $Q'_n$, so that its metric graph naturally corresponds to a subset of the metric graph $Z_n$ of $Q'_n$.

We now make an innocent but crucially important observation: $X_n$  equipped with the distance induced by that of $Z_n$ is actually \textit{isometric} to the metric graph of $
\Qirrfrak_n$ with distances rescaled by $\bigl({9}/{8n}\bigr)^{1/4}$.
Equivalently, the graph distance on the vertices of $\Qirrfrak_n$ is precisely the distance induced on $V(\Qirrfrak_n)$ by the graph distance of $Q'_n$.
This follows from the observation that given a quadrangulation $\q$ and one of its irreducible components $\sm$, there is no strict shortcut between two points of $\sm$ using edges of $E(\q)\setminus E(\sm)$.
Indeed, the antipodal vertices on a 4-cycle cannot be at distance $0$, for otherwise the 4-cycle is not simple; neither at distance $1$, since a quadrangulation is bipartite.

Now, since $\Qirrfrak_n$ is a quadrangulation and a submap of $Q'_n$, the connected components $(Y_n^{(i)})_i$ of $Z_n\setminus X_n$ each intersect $X_n$ at points which correspond (under the metric graph construction) to vertices of $Q'_n$ on the boundary of a face of $\Qirrfrak_n$, and said boundary is a 4-cycle.
Hence, in view of how distances are renormalized in the definition of $Z_n$, we deduce that $\sup_i \diam(X_n\cap Y_n^{(i)})\leq 4 \cdot (9/8n)^{1/4}\to 0$.
Lastly, by construction $\mu_n(X_n)$ equals $(N_n+2)/(n+2)$, where $N_n+2$ is the number of vertices of $\Qirrfrak_n$.
Under our coupling, we have $\liminf_n \mu_n(X_n)=\liminf N_n/n=1/9$ almost surely, which is positive.

All in all, the assumptions of Lemma~\ref{lem:star-decomposition-with-diffuse-lim} have all been checked, allowing us to deduce that $(X_n)_n$ converges almost surely to the Brownian sphere $Z_\infty=\lim_n Z_n$ in the Hausdorff sense.
Also, using Lemma~\ref{lem:star-decomp-discrete-approx}, we deduce that the image of the measure $\frac{1}{n+2}\cdot(\mu_{\Qirrfrak_n}+\boldnu_{\Decs_n})$ (defined in the statement of Lemma~\ref{lem:cvg-after-projection}), under the metric graph construction and the embedding in $(Z,\delta)$, converges as $n\to\infty$ to the measure $\mu_\infty=\lim_n\mu_n$ carried by the Brownian sphere $Z_\infty$, almost surely in the Prokhorov sense.
Also, by Lemma~\ref{lem:star-decomposition-diffuse-lim-bis}, we have that $\sup_x\frac{1}{n+2}\cdot \boldnu_{\Decs_n}(x)\to 0$ almost surely.

We can now conclude.
By the preceding, abstracting away from the embedding into $(Z,\delta)$, we get that the metric graph version of $\Qirrfrak_n$ (with distances renormalized by $(9/8n)^{1/4}$), equipped with the measure $\frac{1}{n+2}\cdot(\mu_{\Qirrfrak_n}+\boldnu_{\Decs_n})$, converges almost surely to $(S,d,\mu)$ in the above coupling.
Again, one easily switches between the metric graph version and the vertex set equipped with the graph distance, which gives us the convergence
\begin{align*}
\left(\vertices(\Qirrfrak_n),\,\left(\frac{9}{8n}\right)^{1/4}\cdot d_{\Qirrfrak_n},\,\frac{1}{n+2}\cdot(\mu_{\Qirrfrak_n}+\boldnu_{\Decs_n})\right)\xrightarrow[n\rightarrow\infty]{\distribGHP}(S,d,\mu),
\end{align*}
as needed.
To complete the proof of Lemma~\ref{lem:cvg-after-projection}, note that we have already justified that $\sup_x\frac{1}{n+2}\cdot \boldnu_{\Decs_n}(x)\to 0$ almost surely in the above coupling, and note also that the exchangeability statement about the measures $\boldnu_{\Decs_n}$, $n\geq2$, follows from the exchangeability of each $\Decs_n$ conditionally on $N_n$, as given by Lemma~\ref{lem:summary-largest-comp}.
Hence the proof of Lemma~\ref{lem:cvg-after-projection} is complete.\qed

In the next section, we continue our proof of the key lemmas presented in Section~\ref{sec:structure-proof} and establish Lemma~\ref{lem:cvg-after-concentration}.

% -----------------------------------------------------------------
%							SECTION 6
% -----------------------------------------------------------------

\section{Step III: Exchangeability and Prokhorov convergence}
\label{sec:exch-and-Prokhorov}

This section is a continuation of Section~\ref{sec:bottlenecks-and-Hausdorff-cvg}.
Our aim is now to gain control on the measure part in Lemma~\ref{lem:cvg-after-projection}, so as to replace $\frac{1}{n+2}\cdot(\mu_{\Qirrfrak_n}+\boldnu_{\Decs_n})$ by a constant multiple of $\frac{1}{n+2}\cdot\mu_{\Qirrfrak_n}$.
Unfortunately, deterministic topological/measure-theoretic arguments are not sufficient in general for that purpose, see Remark~\ref{rem:comb-counter-example}.
Using ideas of Addario-Berry and Wen \cite{AddarioBerryWen17}, it is possible to circumvent this difficulty using the \textit{exchangeability} property that our random measure $\boldnu_{\Decs_n}$ satisfies.
We revisit their argument in a general context, and propose a replacement for a concentration bound that seems to contain an error, see Remark 6.4. This is then applied to prove Lemma~\ref{lem:cvg-after-concentration}.

\subsection{Revisiting the Addario-Berry--Wen exchangeability argument}
A random collection indexed by a finite set is said to be \textit{exchangeable} if its distribution is invariant under permutation of the coordinates. For instance, the random vector $(\xi_1,\dots,\xi_n)$ is exchangeable if for every permutation $\sigma\in\perm_n$, the vector $(\xi_{\sigma(1)},\dots,\xi_{\sigma(n)})$ has the same distribution as $(\xi_1,\dots,\xi_n)$.

We let $(Z,\delta)$ be a compact space, and we consider a sequence $(D_n,n\geq1)$ of finite subsets of $Z$.
Suppose that $(\boldnu_n,n\geq1)$ are \textit{random} measures supported on $D_n$ respectively.
We shall prove the following.

\begin{prop}\label{prop:Add-Wen-exch-argument}
	Assume that the collection $(\boldnu_n(x))_{x\in D_n}$ is exchangeable
	for every $n\geq1$.
	If $\max_{x\in D_n} \boldnu_n(x) \to 0$ in probability, then
	\begin{align*}
	\dP\Bigl(\boldnu_n,|\boldnu_n|\cdot\rmUnif_{D_n}\Bigr)\xrightarrow[n\to\infty]{}0
	\end{align*}
	in probability.
\end{prop}

The remainder of this section is dedicated to the proof of Proposition~\ref{prop:Add-Wen-exch-argument}.
We first need the following lemma, which tells that the Prokhorov distance between two measures can be bounded by the total variation distance between suitable ``coarse-grained'' versions of the measures.

\begin{lemma}\label{lem:coarse-graining-and-dP}
	Let $\epsilon>0$, and consider a finite partition $Z=\bigsqcup_{i=1}^{K_\epsilon} Z_{i,\epsilon}$ into Borel subsets of the compact space $Z$, such that $\max_i \diam(Z_{i,\epsilon})<\epsilon$.
	Then for any two finite Borel measures $\mu,\nu$ on $Z$, we have:
	\begin{align}
	\dP(\mu,\nu)\leq \epsilon+\sum_{1\leq i\leq K_\epsilon}|\mu(Z_{i,\epsilon})-\nu(Z_{i,\epsilon})|.
	\end{align}
\end{lemma} 

\begin{proof}
	For all Borel subset $A\subset Z$, using that the $(Z_{i,\epsilon})_i$ partition $Z$ and that $\max_i \diam(Z_{i,\epsilon})<\epsilon$, we have:
	\begin{align*}
	\mu(A^{\epsilon})
		=		\sum_i \mu(A^{\epsilon}\cap Z_{i,\epsilon})
		&\geq 	\sum_i \mu(Z_{i,\epsilon})\indic{\{ Z_{i,\epsilon}\subset A^{\epsilon}\}}
		\geq	\sum_i \mu(Z_{i,\epsilon})\indic{\{A\cap Z_{i,\epsilon}\neq\emptyset\}}.
	\end{align*}
	If we let $\Delta=\sum_i|\mu(Z_{i,\epsilon})-\nu(Z_{i,\epsilon})|$, then we can lower bound further by:
	\begin{align*}
	\left(\sum_i \nu(Z_{i,\epsilon})\indic{\{A\cap Z_{i,\epsilon}\neq\emptyset\}}\right)-\Delta
		&\geq	\left(\sum_i \nu(A\cap Z_{i,\epsilon})\indic{\{A\cap Z_{i,\epsilon}\neq\emptyset\}}\right)-\Delta\\
		&= 	\nu(A)-\Delta.
	\end{align*}
	Hence, $\mu(A^{\epsilon})\geq \nu(A)-\Delta$, and also $\nu(A^{\epsilon})\geq \mu(A)-\Delta$ by symmetry. 
	Since this holds for all $A$, we get $\dP(\mu,\nu)\leq \epsilon +\Delta$.
\end{proof}

Next, we will need the following concentration result which is proven in Section~\ref{subsec:variance-inequ-exch}.
It is used as a replacement for Addario-Berry and Wen's Lemma~5.3 in \cite{AddarioBerryWen17}, see Remark~\ref{rem:problem-Addario-Berry-Wen}.

\begin{lemma}
\label{lem:reproduced-cor:concentration-exch-proba-measure}
	Let $X$ be a non-empty finite set and let $\boldnu$ be a random probability measure on $X$ such that the collection $(\boldnu(x))_{x\in X}$ is exchangeable.
	Then the inequality
	\begin{align*}
	\Var\Bigl(\boldnu(A)\Bigr)\leq \frac{|A|}{|X|}\cdot\Expect{\max_{x\in X}\boldnu(x)}
	\end{align*}
	holds for all $A\subset X$.
\end{lemma}

We can now prove Proposition~\ref{prop:Add-Wen-exch-argument}.

\begin{proof}[Proof of Proposition~\ref{prop:Add-Wen-exch-argument}]
	Let $\epsilon>0$.
	Since $Z$ is a compact space, one can easily construct a finite partition $Z=\bigsqcup_{i=1}^{K_\epsilon} Z_{i,\epsilon}$ into Borel subsets, in such a way that $\max_i \diam(Z_{i,\epsilon})<\epsilon$.
	Fix $n\geq1$.
	By Lemma~\ref{lem:coarse-graining-and-dP} we have:
	\begin{align*}
	\dP\Bigl(\boldnu_n,|\boldnu_n|\cdot\rmUnif_{D_n}\Bigr)
		&\leq \epsilon+\sum_{1\leq i\leq K_\epsilon}\Bigl|\boldnu_n(Z_{i,\epsilon})-|\boldnu_n|\cdot\rmUnif_{D_n}(Z_{i,\epsilon})\Bigr|\\
		&=	\epsilon+\sum_{1\leq i\leq K_\epsilon}\Bigl|\boldnu_n(Z_{i,\epsilon})-\Expect{\boldnu_n(Z_{i,\epsilon})}\Bigr|,
	\end{align*}
	The last equality can be justified as follows: from the exchangeability of the collection $(\boldnu_n(x))_{x\in D_n}$, we have $$\Expect{|\boldnu_n|}=\Expect{\sum_{x\in D_n} \boldnu_n(x)}=|D_n|\cdot\Expect{\boldnu(x_0)}$$ for every $x_0\in D_n$, and therefore $\Expect{\boldnu_n(Z_{i,\epsilon})}=\Expect{\sum_{x\in Z_{i,\epsilon}}\boldnu_n(x)}=|Z_{i,\epsilon}|\cdot|\boldnu_n|/|D_n|$, that is $|\boldnu_n|\cdot\rmUnif_{D_n}(Z_{i,\epsilon})$.
	In particular, from the last display and Chebychev's inequality we deduce the bound:
	\begin{multline*}
	\Prob{\dP\Bigl(\boldnu_n,|\boldnu_n|\cdot\rmUnif_{D_n}\Bigr)\geq 2\epsilon}\\
	\begin{aligned}[t]
		&\leq	\Prob{ \sum_{1\leq i\leq K_\epsilon}\Bigl|\boldnu_n(Z_{i,\epsilon})-\Expect{\boldnu_n(Z_{i,\epsilon})}\Bigr|\geq\epsilon }\\
		&\leq 	\sum_{1\leq i\leq K_\epsilon}\Prob{ \Bigl|\boldnu_n(Z_{i,\epsilon})-\Expect{\boldnu_n(Z_{i,\epsilon})}\Bigr|\geq\epsilon/K_\epsilon }\\
		&\leq	\sum_{1\leq i\leq K_\epsilon} (\epsilon/K_\epsilon)^{-2}\cdot\Var\Bigl(\boldnu_n(Z_{i,\epsilon})\Bigr).
	\end{aligned}
	\end{multline*}
	Hence, using Lemma~\ref{lem:reproduced-cor:concentration-exch-proba-measure} to bound the above variances, we finally obtain:
	\begin{align*}
	\Prob{\dP\Bigl(\boldnu_n,|\boldnu_n|\cdot\rmUnif_{D_n}\Bigr)\geq 2\epsilon}
		\leq K_\epsilon \cdot(\epsilon/K_\epsilon)^{-2} \cdot\Expect{\max_{x\in D_n} \boldnu_n(x)}.
	\end{align*}
	We have by assumption $\max_{x\in D_n} \boldnu_n(x)\to 0$ in probability.
	We also trivially have $\max_{x\in D_n} \boldnu_n(x)\leq 1$ since $\boldnu_n$ is a probability measure, so that the convergence $\max_{x\in D_n} \boldnu_n(x)\to 0$ holds in the $L^1$ sense.
	Hence, the right-hand side in the last display goes to zero as $n\to\infty$ and the result follows.
\end{proof}

\begin{rem}\label{rem:problem-Addario-Berry-Wen}
	Instead of the variance estimate in Lemma~\ref{lem:reproduced-cor:concentration-exch-proba-measure}, Addario-Berry and Wen use the concentration bound \cite[Lemma~5.3]{AddarioBerryWen17}, which seems to contain an error.
	In substance, the claimed concentration bound amounts to saying%
		\footnote{
			The bound \eqref{eq:bound-add-wen} is a re-writing (with different notation) of the bound at the end of \cite[p.~1908]{AddarioBerryWen17}.
		}
	that if we are given an exchangeable non-negative random vector $\boldX=(X_1,\dots,X_n)$ with $\lnorm2\boldX>0$, then for all $1\leq k\leq n$ and all $t>0$ we have
	\begin{align}\label{eq:bound-add-wen}
	\condProb{\Bigl|\sum_{i=1}^k (X_i-\Expect{X_i})\Bigr|>2t}{\lnorm2\boldX}\leq 2\exp\left(-\frac{2 t^2}{\lnorm2\boldX^2}\right).
	\end{align}
	As a counter-example to this inequality, let $Y$ be an arbitrary non-constant random variable on the positive reals.
	For $n\geq1$, if we consider the exchangeable vector $(X_1,\dots,X_n)=(Y,Y,\dots,Y)$, then $\lnorm2\boldX=\sqrt n Y$.
	Hence the left-hand side in the last display is the indicator of the event $\{k\cdot|Y-\Expect{Y}|>2t\}$, while the right-hand side equals $2\exp(-2t^2/(nY^2))$.
	For every $\epsilon>0$ and $n\geq1$, upon setting $k=n$ and $t=\epsilon n$ in \eqref{eq:bound-add-wen} we would thus get:
	\begin{align*}
	\indic{\{|Y-\Expect{Y}|>2\epsilon\}}\leq 2\exp\left(-{2n\epsilon^2}/{Y^2}\right).
	\end{align*}
	The latter goes to zero as $n\to\infty$ so that $Y$ is almost surely equal to its mean, hence a contradiction.
\end{rem}

\subsection{A variance bound for exchangeable random vectors}
\label{subsec:variance-inequ-exch}

For $n\geq1$ and $\y=(y_1,\dots,y_n)\in\R^n$, we let:
\begin{align*}
\lnorm{1}{\y}=\sum_{i=1}^n y_i,
\qquad
\lnorm{2}{\y}=\biggl(\sum_{i=1}^n y_i^2\biggr)^{1/2},
\qquad
\lnorm{\infty}{\y}=\max_{i=1,\dots,n}y_i.
\end{align*}
Our goal in this section is to prove Lemma~\ref{lem:reproduced-cor:concentration-exch-proba-measure}.
We will prove in fact the following more general statement.

\begin{prop}\label{prop:variance-bound}
	Let $n\geq1$, and let $\boldxi=(\xi_1,\xi_2,\dots,\xi_n)\in\R^n$ be an exchangeable random vector such that $\Expect{\lnorm2\boldxi^2}<\infty$.
	We let $S_k=\xi_1+\dots+\xi_k$ for $1\leq k\leq n$.
	Then for all $1\leq k\leq n$, we have:
	\begin{align*}
	\Var(S_k)\leq \frac{k}{n}\cdot \Expect{\lnorm2\boldxi^2\Bigm.\!}
	+\frac{k^2}{n^2}\cdot\Var(S_n).
	\end{align*}
\end{prop}

Before turning to the proof of Proposition~\ref{prop:variance-bound}, let us show that Lemma~\ref{lem:reproduced-cor:concentration-exch-proba-measure} indeed follows from it.

\begin{proof}[Proof of Lemma~\ref{lem:reproduced-cor:concentration-exch-proba-measure}.]
	Let $A\subseteq X$, say $A\neq\emptyset$ to avoid trivialities, and set $k=|A|$ and $n=|X|$.
	We enumerate the vertices of $X$ as $(x_1,\dots,x_n)$, in such a way that $A=\{x_1,\dots,x_k\}$.
	If we set $\xi_\ell=\boldnu(x_\ell)$ for $1\leq\ell\leq n$, then by construction the vector $\boldxi=(\xi_1,\dots,\xi_n)$ is exchangeable and $\boldnu(\{x_1,\dots,x_\ell\})$ equals $S_\ell=\xi_1+\dots+\xi_\ell$ for every $1\leq\ell\leq n$.
	Hence, by Proposition~\ref{prop:variance-bound}, we have:
	\begin{align*}
	\Var(\boldnu(A))
		=		\Var(S_k)
		\leq 	\frac{k}{n}\cdot \Expect{\lnorm2\boldxi^2\Bigm.\!}+\Var(S_n)
		&\leq	\frac{k}{n}\cdot\Expect{\lnorm\infty\boldxi\Bigm.\!}+0\\
		&=		\frac{|A|}{|X|}\cdot\Expect{\max_{x\in X}\boldnu(x)},
	\end{align*}
	where we used that $\boldnu$ is a probability measure, so that on the one hand $\lnorm2\boldxi^2\leq \lnorm1\boldxi\cdot\lnorm\infty\boldxi=1\cdot\lnorm\infty\boldxi$, and on the other hand $\Var(S_n)=0$ since $S_n=1$ almost surely.
	Since $A$ is an arbitrary subset of $X$, the result  follows.
\end{proof}

The remainder of Section~\ref{subsec:variance-inequ-exch} is dedicated to the proof of Proposition~\ref{prop:variance-bound}.

We recall that given $(\Omega,\Fcal,\P)$ a probability space and $\Gcal\subset\Fcal$ a $\sigma$-field, the conditional covariance of two random variables $X,Y\in L^2(\Omega,\Fcal,\P)$ is the random variable $\Cov(X,Y|\Gcal)\in L^1(\Omega,\Gcal,\P)$ given by
\begin{align*}
\Cov(X,Y|\Gcal)
	&=	\condExpect{
			\bigl(X-\condExpect{X}{\Gcal}\bigr)\cdot\bigl(Y-\condExpect{Y}{\Gcal}\bigr)	\Big.\!
		}{\Gcal}\\
	&=	\condExpect{X\cdot Y}{\Gcal}-\condExpect{X}{\Gcal}\cdot\condExpect{Y}{\Gcal}.
\end{align*}
The conditional covariance is linked with the unconditional covariance \textit{via}:
\begin{align}\label{eq:cov-vs-conditional-cov}
\Cov(X,Y)=\Expect{\Cov(X,Y|\Gcal)\Bigm.\!} + \Cov\Bigl(\condExpect{X}{\Gcal},\condExpect{Y}{\Gcal}\Bigr).
\end{align}
In the following lemma, we make the simple observation that when conditioning with respect to the ``empirical distribution'' of an exchangeable vector $\boldxi$, conditional expectations can be expressed as ``empirical expectations''.
More precisely, given $\boldxi=(\xi_1,\dots,\xi_n)\in\R^n$, we denote by $\pclass{\boldxi}$ the class of $\boldxi$ modulo the action of $\perm_n$ by permutation of coordinates.
Note that for $\boldxi\in\R^n$, the class $\pclass{\boldxi}$ depends only on the empirical distribution $\sum_i \delta_{\xi_i}(\diff y)$ of $\boldxi$.
Then we have the following.

\begin{lemma}\label{lem:expr-expect-exch}
	Let $n\geq1$, and let $\boldxi=(\xi_1,\xi_2,\dots,\xi_n)\in\R^n$ be an exchangeable random vector.
	Then, for any measurable functions $f\colon \R\to \R$ and $g\colon\R^2\to\R$ such that $f(\xi_1)$ and $g(\xi_1,\xi_2)$ are integrable, we have:
	\begin{align*}
	\condExpect{f(\xi_k)}{\pclass{\boldxi}}
		&=\frac{1}{n}\sum_i f(\xi_i),	& k&\in\{1,\dots,n\}\\
	\condExpect{g(\xi_k,\xi_\ell)}{\pclass{\boldxi}}
		&=\frac{1}{n(n-1)}\sum_{i\neq j} g(\xi_i, \xi_j), &  k,\ell&\in\{1,\dots,n\}, \quad k\neq\ell.
	\end{align*}
\end{lemma}

\begin{proof}
	We fix $k,\ell\in\{1,\dots,n\}$ such that $k\neq \ell$.
	By definition of $\boldxi\mapsto\pclass{\boldxi}$, conditionally on $\pclass{\boldxi}$, the vector $\boldxi$ is still exchangeable.
	Hence, conditionally on $\pclass{\boldxi}$, the pair $(\xi_k,\xi_\ell)$ has the same distribution as $(\xi_i,\xi_j)$ for every pair $(i,j)\in\{1,\dots,n\}^2$ with $i\neq j$, so that:
	\begin{align*}
	\condExpect{g(\xi_k,\xi_\ell)\Bigm.\!}{\pclass{\boldxi}}
	&=\frac{1}{n(n-1)}\sum_{i\neq j}	\condExpect{g(\xi_i,\xi_j)\Bigm.\!}{\pclass{\boldxi}}\\
	&=\frac{1}{n(n-1)} \condExpect{\sum_{i\neq j} g(\xi_i,\xi_j)}{\pclass{\boldxi}}\\
	&=\frac{1}{n(n-1)} \sum_{i\neq j} g(\xi_i,\xi_j),
	\end{align*}
	where the last equality uses that $\sum_{i\neq j} g(\xi_i,\xi_j)$ is a $\pclass{\boldxi}$-measurable random variable since it is defined as a permutation-invariant function of $(\xi_1,\dots,\xi_n)$.
	Similarly, conditionally on $\pclass{\boldxi}$, the random variable $\xi_k$ has the same distribution as $\xi_i$ for every $1\leq i\leq n$, so that:
	\begin{align*}
	\condExpect{f(\xi_k)}{\pclass{\boldxi}}
		=\frac{1}{n}\sum_i\condExpect{f(\xi_i)}{\pclass{\boldxi}}
		=\frac{1}{n}\condExpect{\sum_i f(\xi_i)}{\pclass{\boldxi}}
		=\frac{1}{n}\sum_i f(\xi_i),
	\end{align*}
	which concludes the proof.
\end{proof}

In particular, we obtain as an easy corollary the following negative dependence statement. 

\begin{cor}\label{cor:neg-cor-exch-given-values}
	Let $n\geq1$, and let $\boldxi=(\xi_1,\xi_2,\dots,\xi_n)\in\R^n$ be an exchangeable random vector such that $\xi_1$ possesses a second moment.
	Then, we have almost surely:
	\begin{align*}
	\Cov\Bigl(\xi_k,\xi_\ell\bigm|\pclass{\boldxi}\Bigr)\leq 0, &&  k,\ell\in\{1,\dots,n\}, \quad k\neq\ell.
	\end{align*}
\end{cor}

\begin{proof}
	We fix $k,\ell\in\{1,\dots,n\}$ such that $k\neq \ell$.
	By Lemma~\ref{lem:expr-expect-exch}, we obtain:
	\begin{align*}
	\Cov\bigl(\xi_k,\xi_\ell\bigm|\pclass{\boldxi}\bigr)
		&=	\condExpect{\xi_k \cdot \xi_\ell}{\pclass{\boldxi}} - 	\condExpect{\xi_k}{\pclass{\boldxi}}\cdot\condExpect{\xi_\ell}{\pclass{\boldxi}}\\
		&=	\frac{1}{n(n-1)}\sum_{i\neq j}\xi_i \xi_j- \frac{1}{n^2}\sum_{i,j}\xi_i \xi_j\\
		&=	\left( \frac{1}{n(n-1)}	- \frac{1}{n^2} \right)
		\sum_{i,j} \xi_i \xi_j
		-\frac{1}{n(n-1)}\sum_i \xi_i^2\\
		&\leq
			\left( \frac{1}{n(n-1)}	- \frac{1}{n^2} \right)
			n\sum_i \xi_i^2
			-\frac{1}{n(n-1)}\sum_i \xi_i^2\\
		&=0,
	\end{align*}
	using that $\sum_i \xi_i\leq \sqrt n\sqrt{\sum_i \xi_i^2}$ by the Cauchy--Schwarz inequality.
\end{proof}

\begin{rem}
	Let us consider the case where $\boldxi=(\xi_1,\xi_2,\dots,\xi_n)\in\R^n$ is an exchangeable random vector taking only a discrete set of values. 
	Then for $\y=(y_1,\dots,y_n)$ in the support of its distribution, conditionally on $\pclass{\boldxi}=\pclass{\y}$, the vector $\boldxi$ has the same distribution as $(Y_1,\dots,Y_n)=(y_{\boldsigma(1)},\dots,y_{\boldsigma(n)})$, where $\boldsigma$ is uniformly sampled in $\perm_n$.
	A distribution of this form is known to be \textit{negatively associated}, see Joag-Dev and Proschan's celebrated work \cite{Joag-DevProschan83}, and more precisely their Theorem~2.11.
	\textit{Negative association} is a strong ``negative dependence'' property.
	For the vector $(Y_1,\dots,Y_n)$, it means that for every partition $\{1,\dots,n\}=I\sqcup J$ of the index set and for every coordinate-wise increasing functions $f\colon \R^I\to\R$ and $g\colon \R^J\to\R$, we have $\Cov(f(Y_i,i\in I),g(Y_j,j\in J))\leq 0$.
\end{rem}

We are now in a position to prove the main statement of this section.

\begin{proof}[Proof of Proposition~\ref{prop:variance-bound}.]
	Let $1\leq k\leq n$.
	By \eqref{eq:cov-vs-conditional-cov}, we have:
	\begin{align}\label{eq:expr-var-vs-conditional-var-proof-bound}
	\Var(S_k)
		=\Expect{\Var\bigl(S_k\bigm|\pclass{\boldxi}\bigr)}
		+\Var\Bigl(\condExpect{S_k}{\pclass\boldxi}\Bigr).
	\end{align}
	Let us begin by treating the first term.
	By Corollary~\ref{cor:neg-cor-exch-given-values}, we have
	\begin{align*}
	\Var\bigl(S_k\bigm|\pclass{\boldxi}\bigr)
		&=\sum_{i=1}^k \Var\bigl(\xi_i\bigm|\pclass{\boldxi}\bigr)
		+\sum_{i,j\leq k\,i\neq j}\Cov\bigl(\xi_i,\xi_j\bigm|\pclass{\boldxi}\bigr)\\
		&\leq\sum_{i=1}^k \Var\bigl(\xi_i\bigm|\pclass{\boldxi}\bigr)\\
		&\leq \sum_{i=1}^k \condExpect{\xi_i^2}{\pclass\boldxi}.
	\end{align*}
	Using Lemma~\ref{lem:expr-expect-exch}, this gives that:
	\begin{align}\label{eq:proof-variance-bound-exch-1}
	\Var\bigl(S_k\bigm|\pclass{\boldxi}\bigr)
		\leq \sum_{i=1}^k\left(\frac 1 n\sum_i \xi_i^2\right)
		= \frac k n \cdot\lnorm2\boldxi^2.
	\end{align}
	On the other hand, by Lemma~\ref{lem:expr-expect-exch}, we also have:
	\begin{align}\label{eq:proof-variance-bound-exch-2}
	\Var\Bigl(\condExpect{S_k}{\pclass\boldxi}\Bigr)
		= \Var\Bigl(\sum_{i=1}^k\condExpect{\xi_i}{\pclass\boldxi}\Bigr)
		&= \Var\Bigl(k\cdot\frac 1 n \sum_{i=1}^n \xi_i\Bigr)\nonumber\\
		&=\frac{k^2}{n^2} \cdot\Var(S_n).
	\end{align}
	The result follows by combining \eqref{eq:expr-var-vs-conditional-var-proof-bound} with \eqref{eq:proof-variance-bound-exch-1} and \eqref{eq:proof-variance-bound-exch-2}.	
\end{proof}

\subsection{Proof of Lemma~\ref{lem:cvg-after-concentration}}
\label{subsec:proof-cvg-after-concentration}

We place ourselves in the setting of Lemma~\ref{lem:cvg-after-projection}.
In particular, with the notation used there, we recall that $\boldnu_{\Decs_n}(\diff x)=\sum_{j} \bigl(|V(\q_j)|-4\bigr)\cdot \delta_{v_j}(\diff x)$,
where $v_j$ is a uniformly chosen vertex incident to the $j$-th face of $\Qirrfrak_n$ in $<_{\Qirrfrak_n}$-order.
The former lemma gives that
\begin{align}\label{eq:proof-after-concentration-1}
\left(\vertices(\Qirrfrak_n),\,\left(\frac{9}{8n}\right)^{1/4}\cdot d_{\Qirrfrak_n},\,\frac{1}{n+2}\cdot(\mu_{\Qirrfrak_n}+\boldnu_{\Decs_n})\right)\xrightarrow[n\rightarrow\infty]{\distribGHP}(S,d,\mu),
\end{align}
and that the collection $(\boldnu_{\Decs_n}(v_j),1\leq j\leq N_n)$ is exchangeable conditionally on $N_n$ and its supremum converges to $0$ in probability as $n\to\infty$.
By the Skorokhod representation theorem, both the convergence \eqref{eq:proof-after-concentration-1} and the convergence $\to 0$ can be strengthened to hold in the almost sure sense, up to properly coupling the involved random variables.

From now on, we reason conditionally on the sequences $(\Qirrfrak_n,n\geq2)$ and $(\sup_x \boldnu_{\Decs_n}(x),n\geq 2)$, coupled as described precedingly.
Using an embedding theorem akin to Proposition~\ref{prop:GHP-common-embedding}, we can assume that $(\vertices(\Qirrfrak_n),\,\left(\frac{9}{8n}\right)^{1/4}\cdot d_{\Qirrfrak_n})$ converges to $(S,d)$, in the Hausdorff distance inside some compact space $(Z,\delta)$.
Crucially, this compact $(Z,\delta)$, and the embedding of the $(\vertices(\Qirrfrak_n),\,\left(\frac{9}{8n}\right)^{1/4}\cdot d_{\Qirrfrak_n})$, $n\geq2$ into it, can be chosen as a deterministic function of $(\vertices(\Qirrfrak_n),\,\left(\frac{9}{8n}\right)^{1/4}\cdot d_{\Qirrfrak_n})$, $n\geq2$, only; that is in a way which does not depend on the measures $(\boldnu_{\Decs_n})_n$.
One (artificial) way to enforce this is to apply Proposition~\ref{prop:GHP-common-embedding} to the convergence
\begin{align*}
\left(\vertices(\Qirrfrak_n),\,\left(\frac{9}{8n}\right)^{1/4}\cdot d_{\Qirrfrak_n},\,0\right)\xrightarrow[n\rightarrow\infty]{\distribGHP}(S,d,0).
\end{align*}
This allows to see the measures $(\boldnu_{\Decs_n})_n$ as random measures on the compact space $(Z,\delta)$, such that the collection $(\boldnu_{\Decs_n}(v_j),1\leq j\leq N_n)$ is exchangeable%
	\footnote{
		We recall that we are reasoning conditionally on $(\Qirrfrak_n,n\geq2)$ and $(\sup_x \boldnu_{\Decs_n}(x),n\geq 2)$.
	}
and satisfies $\sup_x \frac{1}{n+2}\boldnu_{\Decs_n}(x)\to 0$ almost surely.

In particular, by Proposition~\ref{prop:Add-Wen-exch-argument}, if we set $D_n=\{v_1,\dots,v_{N_n}\}$, then
\begin{align*}
\dP\Bigl(\tfrac{1}{n+2}\cdot\boldnu_{\Decs_n},\tfrac{|\boldnu_{\Decs_n}|}{n+2}\cdot\rmUnif_{D_n}\Bigr)\xrightarrow[n\to\infty]{}0.
\end{align*}
Now by construction, the mass $|\boldnu_{\Decs_n}|$ counts those vertices among the $n+2$ ones of $Q'_n$ (defined in Lemma~\ref{lem:summary-largest-comp}) which are not incident to its unique largest component.
Hence $|\boldnu_{\Decs_n}|=(n+2)-(|\Qirrfrak_n|+2)$, which is asymptotically equivalent to $8n/9$ by the bound on $K_n$ in Lemma~\ref{lem:summary-largest-comp}.
In particular, we can turn the last display into
\begin{align}\label{eq:proof-after-concentration-2}
\dP\Bigl(\tfrac{1}{n+2}\cdot\boldnu_{\Decs_n},\tfrac{8}{9}\cdot\rmUnif_{D_n}\Bigr)\xrightarrow[n\to\infty]{}0.
\end{align}
Now $\rmUnif_{D_n}$ is the uniform measure on the set $\{v_1,\dots,v_{N_n}\}$ which contains one vertex incident to each face.
That is, it projects the uniform measure on the face-set of $\Qirrfrak_n$ to vertices incident to each face.
Since the faces have diameter $O(n^{-1/4})$ under $\left(\frac{9}{8n}\right)^{1/4}\cdot d_{\Qirrfrak_n}$, it is easily seen that $\rmUnif_{D_n}$ is an approximation of the measure which instead spreads the uniform measure on the face-set of $\Qirrfrak_n$ evenly on the vertices incident to each face, namely the degree-biased measure $\mubiased_n$ on $V(\Qirrfrak_n)$ which gives mass $\deg(v)/2|E(\Qirrfrak_n)|$ to a vertex $v$.
Using a deterministic comparison, due to Addario-Berry and Wen~\cite[Lemma~5.1]{AddarioBerryWen17}, between the degree-biased measure and the uniform measure on vertices of any quadrangulation, we therefore conclude that $\mubiased_n$ is asymptotically close in the Prokhorov metric to the uniform measure on $V(\Qirrfrak_n)$, which is itself close to $\frac{9}{n}\mu_{\Qirrfrak_n}$ since $V(\Qirrfrak_n)\sim n/9$ in probability.
The conclusion of this whole discussion is that:
\begin{align}\label{eq:proof-after-concentration-3}
\dP\Bigl(\tfrac{8}{9}\cdot\rmUnif_{D_n},\frac{8}{n+2}\mu_{\Qirrfrak_n}\cdot\Bigr)\xrightarrow[n\to\infty]{}0.
\end{align}
Wrapping up, \eqref{eq:proof-after-concentration-1}, \eqref{eq:proof-after-concentration-2}, and \eqref{eq:proof-after-concentration-3} together yield the convergence
\begin{align*}
\left(\vertices(\Qirrfrak_n),\,\left(\frac{9}{8n}\right)^{1/4}\cdot d_{\Qirrfrak_n},\,\frac{9}{n+2}\cdot\mu_{\Qirrfrak_n}\right)\xrightarrow[n\rightarrow\infty]{\distribGHP}(S,d,\mu),
\end{align*}
where $(S,d,\mu)$ is the Brownian sphere.\qed

We now continue our proof of the key lemmas stated in Section~\ref{sec:structure-proof}, with a shift of focus towards verifying the \textit{Tauberian assumption} needed to apply Theorem~\ref{thm:GHP-Tauberian-thm}  with the convergence in the last display. The next two sections accomplish this, by first leveraging (in the next section) Addario-Berry's increasing couplings for irreducible of the hexagon~\cite{Addario-Berry14}, and then in Section~\ref{sec:removing-hexagon} by ``removing the hexagon''.

% -----------------------------------------------------------------
%							SECTION 7
% -----------------------------------------------------------------

\section{Step IV: Growing irreducible quadrangulations of the hexagon}
\label{sec:growing-irred-of-hexagon}

Our aim in this section is to show the stochastic increase, with respect to the partial order $\orderGHP$, of uniformly random irreducible quadrangulations \textit{of the hexagon} with $n$ faces---or more precisely of the associated measure metric spaces obtained by equipping the vertex-set with the graph distance and the counting measure. We recall that quadrangulations of the hexagon are defined in Section~\ref{subsec:step-growing-irred-hex}.

\subsection{Surjective graph homomorphisms and GHP ordering}
\label{subsec:graph-hom-and-GHP-order}

We represent (non-oriented multi-)graphs as pairs $G=(V(G),E(G))$, where the vertex-set $V(G)$ is a finite set, the edge-set $E(G)$ is a multiset%
	\footnote{
		For us, a multiset is a set $X$, together with an (implicit) function $m\colon X\to \{1,2,\dots\}$ assigning a multiplicity to each element.
	}
of 1-~or 2-element subsets of $V(G)$.
In particular, this definition allows loops and multiple edges.

Given two graphs $G,G'$, we recall that a \textit{graph homomorphism} $G\to G'$ is a mapping $\phi\colon V(G)\rightarrow V(G')$ such that if $e=\{u,v\}\in E(G)$, then $\{\phi(u),\phi(v)\}\in E(G')$.
The latter edge $\{\phi(u),\phi(v)\}$ is denoted by $\phi(e)$.

Given a multi-graph $G=(V,E)$, we denote by $\X(G)=(V(G),d_G,\mu_G)$ the measured metric space obtained by endowing $V(G)$ with the graph distance $d_G$ and the counting measure $\mu_G$ which assigns unit mass to every vertex.

\begin{lemma}\label{lem:graph-hom-and-GHP-order}
	Consider two finite graphs $G,G'$.
	If there exists a surjective graph homomorphism $G'\to G$, then $\X(G)\orderGHP\X(G')$.
\end{lemma}

\begin{proof}
	Let $\phi\colon V(G')\to V(G)$ be a surjective graph homomorphism $G'\to G$.
	In particular, if $(e'_1,\dots,e'_\ell)$ is a path of length $\ell$ in $G'$ between some $u',v'\in V(G')$, then the path $(\phi(e'_1),\dots,\phi(e'_\ell))$ joins $\phi(u')$ and $\phi(v')$ in $G$.
	Hence, we have $d_{G}(\phi(u'),\phi(v'))\leq d_{G'}(u',v')$.
	Also for every $v\in V(G)$, we have $\mu_{G}(v)=1\leq |\phi^{-1}(v)|=\mu_{G'}(\phi^{-1}(v))$, where the inequality follows from the surjectivity of $\phi$.
	All in all, we have proven that $\phi\colon V(G')\to V(G)$ is a surjective mapping which is 1-Lipschitz and measure-contracting.
	By definition, this means that $\X(G)\orderGHP\X(G')$.
\end{proof}

\begin{rem}
	We leave to the reader to check that, in fact, given two finite graphs $G,G'$, we have $\X(G)\orderGHP\X(G')$ if and only if there exists a surjective graph homomorphism $G'\to \Loop(G)$, where $\Loop(G)$ is the graph obtained from $G$ by adding a loop at every vertex, that is by adding to $E(G)$ the edge $\{v,v\}$ for every $v\in V(G)$.
\end{rem}

\begin{figure}
	\begin{center}
		\includegraphics[page=1,scale=.6]{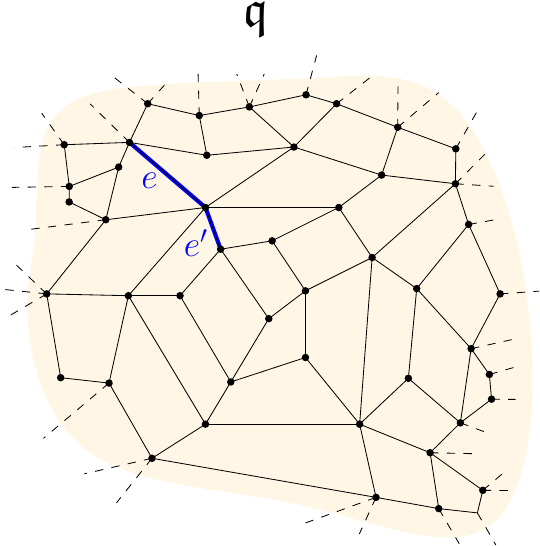}
		\hspace{2em}
		\includegraphics[page=2,scale=.6]{figures/ouverture_face}
	\end{center}
	\caption{Depiction of how $\open(\q,e,e')$ is constructed, starting from a quadrangulation $\q$ and two edges $e,e'$ sharing one endpoint.}
	\label{fig:face-opening}
\end{figure}

\subsection{Addario-Berry's coupling}

Addario-Berry proves~\cite{Addario-Berry14}, that there exists a coupling $(\incrQirrhex_n,n\geq3)$ of the uniform distributions on $(\ensQirredhex_{n}, n\geq 3)$ respectively, in such a way that for every $n\geq3$, the quadrangulation of the hexagon $\incrQirrhex_{n+1}$ is obtained from $\incrQirrhex_n$ by ``opening a face''.

More precisely, given a planar map $\m$, and two distinct non-loop edges $e,e'\in E(\q)$ which share exactly one of their endpoints, we let $\open(\m,e,e')$ denote the planar map obtained by slitting along the path formed by $e$ and $e'$ in $\q$ and then filling the created hole with a quadrangular face.
This operation is depicted in \ref{fig:face-opening} in a case where $\m$ is a quadrangulation.
Then Addario-Berry's coupling result can be formulated as follows.

\begin{prop}[Section~3.2 in \cite{Addario-Berry14}]\label{prop:Addario-Berry-coupling-irred}
	There exists a coupling $(\incrQirrhex_n,n\geq3)$ of the uniform distributions on $(\ensQirredhex_{n}, n\geq 3)$ respectively, such that for every $n\geq3$, we have $\incrQirrhex_{n+1}=\open(\incrQirrhex_n,e,e')$ for some $e,e'\in E(\incrQirrhex_n)$ sharing one endpoint.
\end{prop}

The crucial observation is that performing a face-opening on a planar map $\m$ yields a ``bigger'' planar map with respect to the partial order $\orderGHP$.
More precisely, if we denote by $\X(\m)$  the metric measure space associated to the underlying graph of a planar map $\m$, as in Section~\ref{subsec:graph-hom-and-GHP-order}, then we have the following lemma, which applies in particular to Addario-Berry's coupling in Proposition~\ref{prop:Addario-Berry-coupling-irred}.

\begin{lemma}\label{lem:opening-face-gives-GHP-increase}
	Let $\m$ be a planar map and $e$, $e'$ be two non-loop edges of $\m$ sharing exactly one endpoint.
	Then, $\X(\m)\orderGHP\X(\open(\m,e,e'))$.
\end{lemma}

\begin{proof}
	The operation of closing the newly opened face in $\open(\m,e,e')$ clearly defines a surjective graph homomorphism from the underlying graph of the planar map $\open(\m,e,e')$ to that of $\m$.
	Hence the result follows by Lemma~\ref{lem:graph-hom-and-GHP-order}.
\end{proof}

The last remaining step in our proof Theorem~\ref{thm:scaling-limit-irred}, described in Section~\ref{sec:structure-proof}, is to ``remove the hexagon''. This is performed by Lemma~\ref{lem:comparison-irred-vs-hexagon}, which is equivalent to Proposition~\ref{prop:dTV-coupling-irred}, which we prove in the next section.

% -----------------------------------------------------------------
%							SECTION 8
% -----------------------------------------------------------------

\section[Step V: Removing the hexagon]{Step V: Removing the hexagon}
\label{sec:removing-hexagon}

In this section, we show that up to appropriately filling the hexagonal ``hole'' in a uniformly sampled irreducible quadrangulation of the hexagon (with fixed size) and re-rooting at a uniformly random oriented edge, the result is a uniformly sampled irreducible quadrangulation without hexagonal face (with different fixed size).
This allows to use Addario--Berry's increasing coupling for uniform irreducible quadrangulations of the hexagon.

\begin{figure}
	\begin{center}
		\includegraphics[page=2,scale=.8]{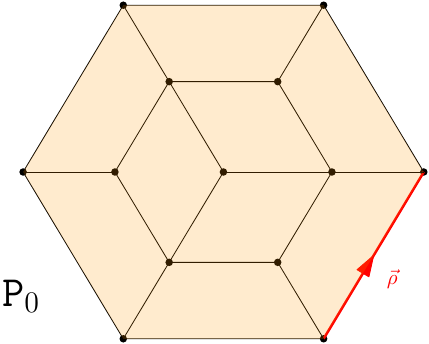}
	\end{center}
	\caption{The ``pattern'' $\pat$, which is a specific irreducible quadrangulation of a hexagon which has 9 interior faces.}
	\label{fig:the-pattern}
\end{figure}

\subsection{Construction of a coupling}
We consider the following construction, which is illustrated in Figure~\ref{fig:filling-pattern}.

\begin{figure}
	\begin{center}
		\includegraphics[scale=.7]{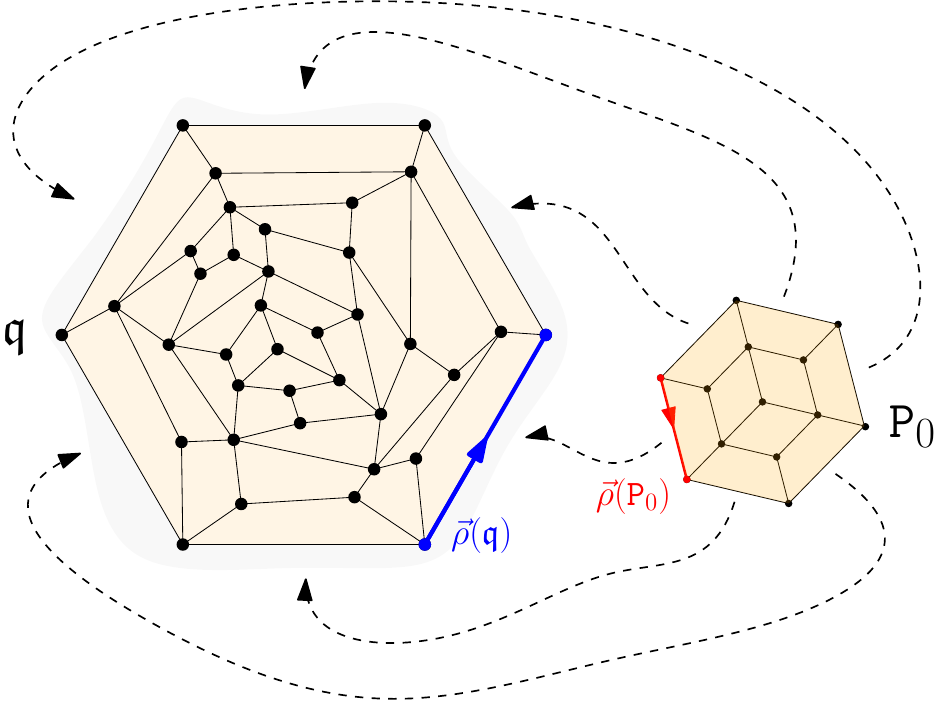}
	\end{center}
	\caption{A depiction of the operation of filling the hexagonal face in an irreducible quadrangulation of the hexagon denoted by $\q$, as in Construction~\ref{constr:coupling-irred}.
		This amounts to gluing the two hexagonal faces so that the root edges of $\q$ and of the pattern coincide with a reversed orientation.}
	\label{fig:filling-pattern}
\end{figure}

\begin{constr}\label{constr:coupling-irred}
	Let $n\geq11$ and consider $\Qirrhex_{n-9}$ a uniformly random element of $\ensQirredhex_{n-9}$.
	Then perform the following:
	\begin{enumerate}
		\item 
			Fill the hexagonal face with the quadrangulation of the hexagon $\pat$ depicted in Figure~\ref{fig:the-pattern}, so that the root edges of $\Qirrhex_{n-9}$ and of $\pat$ are glued together, with a reversed orientation.
			This operation is depicted in Figure~\ref{fig:filling-pattern}.
		\item
			The resulting \textit{quadrangulation} is denoted by $\Qfilled_n$.
			Instead of inheriting the edge-rooting of $\Qirrhex_{n-9}$, the quadrangulation $\Qfilled_n$ is re-rooted at a uniformly random oriented edge of $\Qfilled_n$.
		\item
			We use the shorthand notation $\vec \ehex$ for the root edge of $\Qirrhex_{n-9}$ with its orientation reversed, which by construction is also an oriented edge of $\Qfilled_n$.
	\end{enumerate}
	This construction yields a random tuple $(\Qirrhex_{n-9},\Qfilled_n,\vec \ehex)$.
\end{constr}

The main result of this section is the following proposition.

\begin{prop}\label{prop:dTV-coupling-irred}
	Let $\Qirr_n$ be uniformly random in $\ensQirred_n$.
	Then, we have:
	\begin{align*}
	\dTV\left(\Law(\Qfilled_n),\Law(\Qirr_n)\right)\xrightarrow[n\to\infty]{}0.
	\end{align*}
\end{prop}

The remainder of Section~\ref{sec:removing-hexagon} is dedicated to the proof of Proposition~\ref{prop:dTV-coupling-irred}, which is obtained at the very end of Section~\ref{subsec:controlling-bias-coupling-irred}.
Heuristically, the possible bias that $\Qfilled_n$ may have is controlled by how many occurences of $\pat$, as a submap, there may be in $\Qfilled_n$, and this number of occurences will be shown to be concentrated.

\subsection{Pattern insertion/removal and irreducibility}

Let us first make precise what we mean by an \textit{occurence} of some pattern in a planar map.
Consider a fixed planar map $\m_0$, that we call the \textit{pattern}, whose root face (the one to the right of the oriented root edge) has a simple boundary.

\begin{defin}[Occurrence of a pattern]
	Given a planar map $\m$ and an oriented edge $\vec e$ of $\m$, we say that there is an occurrence of $\m_0$ at $\vec e$ in $\m$ if there exists a (necessarily unique) cyclic simple path $\gamma_{\vec e}$ in $\m$ containing $\vec e$ such that the submap rooted at $\vec e$ enclosed by $\gamma_{\vec e}$ to the \textit{left} of $\vec e$, is precisely $\m_0$.
	We denote by $\Occ_{\m_0}(\m)$ the set of oriented edges $\vec e$ of $\m$ such that there is an occurrence of $\m_0$ at $\vec e$ in $\m$.	
\end{defin}

Given an occurrence of a pattern, we may \textit{remove} this occurrence, thus leaving a face whose degree is the length of the boundary of said pattern.
This is the reverse operation to the operation of \textit{filling with some pattern} used in Construction~\ref{constr:coupling-irred}.

\begin{defin}[Removing an occurence]
	Given $\vec e\in\Occ_{\m_0}(\m)$, we denote by $\remove_{\m_0}(\m;\vec e)$ the map obtained form $\m$ by removing the occurence of $\m_0$ corresponding to $\vec e$.
	More formally, $\remove_{\m_0}(\m;\vec e)$ is the submap of $\m$ enclosed by $\gamma_{\vec e}$ to the \textit{right} of $\vec e$, and its oriented root edge is chosen to be the edge $\rev(\vec e)$, which we define to be $\vec e$ with a reversed orientation.
\end{defin}

\begin{figure}
	\centering
	\resizebox{\textwidth}{!}{%
	\begin{tabular}{ c c c }
		\includegraphics[page=1,scale=.3]{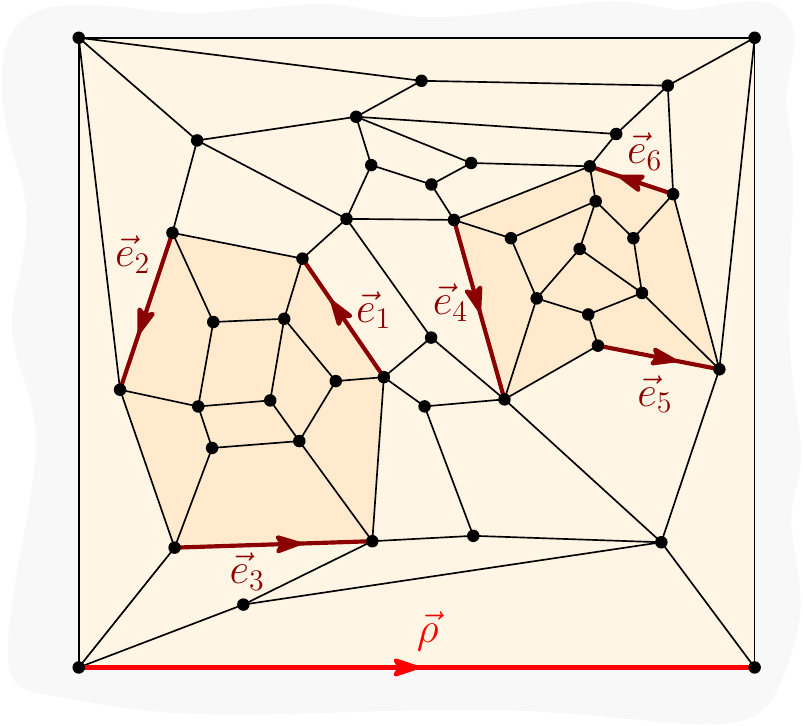}
		&&
		\includegraphics[page=2,scale=.3]{figures/remove_pattern.pdf}
	\end{tabular}
	}
	\includegraphics[page=3,scale=.55]{figures/remove_pattern.pdf}
	\caption{\textit{On the top left}, an irreducible quadrangulation $(\q,\vec\rho)$ together with the oriented edges $(\vec e_i)_{1\leq i\leq 6}$ which are in $\occPat(\q)$.
		Notice how each unrooted copy of $\pat$ comes with three distinct oriented edges belonging to $\occPat(\q)$.
		\textit{On the top right}, the irreducible quadrangulation of the hexagon $\rempat(\q;\vec e_1)$. \textit{At the bottom}, the same map $\rempat(\q;\vec e_1)$ with its root face drawn as the unbounded face.}
	\label{fig:remove-pattern}
\end{figure}

\subsubsection{The importance of the pattern $\pat$}

From now on, the \textit{pattern} will be the quadrangulation of the hexagon $\pat$, depicted in Figure~\ref{fig:the-pattern}.
We could have made other choices for $\pat$, as long as the following proposition is satisfied, which says that $\pat$ does not affect irreducibility upon insertion or removal.

\begin{lemma}\label{lem:pattern-vs-irreducibility}
	Let $\q$ be a quadrangulation and let $\vec e\in\occPat(\q)$.
	Then $\q$ is an irreducible quadrangulation if and only if $\rempat(\q;\vec e)$ is an irreducible quadrangulation of the hexagon.
\end{lemma}

\begin{proof}
	First, it is clear that $\q$ is simple if and only if $\rempat(\q;\vec e)$ is simple.
	Hence we only need to consider the case where $\q$ is simple, which we now assume.
	We let $\q_0$ be the submap of $\q$ isomorphic to $\pat$ which is rooted at $\vec e$.
	By definition, the quadrangulation of the hexagon $\q'=\rempat(\q;\vec e)$ is the submap of $\q$ spanned by the faces of $\q$ which do not belong to $\q_0$.
	We observe that every (non-oriented) simple 4-cycle in $\q$ takes one of the following forms:
	\begin{enumerate}
		\item
		It does not use inner edges of $\q_0$.
		Hence it is a separating 4-cycle in $\q$ if and only if it is a separating 4-cycle in $\q'$. 
		\item
		It uses only edges belonging to $\q_0$.
		Hence it is not separating since there is no separating 4-cycle in $\pat$.
		\item
		It uses at least one inner edge of $\q_0$ and one edge outside of it.
		Then it is necessarily%
		\footnote{
			Indeed, notice that $\pat$ has the property that every path joining two distinct vertices incident to the hexagonal face has length at least $3$, and the inequality is strict unless said vertices are neighbors.
		}
		of the following form: start at some $v_1$ on the boundary of $\q_0$, use the unique path of length 3 in $\q_0$ towards the clockwise neighbor $v_2$ of $v_1$, and close the $4$-cycle using an edge between $v_2$ and $v_1$ which does not belong to $\q_0$.
	\end{enumerate}
	If the quadrangulation $\q$ we started from is simple, then the last case never occurs since there is then a unique edge between $v_2$ and $v_1$, namely the edge which figures in $\q_0$.
	Hence, using the first two cases above, we see that $\q$ is irreducible if and only if $\q'=\rempat(\q;\vec e)$ is irreducible.
\end{proof}

\begin{cor}
	For every $n\geq11$, the random quadrangulation $\Qfilled_n$ in the Construction~\ref{constr:coupling-irred} is irreducible, and therefore takes values in $\ensQirred_n$.
\end{cor}

\begin{proof}
	By construction, we have $\Qirrhex_{n-9}=\rempat(\Qfilled_n;\vec \ehex)$, where the edge $\vec \ehex$ is defined in Construction~\ref{constr:coupling-irred}.
	Since $\Qirrhex_{n-9}$ is irreducible, the result follows by Lemma~\ref{lem:pattern-vs-irreducibility}, and by observing that removing an occurence of $\pat$ in a quadrangulation decreases by $9$ the number of quadrangular faces, since $\pat$ has 9 quadrangular faces.
\end{proof}

\subsection{Controlling the introduced bias}
\label{subsec:controlling-bias-coupling-irred}

Given a quadrangulation $\q$, which may be a quadrangulation of the hexagon, we define $N(\q)$ to be the number of occurrences of $\pat$ in $\q$, that is $N(\q)=\card{\occPat(\q)}$.

\begin{rem}\label{rem:couting-occ-unrooted-vs-rooted}
	Alternatively, $N(\q)$ is three times the number of unrooted submaps of $\q$ which are isomorphic to the unrooted version of $\pat$.
	Indeed, each such unrooted submaps comes with three distinct possible rootings which make them isomorphic as rooted maps to $\pat$, see for instance Figure~\ref{fig:remove-pattern}.
\end{rem}

\begin{prop}\label{prop:comparison-coupling-irred}
	Let $n\geq11$ and let $(\Qirrhex_{n-9},\Qfilled_n,\vec \ehex)$ be as given by Construction~\ref{constr:coupling-irred}.
	Consider $\Qirr_n$ uniformly random in $\in\ensQirred_n$ and let $\vec U_n$ be a uniformly random oriented edge of $\Qirr_n$.
	Then there is a non-zero probability that $\vec U_n\in\occPat(\Qirr_n)$, and for every $\R$-valued function $\phi$ on the set of quadrangulations with a distinguished oriented edge, we have:
	\begin{align*}
	\Expect{\phi(\Qfilled_n,\vec \ehex)}
		=\condExpect{\phi(\Qirr_n,\vec U_n)}{\vec U_n\in\occPat(\Qirr_n)}.
	\end{align*}
\end{prop}

\begin{proof}
	The assertion that $\smash{\Prob{\vec U_n\in\occPat(\Qirr_n)}\neq0}$ is clear.
	Now for $\q\in\ensQirred_n$ and $\vec e\in\occPat(\q)$, we can recover $\q$ from $\q'=\rempat(\q;\vec e)$ by filling its hexagonal face with $\pat$ so that $\vec \rho(\pat)$ is identified with $\rev(\vec e)$, and by choosing the correct oriented root edge for the obtained quadrangulation.
	Since by Construction~\ref{constr:coupling-irred}, we have $\vec\ehex\in\occPat(\Qfilled_n)$ and $\smash{\Qirrhex_{n-9}=\rempat(\Qfilled_n,\vec\ehex)}$, the preceding discussion entails that for $\q\in\ensQirred_n$ and $\vec e\in\occPat(\q)$, we have:
	\begin{align*}
	\Prob{%
		\Qfilled_n=\q,\vec\ehex=\vec e
	}
	&= \Prob{%
		\Qirrhex_{n-9}=\rempat(\q;\vec e),\,
		\vec\rho(\Qfilled_n)=\vec \rho(\q)
	}\nonumber\\
	&= \frac{1}{4n\cdot|{\ensQirredhex_{n-9}}|},
	\end{align*}
	where we used the fact that $\Qirrhex_{n-9}$ is uniform in $\ensQirredhex_{n-9}$ and that $\Qfilled_n$ is uniformly re-rooted at the end of Construction~\ref{constr:coupling-irred}, at one of its $4n$ oriented edges (note that $\vec\rho(\Qfilled_n)$ can be an edge of the copy of $\pat$ we used to fill the hexagon).
	
	Hence, the pair $(\Qfilled_n,\vec\ehex)$ is uniformly distributed on the set of pairs $(\q,\vec e)$ where $\q\in\ensQirred_n$ and $\vec e\in\occPat(\q)$.
	In order to conclude, we only need to show that the same goes for the pair $(\Qirr_n,\vec U_n)$ conditioned to the event $\{\vec U_n\in\occPat(\Qirr_n)\}$, which we now do.
	For $\q\in\ensQirred_n$ and $\vec e\in\occPat(\q)$,	we have
	\begin{multline*}
	\condProb{ \Qirr_n=\q,\vec U_n=\vec e	}{\vec U_n\in\occPat(\Qirr_n)}
		=	\frac{%
				\Prob{ \Qirr_n=\q,\vec U_n=\vec e}
			}{%
				\Prob{\vec U_n\in\occPat(\Qirr_n)}
			}\\
		=	\frac{%
			1
		}{%
			4n\cdot|\ensQirred_n|\cdot\Prob{\vec U_n\in\occPat(\Qirr_n)}
		},
	\end{multline*}
	which does not depend on the pair $(\q,\vec e)$, as needed.
\end{proof}

The preceding proposition allows to express the distribution of $\Qfilled_n$ in terms of that of $\Qirr_n$, as follows.

\begin{cor}\label{cor:bias-coupling-irred}
	Let $n\geq11$ and let $\Qfilled_n$ be as given by Construction~\ref{constr:coupling-irred}.
	Let also $\Qirr_n$ be uniformly random in $\in\ensQirred_n$.
	For every mapping $\phi\colon\ensQirred_n\to \R$, we have the identity
	\begin{align*}
	\Expect{\phi(\Qfilled_n)}
		=	\Expect{\phi(\Qirr_n)\beta_n(\Qirr_n)},
	\end{align*}
	where the factor $\beta_n(\q)$ is defined for $\q\in\ensQirred_n$ by:
	\begin{align}\label{eq:expression-bias-coupling-irred}
	\beta_n(\q)=\frac{N(\q)}{\Expect{N(\Qirr_n)}}.
	\end{align}
\end{cor}

\begin{proof}
	Let $\phi\colon\ensQirred_n\to \R$ be a fixed mapping.
	If $\vec U_n$ is a uniformly random oriented edge of $\Qirr_n$, then we have $\condProb{U_n\in\occPat(\Qirr_n)}{\Qirr_n}=N(\Qirr_n)/4n$.
	Hence by Proposition~\ref{prop:comparison-coupling-irred}, we get:
	\begin{align*}
	\Expect{\phi(\Qfilled_n)}
		=	\frac{%
				\Expect{%
					\phi(\Qirr_n)\cdot
					\indic{ \{\vec U_n\in\occPat(\Qirr_n)\} }}
			}{%
				\Prob{\vec U_n\in\occPat(\Qirr_n)}
			}
		=	\frac{%
				\Expect{%
					\phi(\Qirr_n)\cdot
					\frac{N(\Qirr_n)}{4n}}
			}{%
				\Prob{\vec U_n\in\occPat(\Qirr_n)}
			}.
	\end{align*} 
	The result follows by observing that if we apply the above with $\phi\equiv 1$, then we get $\Expect{N(\Qirr_n)}=4n\cdot\P{(\vec U_n\in\occPat(\Qirr_n))}$.
\end{proof}

Corollary~\ref{cor:bias-coupling-irred} tells us that it is sufficient to prove that $\beta_n(\Qfilled_n)\to1$ in the $L^1$ sense to get Proposition~\ref{prop:dTV-coupling-irred}.
We will actually prove that this convergence holds in $L^2$.
Let us begin with a simple preparatory lemma.

\begin{lemma}\label{lem:number-occ-after-removal}
	Let $n\geq11$.
	For $\q\in\ensQirred_n$ and $\vec e\in\occPat(\q)$, we have the identity $N(\q)=N(\rempat(\q;\vec e))+3$.
\end{lemma}

\begin{proof}
		It is easily seen that two unrooted copies of $\pat$ in $\q$ can overlap if and only if they are the same copy, so that ``removing'' one copy by forming $\rempat(\q;\vec e)$ decreases by exactly one the number of unrooted copies.
		The result follows since $N(\q)$ is three times the number of such copies by Remark~\ref{rem:couting-occ-unrooted-vs-rooted}.
\end{proof}

The next lemma will be central to our argument to prove that $N(\Qirr_n)$ is concentrated.
It relies on local limit results by Addario-Berry~\cite{Addario-Berry14}.

\begin{lemma}\label{lem:local-lim-irred}
	Let $n\geq11$.
	We let $(\Qirrhex_{n-9},\Qfilled_n,\vec \ehex)$ be given by Construction~\ref{constr:coupling-irred} and we let $\Qirr_n$ be uniformly random in $\ensQirred_n$.
	For $n\geq11$, we let $\smash{\vec U_{n}}$ and $\smash{\vec U'_{n}}$ be uniformly random oriented edges of $\Qirr_n$ and $\Qirrhex_{n-9}$ respectively.
	Then the following limits exist and are equal:
	\begin{align}\label{eq:lim-proba-local-lim-irred}
	\lim_{n\to\infty}\Prob{\vec U_n\in\occPat(\Qirr_n)}
	= \lim_{n\to\infty}\Prob{\vec U'_n\in\occPat(\Qirredhex_{n-9})}.
	\end{align}
	Furthermore, if we denote by $c$ the common limit, then we have $c>0$.
\end{lemma}

\begin{proof}
	It is a consequence%
	\footnote{Addario--Berry proves in Corollary~4.4.~of \cite{Addario-Berry14} that $(\Qirredhex_{n})_n$ admits a Benjamini--Schramm local limit, and in the proof of his Theorem~5.2., he shows that if we add uniformly an edge in $\Qirredhex_{n}$ which cuts the hexagon into two quadrangles, then conditioning on the event that the resulting quadrangulation is irreducible does not change the local limit \textit{around a uniformly random oriented edge}.
		Under such a conditioning, the resulting irreducible quadrangulation is uniform in $\ensQirred_{n+1}$, as is easily verified and implicitly used in the proof of \cite[Theorem 5.2]{Addario-Berry14}.}
	of results of Addario-Berry in \cite{Addario-Berry14} that $\Qirr_n$ and $\Qirredhex_{n-9}$, re-rooted at a uniformly random oriented edge respectively, both admit a Benjamini--Schramm \textit{local limit} in distribution as $n\to\infty$, and that the limits are the same in distribution.
	This means that if $\smash{\vec U_{n}}$ and $\smash{\vec U'_{n}}$ are uniformly random oriented edges of $\Qirr_n$ and $\Qirrhex_{n-9}$ respectively, then for every $k\geq1$, there exists a probability measure $p^{(k)}(\diff\m)$ on planar maps such that for every set $\Scal$ of finite planar maps, we have
	\begin{align}\label{eq:local-lim-irred-quad}
	\lim_{n\to\infty}\Prob{\ball^{\Qirr_n}(\vec U_n,k)\in\Scal}
	= \lim_{n\to\infty}\Prob{\ball^{\Qirrhex_{n-9}}(\vec U'_n,k)\in\Scal}
	=p^{(k)}(\Scal),
	\end{align}
	where $\ball^\q(\vec e,k)$ is the ball of radius $k$ centered at the base vertex of the oriented edge $\vec e$, in a quadrangulation $\q$ seen as a planar map rooted at $\vec e$.
	
	Since $\pat$ has diameter $5$, we have for every $k\geq5$ the identities of events $\{\vec U_n\in\occPat(\Qirr_n)\}=\{\ball^{\Qirr_n}(\vec U_n,k)\in\Scal_{\pat}\}$ and  $\{\vec U'_n\in\occPat(\Qirredhex_{n-9})\}=\{\ball^{\Qirrhex_{n-9}}(\vec U'_n,k)\in\Scal_{\pat}\}$, where $\Scal_{\pat}$ is the set of finite planar maps which have an occurrence of $\pat$ at their oriented root edge.
	Hence by \eqref{eq:local-lim-irred-quad}, then
	\begin{align*}
	\lim_{n\to\infty}\Prob{\vec U_n\in\occPat(\Qirr_n)}
	= \lim_{n\to\infty}\Prob{\vec U'_n\in\occPat(\Qirredhex_{n-9})}
	=p^{(5)}(\Scal_{\pat}).
	\end{align*}
	It remains to justify that $c=p^{(5)}(\Scal_{\pat})$ is positive.
	Since the uniform distribution on $\ensQirred_n$ is invariant under the operation of re-rooting at a uniformly random oriented edge, we have that
	\begin{align*}
	c	=\lim_{n\to\infty}\Prob{\vec U_n\in\occPat(\Qirr_n)}
		= \lim_{n\to\infty}\Prob{\vec\rho(\Qirr_n)\in\occPat(\Qirr_n)},
	\end{align*}
	where $\vec\rho(\Qirr_n)$ is the oriented root edge of $\Qirr_n$.
	Consider the operation of filling with $\pat$ the hexagon in an irreducible quadrangulation of the hexagon $\q$ such that the root edges of $\pat$ and $\q$ are matched, and then rooting at the root edge of $\q$ with a reversed orientation.
	Clearly this defines a bijection between $\ensQirredhex_{n-9}$ and the subset of $\ensQirred_n$ consisting of irreducible quadrangulations which have an occurrence of $\pat$ at their oriented root edge.
	Hence, we get that:
	\begin{align*}
	\Prob{\vec U_n\in\occPat(\Qirr_n)}
		=	\frac{|\ensQirredhex_{n-9}|}{|\ensQirred_n|}.
	\end{align*}
	By known enumerative results of Fusy, Poulalhon, and Schaeffer \cite{FusySchaefferPoulalhon08}, we have the expression%
		\footnote{
			The formula for $|\ensQirredhex_{n+3}|$ is found in Corollary 4.9 in \cite{FusySchaefferPoulalhon08}.
			Note that in their formulation, $|\mathcal D'_n|$ counts irreducible quadrangulations of the hexagon having $n$ vertices, that is $n+2$ quadrangular faces, or $n+3$ faces in total.
		}
	\begin{align*}
	|\ensQirredhex_{n+3}|=\frac{6}{(n+2)(n+1)}\binom{2n}{n}\sim \frac{6}{\sqrt\pi}\cdot\frac{4^n}{n^{5/2}}.
	\end{align*}
	On the other hand, by results from Tutte's seminal paper \cite{Tutte63}, we have the asymptotic:
	\begin{align*}
	|\ensQirred_n|\sim \frac{2}{3^5\sqrt\pi}\cdot\frac{4^n}{n^{5/2}}.
	\end{align*}
	We deduce from the two last displays that $\lim_n |\ensQirredhex_{n-9}|/|\ensQirred_n|=3^6/4^{12}$, so that the constant $c$ is indeed positive.
\end{proof}

We can now justify the asymptotic concentration of $N(\Qirr_n)/\Expect{N(\Qirr_n)}$.

\begin{lemma}\label{lem:conv-proba-number-occ}
	We have $N(\Qirr_n)/\Expect{N(\Qirr_n)}\to 1$ in the $L^2$ sense, as $n\to\infty$.
\end{lemma}

\begin{proof}
	Let us put ourselves in the setting of Lemma~\ref{lem:local-lim-irred}, and use the same notation for $\vec U_n$ and $\vec U'_n$.
	The first convergence in \eqref{eq:lim-proba-local-lim-irred} from this lemma gives an estimate on the first moment of $N(\Qfilled_n)$, since it can be re-written as:
	\begin{align}\label{eq:proof-cvg-bias-coupling-irred-0}
	\lim_{n\to\infty}\frac{1}{4n}\cdot \Expect{N(\Qirr_n)}
		=c,
	\qquad\text{that is}\quad
	\Expect{N(\Qirr_n)}\sim 4nc,
	\end{align}
	for some $c>0$.
	Let us now estimate the second moment of $N(\Qirr_n)$.
	Since $\vec U_n$ is a uniformly random oriented edge of $\Qirr_n$, we have:
	\begin{multline*}
	\Expect{N(\Qirr_n)^2}
		=	4n \cdot\Expect{N(\Qirr_n)\cdot\indic{\{\vec U_n\in\occPat(\Qirr_n)\}}}\\
		=	4n \cdot
			\Prob{\vec U_n\in\occPat(\Qirr_n)}
			\cdot
			\condExpect{N(\Qirr_n)}{\vec U_n\in\occPat(\Qirr_n)}.
	\end{multline*}
	Hence, using Proposition~\ref{prop:comparison-coupling-irred} and the last display, we get:
	\begin{align}\label{eq:proof-cvg-bias-coupling-irred-1}
	\Expect{N(\Qirr_n)^2}
		&= 4n \cdot
			\Prob{\vec U_n\in\occPat(\Qirr_n)}
			\cdot
			\Expect{N(\Qfilled_n)}\nonumber \\
		&\sim  4n c
			\cdot
			\Expect{N(\Qfilled_n)},
	\end{align}
	where we used \eqref{eq:lim-proba-local-lim-irred} again, to get the asymptotic equivalence.
	Now in Construction~\ref{constr:coupling-irred}, we have $\Qirredhex_{n-9}=\rempat(\Qfilled_n;\vec\ehex)$, and thus $N(\Qfilled_n)=N(\Qirrhex_{n-9})+3$ by Lemma~\ref{lem:number-occ-after-removal}.
	Since $\vec U'_n$ is a uniformly random oriented edge of $\Qirrhex_{n-9}$, and since $\Qirrhex_{n-9}$ has $2(2n-15)$ oriented edges, we obtain:
	\begin{align*}
	\condProb{\vec U'_n\in\occPat(\Qirrhex_{n-9})}{\Qirrhex_{n-9}}
		=\frac{N(\Qirrhex_{n-9})}{2(2n-15)}
		=\frac{N(\Qfilled_n)-3}{2(2n-15)}
	\end{align*}
	The expectation of the left-hand side converges to $c>0$ by \eqref{eq:lim-proba-local-lim-irred}, so that the last display gives:
	\begin{align*}
	\Expect{N(\Qfilled_n)}\sim 4nc.
	\end{align*}
	If we combine this with \eqref{eq:proof-cvg-bias-coupling-irred-1}, we get:
	\begin{align}\label{eq:proof-cvg-bias-coupling-irred-2}
	\Expect{N(\Qirr_n)^2}
		\sim 16n^2 c^2.
	\end{align}
	Wrapping up, by \eqref{eq:proof-cvg-bias-coupling-irred-0} and \eqref{eq:proof-cvg-bias-coupling-irred-2}, we have that $\Expect{N(\Qirr_n)^2}\sim\Expect{N(\Qirr_n)}^2$.
	Hence,
	\begin{align*}
	\Var\Bigl(\frac{N(\Qirr_n)}{\Expect{N(\Qirr_n)}}\Bigr)
	=\frac{\Expect{N(\Qirr_n)^2}-\Expect{N(\Qirr_n)}^2}{\Expect{N(\Qirr_n)}^2}
	\xrightarrow[n\to\infty]{}0.
	\end{align*}
	Since furthermore the random variable $N(\Qirr_n)/\Expect{N(\Qirr_n)}$ has mean $1$, the preceding proves that $N(\Qirr_n)/\Expect{N(\Qirr_n)}\to 1$ in $L^2$ sense.
\end{proof}

Wrapping up, we can now prove Proposition~\ref{prop:dTV-coupling-irred}.

\begin{proof}[Proof of Proposition~\ref{prop:dTV-coupling-irred}]
	We have $\beta_n(\Qirr_n)=N(\Qirr_n)/\Expect{N(\Qirr_n)}\to 1$ in the $L^2$ sense, and therefore also in the $L^1$ sense by the Cauchy--Schwarz inequality.
	Hence $\dTV(\Law(\Qfilled_n),\Law(\Qirr_n))\to 0$ by Corollary~\ref{cor:bias-coupling-irred}.
\end{proof}

% -----------------------------------------------------------------
%							SECTION 9
% -----------------------------------------------------------------

\section{Conclusion}

We have detailed all the different steps of our proof of Theorem~\ref{thm:scaling-limit-irred} whose structure is laid out in Section~\ref{sec:structure-proof}, thus establishing that the Brownian sphere is the scaling limit of uniformly random \textit{irreducible} quadrangulations with $n$ faces.
We believe that our method should adapt to other models, and in particular to \textit{simple quadrangulations} and \textit{simple triangulations}.
This would give a new proof of a result of Addario-Berry and Albenque \cite{AddarioBerryAlbenque17}.
The main ingredient we need to adapt is the existence of couplings by ``face-openings'' for uniformly random simple quadrangulations/triangulations with $n$ faces, $n\geq 2$.
In the triangulation case, such couplings have been constructed by Caraceni and Stauffer in \cite{CaraceniStauffer23}, and we believe that their methods can be adapted to construct such couplings in the quadrangulation case.

Let us also note that the methodology developed by Addario-Berry and Wen in \cite{AddarioBerryWen17} could also be used to deduce the scaling limit of simple quadrangulations from that of its irreducible blocks, whose scaling limit can be deduced from our Theorem~\ref{thm:scaling-limit-irred}.
This gives a second new proof path for this result.

Lastly, let us mention that if we only want to prove that a uniformly random largest simple component $\Qsfrak_n$ of $Q_n\sim\rmUnif(\Quads_n)$ converges jointly with $Q_n$ to the same Brownian sphere (thus bypassing the arguments of Sections~\ref{sec:bottlenecks-and-Hausdorff-cvg} and~\ref{sec:exch-and-Prokhorov}),
then we can do as follows.
\begin{enumerate}
	\item Our proof of Theorem~\ref{thm:scaling-limit-irred} shows that $(Q_n,\Qirrfrak_n)$ suitably renormalized converge jointly to the same Brownian sphere, where  $\Qirrfrak_n$ is a uniformly random largest irreducible component of $\Qsfrak_n$.
	\item Given $\q$ a quadrangulation, $\qsim$ a simple component of $\q$, and $\qirr$ an irreducible component of $\qsim$, the corresponding measured metric spaces satisfy
	\begin{align*}
	\X(\qirr)\orderGHP\X(\qsim)\orderGHP \X(\q).
	\end{align*}
	\item We can therefore conclude by a straightforward application of our GHP sandwich theorem given as Theorem~\ref{thm:sandwich-thm-GHP-random} that $(Q_n,\Qsfrak_n,\Qirrfrak_n)$ suitably renormalized converge jointly to the same Brownian sphere.
\end{enumerate}
We let the reader fill in the details, and in particular we leave item 2.~above as an exercise to the reader.

\printbibliography

\end{document}

%% file: macros.tex
\newcommand{\R}{{\mathbb R}}

\newcommand{\N}{{\mathbb N}}
\renewcommand{\emptyset}{\varnothing}
\renewcommand{\phi}{\varphi}
\renewcommand{\epsilon}{\varepsilon}
\newcommand{\indic}[1]{{\mathbbm 1}_{#1}}
\newcommand\diff{\mathop{}\!\mathrm{d}}

\let\P\relax\DeclareMathOperator{\P}{\mathbf{P}}
\DeclareMathOperator{\E}{\mathbf{E}}
\newcommand{\Expect}[2][]{%
	\E_{#1}\mathopen{}\mathclose\bgroup\left[\,#2\,\aftergroup\egroup\right]}
\newcommand{\condExpect}[3][]{%
	\E_{#1}\mathopen{}\mathclose\bgroup\left[\,#2\;\middle\vert\;#3\,\aftergroup\egroup\right]}
\newcommand{\Prob}[2][]{\P_{#1}\mathopen{}\mathclose\bgroup\left(\,#2\,\aftergroup\egroup\right)}
\newcommand{\condProb}[3][]{%
	\P_{#1}\mathopen{}\mathclose\bgroup\left(\,#2\;\middle\vert\;#3\,\aftergroup\egroup\right)}

\newcommand{\Acal}{\mathcal{A}}
\newcommand{\Bcal}{\mathcal{B}}

\newcommand{\Dcal}{\mathcal{D}}
\newcommand{\Ecal}{\mathcal{E}}
\newcommand{\Fcal}{\mathcal{F}}
\newcommand{\Gcal}{\mathcal{G}}
\newcommand{\Scal}{\mathcal{S}}

\newcommand{\boldxi}{{\boldsymbol{\xi}}}
\newcommand{\boldsigma}{{\boldsymbol{\sigma}}}
\newcommand{\boldK}{\mathcal{K}}
\newcommand{\X}{\mathcal{X}}
\newcommand\Xf{\mathfrak{X}}

\newcommand{\y}{{\mathbf{y}}}
\newcommand{\m}{\mathfrak{m}}
\renewcommand{\b}{\mathfrak{b}}
\newcommand{\q}{\mathfrak{q}}

\newcommand{\perm}{\mathfrak{S}}

\newcommand\mes{\mathcal{M}}
\newcommand{\diam}{\mathrm{diam}}
\newcommand{\size}[1]{|#1|}
\newcommand{\ball}{{B}}

\newcommand\bbM{\mathbb{M}}

\newcommand{\setGHP}{\bbM_{\mathrm{GHP}}}
\newcommand\bbK{\mathbb{K}}
\newcommand{\setGH}{\bbK_{\mathrm{GH}}}
\newcommand\drm{\mathrm{d}}
\newcommand{\dGHP}{\drm_{\mathrm{GHP}}}
\newcommand{\toGHP}{\xrightarrow{\mathrm{GHP}}}
\newcommand\ToGHP{\xRightarrow{\mathrm{GHP}}}
\newcommand{\dGH}{\drm_{\mathrm{GH}}}
\newcommand{\dH}{\updelta_{\mathrm{H}}}
\newcommand{\dP}{\updelta_{\mathrm{P}}}
\newcommand{\isom}{\sim}
\newcommand\distort{\mathrm{dis}}

\newcommand{\down}{\mathrm{Down}}

\newcommand\stpreceq{\preceq_{\mathrm{st}}}
\newcommand{\storder}{\unlhd_{\mathrm{st}}}
\newcommand{\orderGHP}{\unlhd}

\newcommand\Xtop{\mathbb{X}}
\newcommand\xtop{{x}}

\newcommand{\vertices}{{V}}
\newcommand{\Qsimple}{{Q}^{\mathrm{s}}}
\newcommand{\Qirr}{Q^{\mathrm{irr}}}

\newcommand{\distribGHP}{(\mathrm{d}),\,\mathrm{GHP}}

\newcommand\mesFin{\mes_{f}}

\newcommand\dTV{\drm_{\mathrm{ TV}}}

\newcommand\rev{\mathrm{rev}}

\newcommand\Occ{\mathrm{Occ}}
\newcommand\Qirrhex{Q^{\mathrm{irr},\Ngon 6}}

\newcommand\Law{\mathrm{Law}}
\newcommand\remove{\mathrm{Rem}}

\newcommand\Qfilled{Q^{\Ngonfilled 6}}

\newcommand\ehex{e_{\Ngon 6}}

\newcommand\open{\mathrm{Open}}
\newcommand\incrQirrhex{Q^{\mathrm{irr,\Ngon 6, \uparrow}}}
\newcommand\Loop{\mathrm{Loop}}

\newcommand\lb{\ell}

\newcommand\lbs{\lb^{\mathrm{s}}}
\newcommand\lbirr{\lb^{\mathrm{irr}}}

\newcommand\irr{{\mathrm{irr}}}

\newcommand\SComp{\mathfrak{B}^{\mathrm s}}

\newcommand\IrrComp{\mathfrak{B}\relax{}^{\mathrm{irr}}}

\newcommand\qsim{\q^{\mathrm s}}

\newcommand\qirr{\q^{\mathrm{irr}}}
\newcommand\Decs{\mathfrak{D}}

\newcommand\mubiased{\mu^{\mathrm{deg}}}

\newcommand\Glue{\mathrm{\Phi}}

\newcommand\sm{\mathfrak{s}}

\newcommand{\orBdry}{\vec\partial}

\newcommand\Detach{\Psi}

\newcommand{\CompRootedQ}[2]{\Quads_{#1}^{(#2)}}
\newcommand\DecSet{\Dcal}
\newcommand\ifrak{\mathfrak{j}}

\newcommand\Smax{\sm_{\mathrm{max}}}
\newcommand\Qirrfrak{\mathfrak{Q}^{\irr}}
\newcommand\Qsfrak{\mathfrak{Q}^{\mathrm{s}}}

\newcommand{\symbQ}{{\mathcal Q}}
\newcommand{\ensQsimple}{\symbQ^{\mathrm{s}}}
\newcommand{\ensQirred}{\symbQ^{\mathrm{irr}}}

\usetikzlibrary{shapes.geometric}
\newcommand{\Ngon}[2][]{{\resizebox{.6em}{!}{\begin{tikzpicture}
			\node[regular polygon,regular polygon sides=#2,draw,minimum size=.1em,#1](#2-gon){};
			\foreach \X in {1,...,#2}{\fill (#2-gon.corner \X) circle[radius=1.5pt];}
			\end{tikzpicture}}}}
\newcommand{\Ngonfilled}[2][]{{\resizebox{.6em}{!}{\begin{tikzpicture}
			\node[regular polygon,regular polygon sides=#2,draw,fill=gray,minimum size=.1em,#1](#2-gon){};
			\foreach \X in {1,...,#2}{\fill (#2-gon.corner \X) circle[radius=1.5pt];}
			\end{tikzpicture}}}}

\newcommand{\ensQirredhex}{\symbQ^{{\mathrm{irr}},{\Ngon{6}}}}

\newcommand{\Qirred}{Q^{\mathrm{irr}}}
\newcommand{\Qirredhex}{Q^{{\mathrm{irr}},\Ngon{6}}}

\newcommand{\pat}{ {\mathtt P_0} }
\newcommand{\occPat}{\Occ_{\pat}}
\newcommand\rempat{\remove_{\pat}}

\newcommand{\card}[1]{\left|#1\right|}

\newcommand{\Quads}{\symbQ}

\newcommand\Var{\mathbf{Var}}

\newcommand\boldnu{\boldsymbol{\nu}}
\newcommand\rmUnif{\mathrm{Unif}}

\newcommand{\lnorm}[2]{|#2|_{#1}}

\newcommand\boldX{\boldsymbol{X}}

\newcommand\Cov{\mathbf{Cov}}
\newcommand{\pclass}[1]{\underline{#1}}